\documentclass{article}

\usepackage{graphicx}
\usepackage{float}
\usepackage{amsthm}
\usepackage{amsmath}
\usepackage{amssymb}
\usepackage{mathrsfs}
\usepackage{dsfont}
\usepackage{tikz-cd}
\usepackage{stmaryrd}
\usepackage{enumitem}
\usepackage{placeins}
\usepackage{hyperref}
\usepackage{pdflscape}
\usepackage{adjustbox}
\usepackage{rotating}

\theoremstyle{definition}
\newtheorem{Definition}{Definition}[subsection]

\theoremstyle{plain}
\newtheorem{Theorem}[Definition]{Theorem}

\theoremstyle{plain}

\theoremstyle{plain}
\newtheorem{Proposition}[Definition]{Proposition}

\theoremstyle{plain}
\newtheorem{Lemma}[Definition]{Lemma}

\theoremstyle{plain}
\newtheorem{Corollary}[Definition]{Corollary}

\theoremstyle{plain}

\theoremstyle{plain}

\theoremstyle{plain}
\newtheorem{Convention}[Definition]{Convention}

\theoremstyle{definition}
\newtheorem{Example}[Definition]{Example}

\theoremstyle{definition}
\newtheorem{Notation}[Definition]{Notation}

\theoremstyle{remark}
\newtheorem{Remark}[Definition]{Remark}

\theoremstyle{plain}
\newcommand{\thistheoremname}{}
\newtheorem*{genericthm*}{\thistheoremname}
\newenvironment{namedthm*}[1]
  {\renewcommand{\thistheoremname}{#1}%
   \begin{genericthm*}}
  {\end{genericthm*}}

\author{Thibault D. Décoppet}
\title{Finite Semisimple Module 2-Categories}

\begin{document}

\bibliographystyle{alpha}

    \maketitle
    \hspace{1cm}
    \begin{abstract}
        Let $\mathfrak{C}$ be a multifusion 2-category. We show that every finite semisimple $\mathfrak{C}$-module 2-category is canonically enriched over $\mathfrak{C}$. Using this enrichment, we prove that every finite semisimple $\mathfrak{C}$-module 2-category is equivalent to the 2-category of modules over a rigid algebra in $\mathfrak{C}$.
    \end{abstract}
    
\tableofcontents
    
\section*{Introduction}
\addcontentsline{toc}{section}{Introduction}

This work fits into our project of investigating the properties of (multi)fusion 2-categories (originally defined in \cite{DR}, see also \cite{D2}). More precisely, our main theorem is a categorification of the main result of \cite{O}, sometimes referred to as Ostrik's theorem. This last theorem asserts that, over a multifusion category, every finite semisimple module category is the category of modules over an algebra. This result appears throughout the theory of fusion categories. For instance, it features prominently in the construction of the relative Deligne tensor product (see \cite{ENO}, and \cite{DSPS14}). Further, it also features prominently in the proof of many deep results on fusion categories such as Ocneanu rigidity (see \cite{EGNO}), or the fact that fusion categories are separable (see \cite{DSPS13}).

Categorifying the definition of a fusion category, Douglas and Reutter introduced in \cite{DR} the notion of a (multi)fusion 2-categories (over an algebraically closed field of characteristic zero), that is a finite semisimple 2-categoryequipped with a rigid monoidal structure. Since their introduction, these objects have been the subject of an increasing level of attention. Their first explicit application was the construction of a state-sum invariant of 4-manifolds, which takes as input a highly structured fusion 2-category (see \cite{DR}). More generally, fusion 2-categories play an important role in the theory of braided fusion categories, sometimes explicitly as in \cite{JFR}, and sometimes implicitly as in \cite{BJS}, \cite{BJSS}, \cite{DN}, \cite{ENO}, or \cite{JMPP}. This role is explained by the fact that the 2-category of finite semisimple module categories over a braided fusion category is a connected fusion 2-category.

In a somewhat different direction, fusion 2-categories have also found applications in condensed matter physics. More precisely, \cite{JF} has given a definition of separable multifusion $n$-category, which is used to propose a definition of topological orders. These objects have also been a recent subject of interest (see \cite{JFY}, \cite{KZ1}, and \cite{KZ2}). What makes separable multifusion $n$-categories useful in the context of high energy physics is that they are (essentially by definition) fully dualizable objects of appropriate higher categories, and so define topological field theories via the coborodism hypothesis of \cite{BD} and \cite{L}. In the case $n=1$, it has been shown in \cite{DSPS13} that every multifusion category (over an algebraically closed field of characteristic zero) is separable. When $n=2$, it has been proven in \cite{BJS} that every braided fusion category is fully dualizable, which can be used to show that connected fusion 2-categories are separable. It remains an open problem to prove that every multifusion 2-category is separable.

The motivation behind our main theorem is therefore twofold. Firstly, in order to develop the theory of fusion 2-categories, one can attempt to categorify classical results of the theory of fusion 1-categories. This has already proven to be fruitful when constructing the 2-Deligne tensor product of two finite semisimple 2-categories (see \cite{D3}). Furthermore, as explained above, such inquiries into the properties of fusion 2-categories might prove useful in the study of (braided) fusion categories. Secondly, we believe that our main theorem will help us towards proving that every multifusion 2-category is separable. More precisely, the definition of separability, while appropriate to study dualizability questions, is difficult to check in practice. To remedy this issue, one wishes to exhibit an alternative characterization of separability. As one of the essential tools used to achieve this in the setting of fusion 1-categories is Ostrik's theorem (see \cite{DSPS13}), it is sensible to believe that our main theorem will play a similar role when studying fusion 2-categories. In fact, it is possible to give such an alternative characterization of separability for multifusion 2-categories, and we shall return to this point in \cite{D6}. In the meantime, let us explain further the relation between our main result and the classical result of Ostrik.

\subsection*{Finite Semisimple Module 1-Categories}
\addcontentsline{toc}{subsection}{Finite Semisimple Module 1-Categories}

Let us fix a multifusion category $\mathcal{C}$ over an algebraically closed field of characteristic zero, and $\mathcal{M}$ a finite semisimple left $\mathcal{C}$-module category, that is a finite semisimple category $\mathcal{M}$ with a coherent left action $\otimes:\mathcal{C}\times\mathcal{M}\rightarrow \mathcal{M}$. Ostrik's theorem (see \cite{O}, or \cite{EGNO}) asserts that $\mathcal{M}$ is equivalent to the category of right modules over an algebra $A$ in $\mathcal{C}$. In fact, the proof allows us to say a little more. Namely, one begins by observing that the finite semisimplicity assumption can be used to upgrade $\mathcal{M}$ to a category enriched over $\mathcal{C}$. Then, given any object $M$ of $\mathcal{M}$, the endomorphism object $\underline{End}(M)$ in $\mathcal{C}$ admits a canonical algebra structure. Further, for any object $N$, the $Hom$-object $\underline{Hom}(M,N)$ in $\mathcal{C}$ is canonically a right $\underline{End}(M)$-module via composition. Thus, we have a functor $\underline{Hom}(M,-):\mathcal{M}\rightarrow \mathrm{Mod}_{\mathcal{C}}(\underline{End}(M))$ from $\mathcal{M}$ to the category of right $\underline{End}(M)$-modules in $\mathcal{C}$. Using universal properties, it can be shown that this functor is a left $\mathcal{C}$-module functor. Let us now assume that $M$ is a $\mathcal{C}$-generator of $\mathcal{M}$, meaning that every object of $\mathcal{M}$ is a direct summand of $C\otimes M$ for some $C$ in $\mathcal{C}$. Then, the restriction of the functor $\underline{Hom}(M,-)$ to the full subcategory on the objects of the form $C\otimes M$ for some $C$ in $\mathcal{C}$ is fully faithful. It then follows from the fact that $\mathcal{M}$ has coequalizers and that they are preserved by $\underline{Hom}(M,-)$ that $\mathcal{M}$ is equivalent to the category of right modules over $A=\underline{End}(M)$ in $\mathcal{C}$. Conceptually, coequalizers appear because Ostrik's theorem is a special case of the Barr-Beck theorem (as remarked in \cite{EGNO}). Finally, we also wish to point out that it was shown in \cite{DSPS13} that Ostrik's theorem holds over any base field.

\subsection*{Finite Semisimple Module 2-Categories}
\addcontentsline{toc}{subsection}{Finite Semisimple Module 2-Categories}

Let $\mathfrak{C}$ be a multifusion 2-category over an algebraically closed field of characteristic zero, and $\mathfrak{M}$ a finite semisimple left $\mathfrak{C}$-module 2-category, i.e.\ a finite semisimple 2-category $\mathfrak{M}$ equipped with a coherent left action $\Diamond:\mathfrak{C}\times\mathfrak{M}\rightarrow\mathfrak{M}$. The proof of our main theorem begins by proving an enrichment result that categorifies the enrichment part of the proof of the classical Ostrik theorem. More precisely, using the definition of a 2-category enriched over a monoidal 2-category given in \cite{GS}, we prove the following result.

\begin{namedthm*}{Theorem \ref{thm:enrichmentmodule}}
Let $\mathfrak{C}$ be a multifusion 2-category, and let $\mathfrak{M}$ be a finite semisimple left $\mathfrak{C}$-module 2-category. Then, $\mathfrak{M}$ admits a canonical $\mathfrak{C}$-enriched structure.
\end{namedthm*}

\noindent As an immediate corollary, this implies that, for any object $M$ in $\mathfrak{M}$, $\underline{End}(M)$ is an algebra in $\mathfrak{C}$ under the enriched composition operation. Then, for every $N$ in $\mathfrak{M}$, $\underline{Hom}(M,N)$ is a right $\underline{End}(M)$-module with action induced by the enriched composition operation. Further, we obtain a linear 2-functor $\underline{Hom}(M,-):\mathfrak{M}\rightarrow \mathbf{Mod}_{\mathfrak{C}}(\underline{End}(M))$ to the 2-category of right $\underline{End}(M)$-modules in $\mathfrak{C}$, which can be upgraded to a left $\mathfrak{C}$-module 2-functor.

In order to establish our categorification of Ostrik's theorem, we will need to understand the structure of the 2-category $\mathbf{Mod}_{\mathfrak{C}}(\underline{End}(M))$. We begin by examining the properties of the algebra $\underline{End}(M)$ in $\mathfrak{C}$. More precisely, recall from \cite{JFR} that an algebra $A$ in a monoidal 2-category is called rigid if its multiplication 1-morphism $m$ has a right adjoint as an $(A,A)$-bimodule 1-morphism. We establish the following theorem.

\begin{namedthm*}{Theorem \ref{thm:rigid}}
Let $\mathfrak{C}$ be a multifusion 2-category, and $\mathfrak{M}$ a finite semisimple left $\mathfrak{C}$-module 2-category. Given any object $M$ of $\mathfrak{M}$, the algebra $\underline{End}(M)$ in $\mathfrak{C}$ is rigid.
\end{namedthm*}

\noindent Our proof proceeds by showing that, for every left $\mathfrak{C}$-module 2-functor $F:\mathfrak{M}\rightarrow\mathfrak{M}$, the object $\underline{Hom}(M,F(M))$ of $\mathfrak{C}$ admits a canonical $(\underline{End}(M),\underline{End}(M))$-bimodule structure. Moreover, this assignment extends functorially to left $\mathfrak{C}$-module 2-natural transformations and $\mathfrak{C}$-module modification. The theorem follows by proving that certain left $\mathfrak{C}$-module 2-natural transformations have right adjoints as a left $\mathfrak{C}$-module 2-natural transformation. We then go on to study the 2-category $\mathbf{Mod}_{\mathfrak{C}}(A)$ of right modules over a rigid algebra $A$ in $\mathfrak{C}$. In particular, we obtain a characterization of rigid algebras in a multifusion 2-category over an algebraically closed field of characteristic zero.

\begin{namedthm*}{Theorem \ref{thm:characterizationrigid}}
An algebra $A$ in a multifusion 2-category $\mathfrak{C}$ is rigid if and only if $\mathbf{Mod}_{\mathfrak{C}}(A)$ is a finite semisimple 2-category.
\end{namedthm*}

\noindent When $\mathfrak{C} = \mathbf{2Vect}$, rigid algebras correspond exactly to multifusion 1-categories. In this case, the theorem above recovers theorem 1.4.8 of \cite{DR}, which asserts that the 2-category $\mathbf{Mod}(\mathcal{C})$ of finite semisimple module categories over the multifusion category $\mathcal{C}$ is finite semisimple.

Let us now assume that $M$ is a $\mathfrak{C}$-generator, meaning that every object of $\mathfrak{M}$ receives a non-zero 1-morphism from $C\Diamond M$ for some $C$ in $\mathfrak{C}$. We wish to show that $\underline{Hom}(M,-)$ is an equivalence of 2-categories. At this point, we do not quite follow the classical proof. Specifically, we do not know how to directly prove the existence of the weighted colimits required by the Barr-Beck theorem for 2-categories (see \cite{CMV}) in finite semisimple 2-categories. Instead, we will rely on the general properties of finite semisimple 2-categories established in \cite{DR} and \cite{D1}. More precisely, it follows by inspection that the restriction of the 2-functor $\underline{Hom}(M,-)$ to the objects of the form $C\Diamond M$ for some $C$ in $\mathfrak{C}$ is essentially surjective on 1-morphisms and fully faithful on 2-morphisms. But, for every right $\underline{End}(M)$-module $N$, there is a canonical right $\underline{End}(M)$-module 1-morphism $N\Box\underline{End}(M)\rightarrow N$. As $\mathfrak{M}$ and $\mathbf{Mod}_{\mathfrak{C}}(\underline{End}(M))$ are finite semisimple, this implies that $\underline{Hom}(M,-)$ is essentially surjective. Thence, we find that $\underline{Hom}(M,-)$ is an equivalence, establishing our main theorem.

\begin{namedthm*}{Theorem \ref{thm:2Ostrik}}
For any multifusion 2-category $\mathfrak{C}$, and finite semisimple left $\mathfrak{C}$-module 2-category $\mathfrak{M}$, there exists a rigid algebra $A$ in $\mathfrak{C}$ such that $\mathfrak{M}$ is equivalent to $\mathbf{Mod}_{\mathfrak{C}}(A)$ as a left $\mathfrak{C}$-module 2-category.
\end{namedthm*}

\noindent With $\mathfrak{C} = \mathbf{2Vect}$, our main theorem reduces to the statement obtained in  theorem 1.4.9 of \cite{DR} that every finite semisimple 2-category is equivalent to $\mathbf{Mod}(\mathcal{C})$ for some multifusion 1-category $\mathcal{C}$.

Given that Ostrik's theorem holds over any base field, it is reasonable to ask whether our results can be similarly generalized. It so happens that, upon making reasonable additional hypotheses, this can be done. More precisely, if we work over a perfect field, that $\mathfrak{C}$ is a locally separable compact semisimple tensor 2-category, and $\mathfrak{M}$ is a locally separable compact semisimple left $\mathfrak{C}$-module 2-category as defined in \cite{D5}, then theorems \ref{thm:enrichmentmodule}, \ref{thm:rigid}, and \ref{thm:2Ostrik} do hold. However, we emphasize that theorem \ref{thm:characterizationrigid} does not hold over fields of positive characteristic. In order not to unnecessarily obscure our arguments, we discuss how to make this generalization separately using remarks at the end of the relevant subsection. 

\subsection*{Outline}
\addcontentsline{toc}{subsection}{Outline}

In section \ref{sec:prelim}, we begin by explaining our various conventions and notations for (monoidal) 2-categories, and (monoidal) 2-functors. We also explain the graphical language of string diagrams, which will be used throughout this article.

Then, in section \ref{sec:module2cat}, we give two equivalent definitions of a left module 2-category over a monoidal 2-category $\mathfrak{C}$. We go on to define the relevant notions of left $\mathfrak{C}$-module 2-functor, left $\mathfrak{C}$-module 2-natural transformation, and left $\mathfrak{C}$-module modification. We end by describing the basic properties satisfied by these objects. In particular, we derive a coherence result for left module 2-categories.

In section \ref{sec:algebramodule}, fixing a monoidal 2-category $\mathfrak{C}$, we recall the definition of an algebra in $\mathfrak{C}$, and the definition of a right module in $\mathfrak{C}$ over an algebra. We review the properties of the 2-category of right modules in $\mathfrak{C}$ over an algebra, and we prove that this 2-category admits the structure of a left $\mathfrak{C}$-module 2-category. We end this section by proving various technical lemmas, which will be used later on.

Next, in section \ref{sec:enrichment}, we specialize our investigations to the case when $\mathfrak{C}$ is a multifusion 2-category. We explain how to construct a $\mathfrak{C}$-enriched 2-category out of a finite semisimple left $\mathfrak{C}$-module 2-category, and prove that the enriched $Hom$-2-functors are left $\mathfrak{C}$-module 2-functors.

In section \ref{sec:2Ostrik}, we show that left $\mathfrak{C}$-module 2-functors between two finite semisimple left $\mathfrak{C}$-module 2-categories can be upgraded to $\mathfrak{C}$-enriched 2-functors. We go on to explain how to assign a bimodule in $\mathfrak{C}$ to any such 2-functor. Using this construction, we show that the algebra in $\mathfrak{C}$ of enriched endomorphisms associated to an object of a finite semisimple left $\mathfrak{C}$-module 2-category is always rigid. Moreover, we show that, over an algebraically closed field of characteristic zero, an algebra in a multifusion 2-category is rigid if and only if the corresponding 2-category of right modules is finite semisimple. We then prove that finite semisimple left $\mathfrak{C}$-module 2-categories are 2-categories of right modules over a rigid algebra in $\mathfrak{C}$.

Finally, in the appendix, we give the string diagram manipulations used in proofs of the results of section \ref{sec:enrichment}.

\subsection*{Acknowledgments}

I am grateful to Christopher Douglas and Andr\'e Henriques for their supervison and valuable feedback. I would also like to thank Dave Penneys for our helpful conversation, and David Reutter for his useful comments on an earlier version of this manuscript.

\section{Graphical Calculus in Monoidal 2-Categories}\label{sec:prelim}

\addtocounter{subsection}{1}

Before we can actually begin our investigations, we need to introduce our conventions and notations regarding (monoidal) 2-categories, which are inspired by those of \cite{GS}. Our first convention is that we shall use the term 2-category to refer to what is sometimes called a bicategory. When we wish to consider their strict variant, we shall use the names strict 2-category. Further, thanks to the coherence theorem for 2-category, we will omit all of the coherence 2-isomorphisms for 2-categories that might appear.

Concerning monoidal 2-categories, we shall for the most part follow the notation of \cite{SP}: The main distinction being that we reserve the use of the symbol $1$ for identity 1-morphisms. In details, a monoidal 2-category is a 2-category $\mathfrak{C}$ equipped with a 2-functor $\mathfrak{C}\times\mathfrak{C}\rightarrow \mathfrak{C}$, a distinguished object $I$, and adjoint 2-natural equivalences $\alpha$, $l$, $r$, given on objects $A,B,C$ in $\mathfrak{C}$ by: $$\alpha_{A,B,C}:(A\Box B)\Box C\rightarrow A\Box (B\Box C),\ \ \ l_B:I\Box B\rightarrow B,\ \ \ r_A:A\rightarrow A\Box I.$$ We use $\alpha^{\bullet}$, $l^{\bullet}$, $r^{\bullet}$ to denote their chosen adjoints pseudo-inverse. Furthermore, there are invertible modifications $\pi$, $\lambda$, $\rho$, $\nu$, which are given on $A,B,C,D$ in $\mathfrak{C}$ by 

\begin{center}
$\begin{tikzcd}[sep=tiny]
                           & (A(BC))D \arrow[ld, "\alpha"'] &                                 & ((AB)C)D \arrow[ll, "\alpha 1"'] \arrow[rd, "\alpha"] &                           \\
A((BC)D) \arrow[rrdd, "1\alpha"']  & {}\arrow[rr, "\pi", Rightarrow, shorten >=11.5ex, shorten <=11.5ex ]                                                      &  &  {}                                      & (AB)(CD), \arrow[lldd, " \alpha"] \\
                           &                                                        &                               &                                       &                           \\
                           &                                                        & A(B(CD))                              &                                       &                          
\end{tikzcd}$

\begin{tabular}{@{}c c c@{}}
$$\begin{tikzcd}[sep=tiny]
BC \arrow[ddd, equal]\arrow[Rightarrow, rrrddd, "\lambda", shorten > = 2ex, shorten < = 2ex]  &                                        &    & (IB)C \arrow[ddd, "\alpha"] \arrow[lll, "l1"'] \\
 &  &    & \\
  &  &  &  \\
BC                                            &                                        &    &  I(BC), \arrow[lll, "l"]
\end{tikzcd}$$

&

$$\begin{tikzcd}[sep=tiny]
A(BI) \arrow[ddd, equal]\arrow[Rightarrow, rrrddd, "\rho", shorten > = 2ex, shorten < = 2ex]  &                                        &    & AB \arrow[ddd, "r"] \arrow[lll, "1 r"'] \\
 &  &    & \\
  &  &  &  \\
A(BI)                                            &                                        &    &  (AB)I, \arrow[lll, "\alpha"]
\end{tikzcd}$$

&

$$\begin{tikzcd}[sep=tiny]
(AI)C  \arrow[rrr, "\alpha", ""{name = a}] &  & & A(IC) \arrow[ddd, "1l"]\\
&  &  &\\
& & &\\
AC \arrow[uuu, "r1"] & & & AC. \arrow[equal, lll, ""{name = b}]

\arrow[Rightarrow,"\mu", from= a, to = b, shorten > = 2ex, shorten < = 2ex]
\end{tikzcd}$$
\end{tabular}
\end{center}

Observe that we have omitted the symbol $\Box$ from the diagrams above to declutter them, and will continue doing so. Finally, these invertible modifications have to satisfy certain axioms given in \cite{SP}.

In order to define composites of 2-morphisms, we will use string diagrams, i.e.\ diagrams consisting of ``spaghettis and meatballs''. More precisely, regions represent objects, strings 1-morphisms, and coupons 2-morphisms. Our string diagrams are to be read from top to bottom (composition of 1-morphisms) and from left to right (composition of 2-morphisms). To prevent the diagrams from being too cluttered, we do not label the regions; If needed these labels can always be recovered from the labels on the strings. We also leave out the strings labelled by identity 1-morphisms, and the coupons labelled by identity 2-morphisms. Further, we omit the symbol $\Box$ in such diagrams. For instance, the 2-isomorphisms $\pi$, $\lambda$, $\rho$, and $\mu$ given above as pasting diagrams are depicted using string diagrams as: $$\begin{tabular}{c c c c}
\ \ \ \ \ \ \includegraphics[width=20mm]{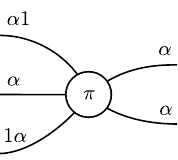},\ \   & \ \ \includegraphics[width=20mm]{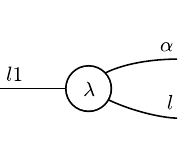},\ \  & \ \ \includegraphics[width=20mm]{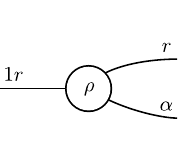},\ \  &
\ \ \includegraphics[width=20mm]{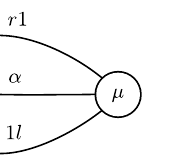}.
\end{tabular}$$ In some of the bigger string diagrams depicted in this paper, particularly those featured in the appendix, we shall omit the labels on certain strings: These can always be recovered from the labels given in other string diagrams. Finally, thanks to the coherence theorem for monoidal 2-categories (see \cite{Gur}), we can and will often assume that the monoidal 2-category under consideration is strict cubical. This simplifies our string diagrams greatly. Namely, we can and do omit both the parentheses and the instances of $I$ in expressions involving the monoidal product. Further, as we do not draw identity 1-morphisms and identity 2-morphisms, the only structure 2-isomorphisms we have to depict are the interchangers $\phi^{\Box}$. This explains why we introduce a special notation for them: Given $f$, $g$ two 1-morphisms in $\mathfrak{C}$, the composite 2-isomorphism $$(\phi^{\Box}_{(1,g),(f,1)})^{-1} \cdot \phi^{\Box}_{(f,1),(1,g)}:(f\Box 1)\circ (1\Box g)\cong (1\Box g)\circ (f\Box 1)$$ is depicted by the string diagram below on the left and its inverse by that on the right: $$\begin{tabular}{c c c c}
\includegraphics[width=20mm]{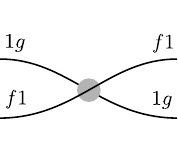},\ \ \ \   & \ \ \ \  \includegraphics[width=20mm]{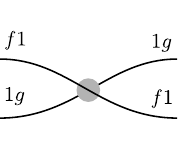}.
\end{tabular}$$

When considering a 2-category $\mathfrak{A}$ that has right adjoints, we will employ the usual calculus of cups and caps. This means that, given a 1-morphism $f$ in $\mathfrak{A}$, we can pick a right adjoint $f^*$, and use the string diagram below on the left to depict the unit of this adjunction, and the one on the right to depict the counit: $$\begin{tabular}{c c}
\includegraphics[width=20mm]{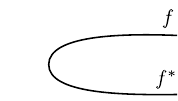},\ \ \ \ \ \ \ \ 
\ \ \ \ \ \ \ \ \includegraphics[width=20mm]{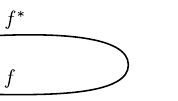}.
\end{tabular}$$ If $\mathfrak{A}$ has left adjoints, we can similarly pick a left adjoint $^*f$ to $f$, and use the obvious analogue of the above notation to depict the corresponding unit and counit. We want to emphasize that this convention regarding cups and caps is different from that used in \cite{GS}.

In accordance with our previous conventions, we use the term 2-functor to designate what is sometimes called a homomorphism of bicategories. Given two 2-categories $\mathfrak{A}$, and $\mathfrak{B}$, a 2-functor $F:\mathfrak{A}\rightarrow \mathfrak{B}$ consists of a 2-isomorphism $\phi^F_A:Id_{F(A)}\cong F(Id_A)$ in $\mathfrak{B}$ for every object $A$ of $\mathfrak{A}$, and a 2-isomorphism $\phi^F_{g,f}:F(g)\circ F(f)\cong F(g\circ f)$ in $\mathfrak{B}$ for every pair of composable 1-morphism in $\mathfrak{A}$. These isomorphisms have to satisfy coherence axioms. In the rare instance we will use monoidal 2-functors, we will follow the notations of \cite{SP}, to which we refer for a precise exposition.

Given two 2-functors $F,G:\mathfrak{A}\rightarrow \mathfrak{B}$, a 2-natural transformation $\beta:F\Rightarrow G$ consists of 1-morphisms $\beta_A:F(A)\rightarrow G(A)$ for every object $A$ of $\mathfrak{A}$ together with 2-isomorphisms $$\begin{tikzcd}[sep=tiny]
F(A) \arrow[ddd, "F(f)"']\arrow[rrr, "\beta_A"]  &                                        &    & G(A) \arrow[ddd, "G(f)"]  \\
 &  &    & \\
  &  &  &  \\
F(B)\arrow[rrr, "\beta_B"']\arrow[Rightarrow, rrruuu, "\beta_f", shorten > = 2ex, shorten < = 2ex]                                            &                                        &    &  G(B), 
\end{tikzcd}$$ for every $f:A\rightarrow B$ in $\mathfrak{A}$, which are natural in $f$, and satisfy well-known coherence axioms. In our graphical language, we depict such 2-isomorphisms by the string diagram below on the left, and its inverse by that on the right: $$\begin{tabular}{c c c c}
\includegraphics[width=20mm]{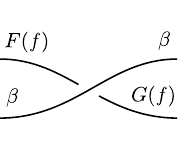},\ \ \ \   & \ \ \ \  \includegraphics[width=20mm]{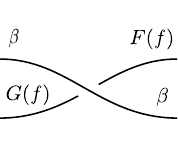}.
\end{tabular}$$ Our convention is that the string labelled $\beta$ remains on top.

\newlength{\prelim}
\settoheight{\prelim}{\includegraphics[width=20mm]{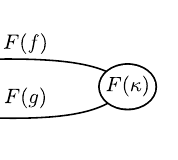}}

Given 1-morphisms $f:A\rightarrow B$, $g:B\rightarrow A$ in $\mathfrak{A}$, and a 2-morphism $\kappa:g\circ f\Rightarrow Id_A$ in $\mathfrak{A}$, we set: $$\includegraphics[width=20mm]{Module2CategoriesPictures/prelim/Fkappa.pdf} \raisebox{0.43\prelim}{$:= (\phi^F_{A})^{-1} \cdot F(\kappa)\cdot \phi^F_{g,f}.$}$$ We generalize this convention in the obvious way to more general 2-morphisms. Further, we keep the notations we have introduced above for the image under a 2-functor of interchangers and 2-isomorphisms witnessing naturality. For instance, if $\beta:F\Rightarrow G$ is a 2-natural transformation and $H:\mathfrak{B}\rightarrow \mathfrak{C}$ is a 2-functor, we have $$\includegraphics[width=37.5mm]{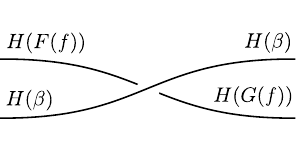} \raisebox{0.50\prelim}{$\ \ \ = (\phi^H_{f, \beta_A})^{-1} \cdot H(\beta^f)\cdot \phi^H_{\beta_B, f}.$}$$

Now, let $F:\mathfrak{A}\rightleftarrows \mathfrak{B}:G$ be two 2-functors. Following definition 2.1 of \cite{Gur}, we say $F$ and $G$ form a coherent 2-adjunction provided that there exists 2-natural transformations $u$ and $c$, given on $A$ in $\mathfrak{A}$ and $B$ in $\mathfrak{B}$ by $$u_A:A\rightarrow G(F(A))\ \ \text{and} \ \ c_B:F(G(B))\rightarrow B,$$ and invertible modifications $\Phi$ and $\Psi$, given on $A$ in $\mathfrak{A}$ and $B$ in $\mathfrak{B}$ by $$\Phi_A:c_{F(A)}\circ F(u_A)\cong Id_{F(A)}\ \ \text{and} \ \ \Psi_B:G(c_B)\circ u_{G(B)}\cong Id_{G(B)},$$ such that for every $B$ in $\mathfrak{B}$, the equality

\settoheight{\prelim}{\includegraphics[width=37.5mm]{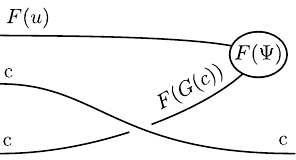}}

$$
\begin{tabular}{@{}ccc@{}}
\includegraphics[width=37.5mm]{Module2CategoriesPictures/prelim/coherence1.pdf} & \raisebox{0.45\prelim}{$=$} &

\includegraphics[width=30mm]{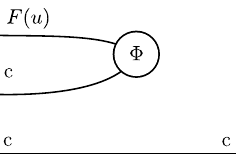}
\end{tabular}
$$

holds in $Hom_{\mathfrak{B}}(F(G(B)),B)$, and for every $A$ in $\mathfrak{A}$, the equality

$$
\begin{tabular}{@{}ccc@{}}
\includegraphics[width=37.5mm]{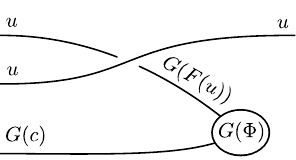} & \raisebox{0.45\prelim}{$=$} &

\includegraphics[width=30mm]{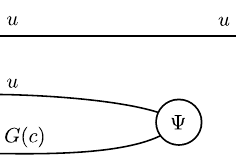}
\end{tabular}
$$

holds in $Hom_{\mathfrak{A}}(A, G(F(A)))$. These two equations are called the swallowtail equations.

Finally, we wish to remark that under our conventions, the swallowtail equations remain valid after application of a 2-functor. We only state this precisely for the first swallow tail equation, a similar result holds for the second.

\settoheight{\prelim}{\includegraphics[width=45mm]{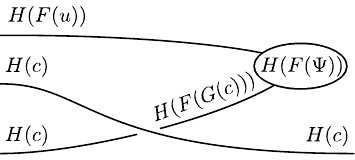}}

\begin{Lemma}\label{lem:biadj2functor}
Let $F:\mathfrak{A}\rightleftarrows \mathfrak{B}:G$ a coherent 2-adjunction, and $H:\mathfrak{B}\rightarrow \mathfrak{C}$ a 2-functor. For any $B$ in $\mathfrak{B}$, the equality $$\begin{tabular}{@{}ccc@{}}
\includegraphics[width=45mm]{Module2CategoriesPictures/prelim/coherence1prime.pdf} & \raisebox{0.45\prelim}{$=$} &
\includegraphics[width=30mm]{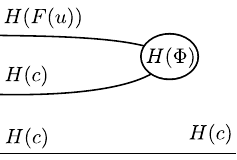}
\end{tabular}$$ holds in $Hom_{\mathfrak{C}}(H(F(G(B))),H(B))$.
\end{Lemma}
\begin{proof}
The proof follows readily from the definitions and the coherence axioms for the 2-functor $H$.
\end{proof}

\section{Module 2-Categories over Monoidal 2-Categories}\label{sec:module2cat}

\subsection{Definitions}

Let us fix $\mathfrak{C}$, and $\mathfrak{D}$ two monoidal 2-categories. We use $\mathfrak{D}^{\Box op}$ to denote the monoidal 2-category whose underlying 2-category is $\mathfrak{D}$, and with the opposite monoidal structure.

\begin{Definition}\label{def:module2cat}
A left $\mathfrak{C}$-module 2-category is a 2-category $\mathfrak{M}$ equipped with a monoidal 2-functor $H^{\mathfrak{M}}:\mathfrak{C}\rightarrow End(\mathfrak{M})$. A right $\mathfrak{D}$-module 2-category is a 2-category $\mathfrak{M}$ equipped with a monoidal 2-functor $H^{\mathfrak{M}}:\mathfrak{D}^{\Box op}\rightarrow End(\mathfrak{M})$. A $(\mathfrak{C},\mathfrak{D})$-bimodule 2-category is a 2-category $\mathfrak{M}$ equipped with a monoidal 2-functor $H^{\mathfrak{M}}:\mathfrak{C}\times\mathfrak{D}^{\Box op}\rightarrow End(\mathfrak{M})$.
\end{Definition}

Alternatively, a left $\mathfrak{C}$-module 2-category can be defined as follows.

\begin{Notation}
In definition \ref{def:module2catunpacked} below, and throughout the paper, we use the same conventions for $\Diamond$ as those which we have explained in section \ref{sec:prelim} for $\Box$.
\end{Notation}

\begin{Definition} \label{def:module2catunpacked}
Let $\mathfrak{C}$ be an arbitrary monoidal 2-category, and $\mathfrak{M}$ an arbitrary 2-category. A left $\mathfrak{C}$-module 2-category structure on $\mathfrak{M}$ consists of the following data:
\begin{enumerate}
\item A 2-functor $\Diamond:\mathfrak{C}\times \mathfrak{M}\rightarrow \mathfrak{M}$;
\item Two adjoint 2-natural equivalences $\alpha^{\mathfrak{M}}$, and $l^{\mathfrak{M}}$ given on $A,B$ in $\mathfrak{C}$ and $M$ in $\mathfrak{M}$ by 
$$\begin{aligned}
\alpha^{\mathfrak{M}}_{A,B,M}&:(A\Box B)\Diamond M \rightarrow A\Diamond (B\Diamond M),\\
l^{\mathfrak{M}}_{M}&: I\Diamond M\rightarrow M;
\end{aligned}$$
\item Three invertible modifications $\mu^{\mathfrak{M}}$, $\lambda^{\mathfrak{M}}$, and $\pi^{\mathfrak{M}}$ given on $A,B,C$ in $\mathfrak{C}$, and $M$ in $\mathfrak{M}$ by
\end{enumerate}

\begin{center}
$\begin{tikzcd}[sep=tiny]
                           & (A(BC))M \arrow[ld, "\alpha^{\mathfrak{M}}"'] &                                 & ((AB)C)M \arrow[ll, "\alpha 1"'] \arrow[rd, "\alpha^{\mathfrak{M}}"] &                           \\
A((BC)M) \arrow[rrdd, "1\alpha^{\mathfrak{M}}"']  & {}\arrow[rr, "\pi^{\mathfrak{M}}", Rightarrow, shorten >=11.5ex, shorten <=11.5ex ]                                                      &  &  {}                                      & (AB)(CM), \arrow[lldd, " \alpha^{\mathfrak{M}}"] \\
                           &                                                        &                               &                                       &                           \\
                           &                                                        & A(B(CM))                              &                                       &                          
\end{tikzcd}$

\begin{tabular}{@{}c c@{}}
$$\begin{tikzcd}[sep=tiny]
B M \arrow[ddd, equal]\arrow[Rightarrow, rrrddd, "\lambda^{\mathfrak{M}}", shorten > = 2.7ex, shorten < = 2.7ex]  &                                        &    & (I B) M \arrow[ddd, "\alpha^{\mathfrak{M}}"] \arrow[lll, "l 1"'] \\
 &  &    & \\
  &  &  &  \\
B M                                            &                                        &    &  I(BM), \arrow[lll, "l^{\mathfrak{M}}"]
\end{tikzcd}$$

&

$$\begin{tikzcd}[sep=tiny]
(A I) M  \arrow[rrr, "\alpha^{\mathfrak{M}}", ""{name = a}] &  & & A (I M) \arrow[ddd, "1 l^{\mathfrak{M}}"]\\
&  &  &\\
& & &\\
A M \arrow[uuu, "r 1"] & & & A M; \arrow[equal, lll, ""{name = b}]

\arrow[Rightarrow,"\mu^{\mathfrak{M}}", from= a, to = b, shorten > = 2ex, shorten < = 2ex]
\end{tikzcd}$$
\end{tabular}
\end{center}

Subject to the following relations:
\begin{enumerate}
    \item [a.] For every $A,B,C,D$ in $\mathfrak{C}$ and $M$ in $\mathfrak{M}$, we have
\end{enumerate}

\newlength{\module}
\settoheight{\module}{\includegraphics[width=52.5mm]{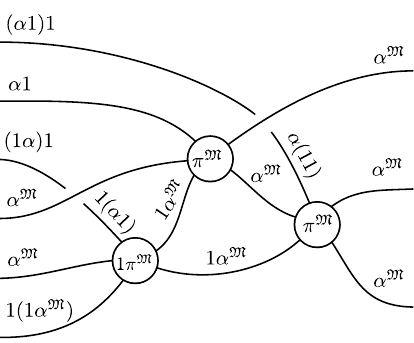}}

\begin{center}
\begin{tabular}{@{}ccc@{}}

\includegraphics[width=52.5mm]{Module2CategoriesPictures/module2cat/2Ostrikmodulepentagonator1.pdf} & \raisebox{0.5\module}{$=$} &

\includegraphics[width=46.5mm]{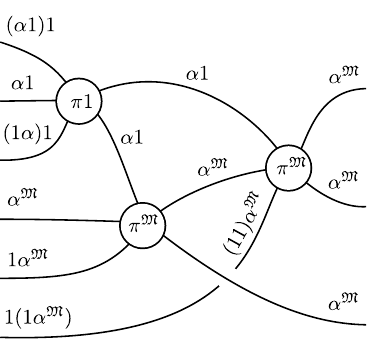}

\end{tabular}
\end{center}

\begin{enumerate}
    \item[] in $Hom_{\mathfrak{M}}((((A\Box B)\Box C)\Box D)\Diamond M,A\Diamond (B\Diamond (C\Diamond (D\Diamond M))))$;

    \item [b.] For every every $A,C$ in $\mathfrak{C}$ and $M$ in $\mathfrak{M}$, we have
    
\settoheight{\module}{\includegraphics[width=45mm]{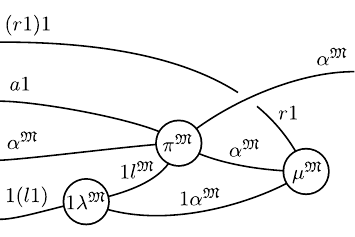}}

\begin{center}
\begin{tabular}{@{}ccc@{}}

\includegraphics[width=45mm]{Module2CategoriesPictures/module2cat/2modulecatunitor1.pdf} & \raisebox{0.45\module}{$=$} &

\includegraphics[width=37.5mm]{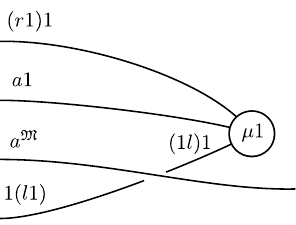}

\end{tabular}
\end{center}

in $Hom_{\mathfrak{M}}((A\Box C)\Diamond M,A\Diamond (C\Diamond M))$;
\item [c.] For every every $A,B$ in $\mathfrak{C}$ and $M$ in $\mathfrak{M}$, we have:

\settoheight{\module}{\includegraphics[width=45mm]{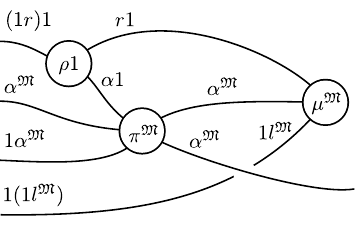}}

\begin{center}
\begin{tabular}{@{}ccc@{}}

\includegraphics[width=45mm]{Module2CategoriesPictures/module2cat/2modulecatunitor3.pdf} & \raisebox{0.45\module}{$=$} &

\includegraphics[width=37.5mm]{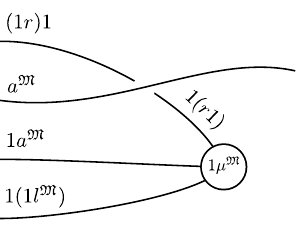}

\end{tabular}
\end{center}

in $Hom_{\mathfrak{M}}((A\Box B)\Diamond M,A\Diamond (B\Diamond M))$.
\end{enumerate}
\end{Definition}

\begin{Example}
The 2-category $\mathfrak{C}$ admits a canonical left $\mathfrak{C}$-module structure induced by its monoidal structure.
\end{Example}

\begin{Lemma}\label{lem:module2cat}
Definitions \ref{def:module2cat} and \ref{def:module2catunpacked} are equivalent. More precisely, given a 2-category $\mathfrak{M}$, there is a bijection between the set of left $\mathfrak{C}$-module structure according to definition \ref{def:module2cat} and the set of left $\mathfrak{C}$-module structure according to definition \ref{def:module2catunpacked}.
\end{Lemma}
\begin{proof}
Given a left $\mathfrak{C}$-module $(\mathfrak{M},H^{\mathfrak{M}})$ as in definition \ref{def:module2cat}, we can use the canonical isomorphism of 2-categories
$$Hom(\mathfrak{C},End(\mathfrak{M}))\cong Hom(\mathfrak{C}\times \mathfrak{M},\mathfrak{M})$$
to get a 2-functor $\Diamond:\mathfrak{C}\times \mathfrak{M}\rightarrow \mathfrak{M}$. Because $H^{\mathfrak{M}}$ is monoidal, we actually get more: We have an adjoint 2-natural equivalence $\chi^{H}$, given on $A,B$ in $\mathfrak{C}$ by $\chi^H_{A,B}:H^{\mathfrak{M}}(A)\circ H^{\mathfrak{M}}(B)\rightarrow H^{\mathfrak{M}}(A\Box B)$. Evaluating at $M$ in $\mathfrak{M}$ yields the components of an adjoint 2-natural equivalence $(\alpha^{\mathfrak{M}}_{A,B,M})^{\bullet}:A\Diamond (B\Diamond M)\rightarrow (A\Box B)\Diamond M$. Similarly, after evaluating at $M$ in $\mathfrak{M}$, the adjoint 2-natural equivalence $\iota^H: Id_{\mathfrak{M}}\Rightarrow H(I)$ induces an adjoint 2-natural equivalence with components $(l^{\mathfrak{M}}_M)^{\bullet}: M\rightarrow I\Diamond M$.

Now, recall that the monoidal 2-category $End(\mathfrak{M})$, while not strict, has coherence 1-morphisms whose components on objects are given by identity 1-morphisms. Furthermore, their components on 1-morphisms are canonical coherence 2-isomorphisms (see \cite{Gur2}). This implies that some of the 1-morphisms present in the target or the source of $\omega^H$, $\gamma^H$, and $\delta^H$ can be safely omitted. This can be used to show that the invertible modifications $\omega^H$, $\gamma^H$, and $\delta^H$ define invertible modifications $\pi^{\mathfrak{M}}$, $\lambda^{\mathfrak{M}}$, and $\mu^{\mathfrak{M}}$, which have the desired source and target. Finally, it is not hard to check that the monoidal coherence axioms on $H^{\mathfrak{M}}$ imply that axioms a, b, and c in definition \ref{def:module2catunpacked} are satisfied. The above procedure can also be carried out in reverse, proving the result.
\end{proof}

One can similarly unpack the definitions of a right $\mathfrak{D}$-module 2-category, and that of a $(\mathfrak{C},\mathfrak{D})$-bimodule 2-category. We leave the details to the reader. Now, we define the appropriate notions of a left $\mathfrak{C}$-module 2-functor.

\begin{Definition}\label{def:module2functor}
Let $(\mathfrak{M},H^{\mathfrak{M}})$, $(\mathfrak{N},H^{\mathfrak{N}})$ be two left $\mathfrak{C}$-module 2-categories. A left $\mathfrak{C}$-module 2-functor $\mathfrak{M}\rightarrow \mathfrak{N}$ is a monoidal 2-functor $F:\mathfrak{M}\rightarrow \mathfrak{N}$ together with an monoidal adjoint 2-natural equivalence $f$ fitting into the following diagram $$\begin{tikzcd}[sep=small]
End(\mathfrak{M}) \arrow[dd, "F^*"'] \arrow[rrdd, Rightarrow, "f", shorten > = 3ex, shorten < = 3ex] &  & \mathfrak{C} \arrow[ll, "H^{\mathfrak{M}}"'] \arrow[dd, "H^{\mathfrak{N}}"] \\
                                                       &  &                                                                             \\
{Hom(\mathfrak{M},\mathfrak{N})}                       &  & End(\mathfrak{N}) \arrow[ll, "F_*"]                                        
\end{tikzcd}$$ where $F_*$ denote the 2-functor given by precomposition by $F$, and $F^*$ denotes the 2-functor given by postcomposition by $F$.
\end{Definition}

\begin{Definition}\label{def:module2fununpacked}
Let $\mathfrak{M}$ and $\mathfrak{N}$ be two left $\mathfrak{C}$-module 2-categories as in definition \ref{def:module2catunpacked}. A left $\mathfrak{C}$-module 2-functor is a 2-functor $F:\mathfrak{M}\rightarrow \mathfrak{N}$ together with:
\begin{enumerate}
    \item A 2-natural adjoint equivalence $\chi^F$ given on $A$ in $\mathfrak{C}$, and $M$ in $\mathfrak{M}$ by $$\chi^F_{A,M}:A\Diamond^{\mathfrak{N}} F(M)\rightarrow F(A\Diamond^{\mathfrak{M}}M);$$
    \item Two invertible modifications $\omega^F$, and $\gamma^F$ given on $A,B$ in $\mathfrak{C}$ and $M$ in $\mathfrak{M}$ by 
\end{enumerate}

\begin{center}
\begin{tabular}{c}
 $$\begin{tikzcd}
 A (B F(M)) \arrow[rr, "1\chi^F", ""{name=a}] &   {} \arrow[Rightarrow, d, "\omega^F", shorten > = 1ex, shorten < = 1ex]    & A F(BM) \arrow[d, "\chi^F"] \\
(A B)F(M) \arrow[r, "\chi^F"'] \arrow[u, "\alpha^{\mathfrak{N}}"] &    F((AB)M) \arrow[r, "F(\alpha^{\mathfrak{M}})"']                                                                                                  &   F(A (BM)),\end{tikzcd}$$ \\

$$\begin{tikzcd}[sep=small]
I F(M) \arrow[dd, "\chi^F"']\arrow[rr, equal]  &                                        &    I F(M) \arrow[dd, "l^{\mathfrak{N}}"]  \\
{}\arrow[Rightarrow, rr,"\gamma^F", shorten > = 5.5ex, shorten < = 5.5ex] &  &  {}  \\
F(IM) \arrow[rr,"F(l^{\mathfrak{M}})"']  &      & F(M); 
\end{tikzcd}$$
\end{tabular}

\end{center}
Subject to the following relations:
\begin{enumerate}
    \item[a.] For every $A,B,C$ in $\mathfrak{C}$, and $M$ in $\mathfrak{M}$, the equality
\end{enumerate}

\settoheight{\module}{\includegraphics[width=52.5mm]{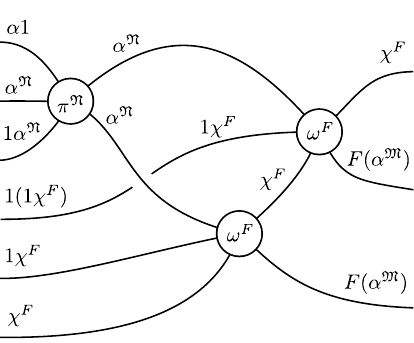}}

\begin{center}
\begin{tabular}{@{}ccc@{}}

\includegraphics[width=52.5mm]{Module2CategoriesPictures/module2fun/module2funpent1.pdf} & \raisebox{0.5\module}{$=$} &

\includegraphics[width=56.25mm]{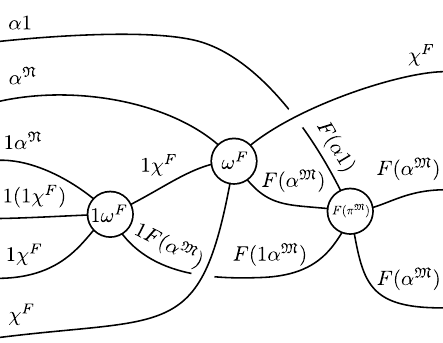}

\end{tabular}
\end{center}

\begin{enumerate}
    \item[] holds in $Hom_{\mathfrak{N}}(((A\Box B)\Box C)\Diamond F(M), F(A\Diamond (B\Diamond (C\Diamond M))))$;

    \item[b.] For every $A$ in $\mathfrak{C}$, and $M$ in $\mathfrak{M}$, the equality

\settoheight{\module}{\includegraphics[width=37.5mm]{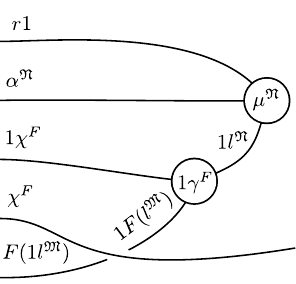}}

\begin{center}
\begin{tabular}{@{}ccc@{}}

\includegraphics[width=37.5mm]{Module2CategoriesPictures/module2fun/module2fununit1.pdf} & \raisebox{0.45\module}{$=$} &

\includegraphics[width=37.5mm]{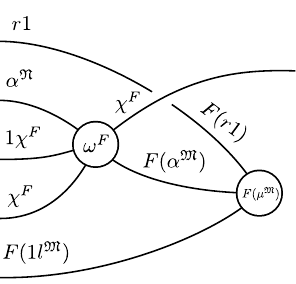}

\end{tabular}
\end{center}
    
    holds in $Hom_{\mathfrak{N}}(A\Diamond F(M), F(A\Diamond M))$;
    
    \item[c.] For every $B$ in $\mathfrak{C}$, and $M$ in $\mathfrak{M}$, the equality

\settoheight{\module}{\includegraphics[width=30mm]{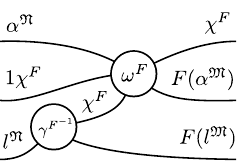}}

\begin{center}
\begin{tabular}{@{}ccc@{}}

\includegraphics[width=30mm]{Module2CategoriesPictures/module2fun/module2fununit3.pdf} & \raisebox{0.45\module}{$=$} &

\includegraphics[width=41.25mm]{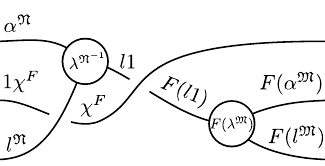}

\end{tabular}
\end{center}

holds in $Hom_{\mathfrak{N}}((I\Box B)\Diamond F(M), F(B\Diamond M)).$

\end{enumerate}
\end{Definition}

\begin{Lemma}
Definitions \ref{def:module2functor} and \ref{def:module2fununpacked} are equivalent.
\end{Lemma}
\begin{proof}
This is similar to the proof of lemma \ref{lem:module2cat}. We leave the details to the reader.
\end{proof}

Finally, we define left $\mathfrak{C}$-module 2-natural transformations and left $\mathfrak{C}$-module modifications. We only give one definition as we shall not use the other one.

\begin{Definition}\label{def:module2nat}
Let $F,G:\mathfrak{M}\rightarrow \mathfrak{N}$ be two left $\mathfrak{C}$-module 2-functors as in definition \ref{def:module2fununpacked}. A left $\mathfrak{C}$-module structure on a 2-natural transformation $\theta:F\Rightarrow G$ consists of an invertible modification $\Pi^{\theta}$ given on $A$ in $\mathfrak{C}$, and $M$ in $\mathfrak{M}$ by $$\begin{tikzcd}[sep=tiny]
A G(M) \arrow[ddd, "\chi^G"']\arrow[Rightarrow, rrrddd, "\Pi^{\theta}", shorten > = 2ex, shorten < = 2ex]  &                                        &    & A F(M) \arrow[ddd, "\chi^F"] \arrow[lll, "1 \theta"'] \\
 &  &    & \\
  &  &  &  \\
G(AM)                                            &                                        &    &  F(AM); \arrow[lll, "\theta"]
\end{tikzcd}$$

Subject to the following relations:

\begin{enumerate}
    \item[a.] For every $A,B$ in $\mathfrak{C}$, and $M$ in $\mathfrak{M}$, the equality
\end{enumerate}

\settoheight{\module}{\includegraphics[width=48.75mm]{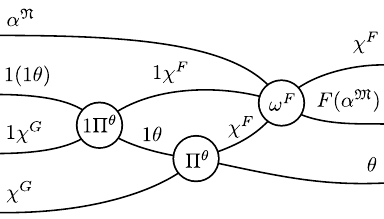}}

\begin{center}
\begin{tabular}{@{}ccc@{}}

\includegraphics[width=48.75mm]{Module2CategoriesPictures/module2natmod/module2nat1.pdf} & \raisebox{0.45\module}{$=$} &

\includegraphics[width=52.5mm]{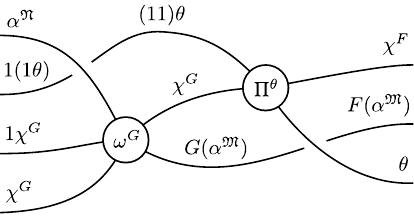}

\end{tabular}
\end{center}

\begin{enumerate}
    
    \item[] holds in $Hom_{\mathfrak{N}}((A\Box B)\Diamond^{\mathfrak{N}} F(M), G(A\Diamond^{\mathfrak{M}}(B\Diamond^{\mathfrak{N}}M)))$;
    
    \item[b.] For every $M$ in $\mathfrak{M}$, the equality

\settoheight{\module}{\includegraphics[width=45mm]{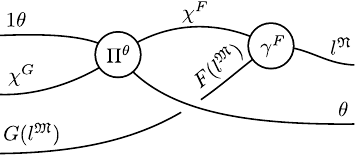}}

\begin{center}
\begin{tabular}{@{}ccc@{}}

\includegraphics[width=45mm]{Module2CategoriesPictures/module2natmod/module2nat3.pdf} & \raisebox{0.45\module}{$=$} &

\includegraphics[width=30mm]{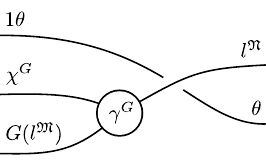}

\end{tabular}
\end{center}

holds in $Hom_{\mathfrak{N}}(I\Diamond^{\mathfrak{N}} F(M), G(M)).$

\end{enumerate}
\end{Definition}

\begin{Definition}\label{def:module2modif}
Let $\theta,\sigma:F\Rightarrow G$ be two left $\mathfrak{C}$-module 2-natural transformations. A left $\mathfrak{C}$-module modification is a modification $\Xi:\sigma\Rrightarrow \theta$ such that for every $A$ in $\mathfrak{C}$, and $M$ in $\mathfrak{M}$ the equality

\settoheight{\module}{\includegraphics[width=30mm]{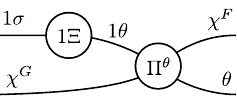}}

\begin{center}
\begin{tabular}{@{}ccc@{}}

\includegraphics[width=30mm]{Module2CategoriesPictures/module2natmod/modulemodif1.pdf} & \raisebox{0.4\module}{$=$} &

\includegraphics[width=30mm]{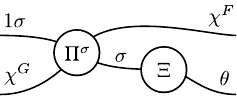}

\end{tabular}
\end{center}

holds in $Hom_{\mathfrak{N}}(A\Diamond^{\mathfrak{N}} F(M), G(A\Diamond^{\mathfrak{M}}M)).$
\end{Definition}

\begin{Convention}\label{con:linearity}
Let us now fix a ring $R$, and assume that $\mathfrak{C}$ is an $R$-linear monoidal 2-category, meaning that its monoidal product $\Box:\mathfrak{C}\times\mathfrak{C}\rightarrow \mathfrak{C}$ is $R$-bilinear. Then, the definitions of left $\mathfrak{C}$-module 2-category and left $\mathfrak{C}$-module 2-functors have to be slightly modified. Namely, in this case, a left $\mathfrak{C}$-module 2-category is an $R$-linear 2-category $\mathfrak{M}$ together with an $R$-linear monoidal 2-functor $\mathfrak{C}\rightarrow End_{R}(\mathfrak{M})$ or, equivalently, an $R$-bilinear 2-functor $\mathfrak{C}\times\mathfrak{M}\rightarrow \mathfrak{M}$ satisfying the axioms of definition \ref{def:module2catunpacked}. Similarly, left $\mathfrak{C}$-module 2-functors are required to be $R$-linear. Assuming that $\mathfrak{D}$ is an $R$-linear monoidal 2-category, we adopt the same conventions when discussing right $\mathfrak{D}$-module 2-categories. Finally, a $(\mathfrak{C},\mathfrak{D})$-bimodule 2-category is an $R$-linear 2-category $\mathfrak{M}$ together with an $R$-bilinear monoidal 2-functor $\mathfrak{C}\times \mathfrak{D}^{\Box op}\rightarrow End_{R}(\mathfrak{M})$ or, equivalently, an $R$-trilinear 2-functor $\mathfrak{C}\times\mathfrak{M}\times\mathfrak{D}^{\Box op}\rightarrow \mathfrak{M}$ satisfying certain axioms.
\end{Convention}

Let us fix an algebraically closed field $\mathds{k}$ of characteristic zero, and assume that $\mathfrak{C}$ is a multifusion 2-category over $\mathds{k}$ (see \cite{DR}, or \cite{D2}). We now give elementary examples of finite semisimple left $\mathfrak{C}$-module 2-categories.

\begin{Example}\label{ex:2Vectmodule}
Let $\mathfrak{M}$ be a finite semisimple 2-category. Then, $\mathfrak{M}$ is canonically a finite semisimple left $\mathbf{2Vect}$-module 2-category. Namely, the linear 2-category $End(\mathfrak{M})$ is a finite semisimple 2-category, as can be seen using the equivalence of theorem 2.2.2 of \cite{D1}. Now, the inclusion of the monoidal unit in $End(\mathfrak{M})$ is clearly a monoidal 2-functor, and thanks to the fact that the target is finite semisimple, it can be extended to a monoidal 2-functor $H:\mathbf{2Vect}\rightarrow End(\mathfrak{M})$. (In fact, this monoidal 2-functor is unique up to equivalence.) Using lemma \ref{lem:module2cat}, this corresponds to the left action given by letting $\mathbf{Vect}^{\oplus n}$ in $\mathbf{2Vect}$ act on $M$ in $\mathfrak{M}$ by $$\mathbf{Vect}^{\oplus n}\Diamond M = M^{\boxplus n}.$$ Similarly, one can define a right $\mathbf{2Vect}$-module structure on $\mathfrak{M}$. Further, these left and right action are suitably compatible providing us with a canonical $(\mathbf{2Vect},\mathbf{2Vect})$-bimodule structure on $\mathfrak{M}$.
\end{Example}

\begin{Example}\label{ex:2VectGmodule}
Fix $G$ a finite group, it is not hard to see that finite semisimple left $\mathbf{2Vect}_G$-module 2-categories are precisely finite semisimple 2-categories equipped with a $G$-action. As an example, one can consider the trivial $G$-action on $\mathbf{2Vect}$. This corresponds to the left $\mathbf{2Vect}_G$-module structure given by the forgetful monoidal 2-functor $\mathbf{2Vect}_G\rightarrow \mathbf{2Vect}$.
\end{Example}

\subsection{Properties}

We fix $\mathfrak{C}$ a monoidal 2-cateogry. We begin by proving that some of the objects defined in the previous susbection form a 2-category.

\begin{Proposition}\label{prop:hom2cat}
Let $\mathfrak{M}$, and $\mathfrak{N}$ be two left $\mathfrak{C}$-module 2-categories. Left $\mathfrak{C}$-module 2-functors $\mathfrak{M}\rightarrow \mathfrak{N}$, left $\mathfrak{C}$-module 2-natural transformations, and left $\mathfrak{C}$-module modifications form a 2-category, which we denote by $Hom_{\mathfrak{C}}(\mathfrak{M},\mathfrak{N})$.
\end{Proposition}
\begin{proof}
Let us begin by observing that left $\mathfrak{C}$-module modifications between two fixed left $\mathfrak{C}$-module 2-natural transformations form a 1-category. Let us now explain how to construct the desired 2-category structure. When doing so, we make use of our convention not to write down the coherence 2-isomorphisms coming from the axioms of a 2-category. Given a $\mathfrak{C}$-module 2-functor $F:\mathfrak{M}\rightarrow \mathfrak{N}$, the identity left $\mathfrak{C}$-module 2-natural transformation $Id_F^{\mathfrak{C}}$ is the identity 2-natural transformation $Id_F:F\Rightarrow F$ equipped with the canonical modification $\Pi^{Id_F^{\mathfrak{C}}}:\chi^F \circ (Id_{\mathfrak{C}}\Diamond^{\mathfrak{M}} Id_{\mathfrak{M}})\Rrightarrow Id_{\mathfrak{N}}\circ\chi^F$ given on $A$ in $\mathfrak{C}$ and $M$ in $\mathfrak{M}$ by $(\Pi^{Id_F^{\mathfrak{C}}})_{A,M}:= \chi^F_{A,M}\circ \phi^{\Diamond^{\mathfrak{N}}}_{A,F(M)}.$

Given $\mathfrak{C}$-module 2-functors $F,G,H:\mathfrak{M}\rightarrow \mathfrak{N}$, and left $\mathfrak{C}$-module 2-natural transformations $\theta:F\Rightarrow G$, and $\xi:G\Rightarrow H$, their composition is given by equipping the composite 2-natural transformation $\xi \cdot \theta$ with the invertible modification

\settoheight{\module}{\includegraphics[width=45mm]{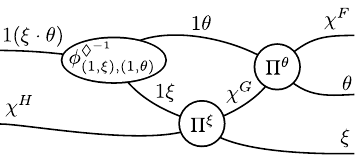}}

\begin{center}
\begin{tabular}{@{}cc@{}}

 \raisebox{0.45\module}{$\Pi^{\xi \cdot \theta}:=$} &

\includegraphics[width=45mm]{Module2CategoriesPictures/module2natmod/composition.pdf}.

\end{tabular}
\end{center}

Using this definition, it follows that if we are given another $\mathfrak{C}$-module 2-functor $K:\mathfrak{M}\rightarrow \mathfrak{N}$, and left $\mathfrak{C}$-module 2-natural transformations $\zeta:H\Rightarrow K$, then $$\Pi^{(\zeta\cdot\xi)\cdot\theta} = \Pi^{\zeta\cdot (\xi\cdot\theta)}.$$ This means that, as the invertible left $\mathfrak{C}$-module modification witnessing that composition is associative, we can take the identity modification $(\zeta\cdot\xi)\cdot\theta\Rrightarrow (\xi\cdot\theta)$. Note that these choices immediately satisfy the pentagon equation.

Let us now provide the 2-isomorphism witnessing unitality. Given $\theta:F\Rightarrow G$ as above, we need to supply two left $\mathfrak{C}$-module invertible modifications $\lambda_{\theta}:Id_G^{\mathfrak{C}}\cdot \theta\Rrightarrow \theta$ and $\rho_{\theta}:\theta\cdot Id_F^{\mathfrak{C}}\Rrightarrow \theta$. Inspection shows that $$\Pi^{Id_G^{\mathfrak{C}}\cdot \theta} = \Pi^{\theta} = \Pi^{\theta\cdot Id_F^{\mathfrak{C}}},$$ meaning that we can set both $\lambda_{\theta}=Id_{\theta}$ and $\rho_{\theta}=Id_{\theta}$. These assignments satisfy the triangle identity, which concludes the proof of the proposition.
\end{proof}

\begin{Remark}
In fact, proposition \ref{prop:hom2cat} is a shadow of the fact that left $\mathfrak{C}$-module 2-categories, etc, form a 3-category.
\end{Remark}

The following lemma can be used to construct examples of module 2-categories.

\begin{Lemma}\label{lem:functorsgivemodules}
Let $\mathfrak{D}$ be a monoidal 2-category, $\mathfrak{M}$ be a left $\mathfrak{D}$-module 2-category, and $F:\mathfrak{C}\rightarrow \mathfrak{D}$ be a monoidal 2-functor. Then, $\mathfrak{M}$ admits a canonical structure of a left $\mathfrak{C}$-module 2-category.
\end{Lemma}
\begin{proof}
The proof follows immediately from definition \ref{def:module2cat}.
\end{proof}

\begin{Example}
Let $\mathfrak{C}$ be a fusion 2-category. The inclusion $\mathfrak{C}^0\hookrightarrow \mathfrak{C}$ of the connected component of the identity is monoidal by proposition 2.4.5 of \cite{D2}, whence we can use lemma \ref{lem:functorsgivemodules} to see that $\mathfrak{C}$ is a finite semisimple left $\mathfrak{C}^0$-module 2-category.
\end{Example}

\begin{Example}
Given a braided monoidal functor $\mathcal{B}_1\rightarrow \mathcal{B}_2$ between two braided fusion categories, we get a monoidal 2-functor $F:\mathbf{Mod}(\mathcal{B}_1)\rightarrow \mathbf{Mod}(\mathcal{B}_2)$ between fusion 2-categories using proposition 2.4.7 of \cite{D2}. Thus, we can use lemma \ref{lem:functorsgivemodules} to find that $\mathbf{Mod}(\mathcal{B}_2)$ admits a left $\mathbf{Mod}(\mathcal{B}_1)$-module structure, and so is a finite semisimple left $\mathbf{Mod}(\mathcal{B}_1)$-module 2-category.
\end{Example}

In fact, one can generalize the above example to characterize finite semisimple left $\mathbf{Mod}(\mathcal{B})$-module 2-categories.

\begin{Lemma}\label{lem:Mod(B)modules}
Equivalence classes of finite semisimple left $\mathbf{Mod}(\mathcal{B})$-module 2-categories correspond precisely to equivalence classes of multifusion categories $\mathcal{C}$ equipped with a central monoidal functor $\mathcal{B}\rightarrow \mathcal{C}$.
\end{Lemma}
\begin{proof}
Let $\mathfrak{M}$ be a finite semisimple left $\mathbf{Mod}(\mathcal{B})$-module 2-category. Fix $M$ an object of $\mathfrak{M}$ such that $$\mathrm{B}End_{\mathfrak{M}}(M)\hookrightarrow \mathfrak{M}$$ is a Cauchy completion (see \cite{D1} for the definition). Further, recall that $\mathcal{B}$ (viewed as a $\mathcal{B}$-module) is such that $\mathrm{B}End_{\mathbf{Mod}(\mathcal{B})}(\mathcal{B})\hookrightarrow \mathbf{Mod}(\mathcal{B})$ is a Cauchy completion. Using the universal property of the Cauchy completion (see \cite{D1}), the left action by $\mathbf{Mod}(\mathcal{B})$ on $\mathfrak{M}$ is specified by the left action of $\mathrm{B}End_{Mod(\mathcal{B})}(\mathcal{B})\simeq \mathbf{Mod}(\mathcal{B})$ on $\mathrm{B}End_{\mathfrak{M}}(M)$. Inspection shows that this is precisely the data of a central monoidal functor $\mathcal{B}\rightarrow End_{\mathfrak{M}}(M)$. The converse follows using the above argument in reverse order.
\end{proof}

We now state a coherence result for left module 2-categories. Similar, results hold for right module 2-categories as well as for bimodule 2-categories.

\begin{Definition}
A pair $(\mathfrak{C},\mathfrak{M})$ consisting of a monoidal 2-category $\mathfrak{C}$ and a left $\mathfrak{C}$-module 2-category $\mathfrak{M}$ is called strict cubical provided that:
\begin{enumerate}
    \item The monoidal 2-category $\mathfrak{C}$ is strict cubical,
    \item The underlying 2-category of $\mathfrak{M}$ is strict,
    \item The 2-functor $\Diamond^{\mathfrak{M}}$ is strict cubical,
    \item The adjoint 2-natural equivalences $\alpha^{\mathfrak{M}}$ and $l^{\mathfrak{M}}$ are given by identity adjoint 2-natural transformations, and the invertible modifications $\mu^{\mathfrak{M}}$, $\lambda^{\mathfrak{M}}$ and $\pi^{\mathfrak{M}}$ are given by identity 2-morphisms.
\end{enumerate}
\end{Definition}

\begin{Proposition}{(Coherence for module 2-categories)}\label{prop:coherence}
Given any pair $(\mathfrak{C},\mathfrak{M})$ consisting of a monoidal 2-category $\mathfrak{C}$, and a left $\mathfrak{C}$-module 2-category $\mathfrak{M}$, there exists
\begin{enumerate}
    \item A monoidal 2-category $\mathfrak{D}$ together with a left $\mathfrak{D}$-module 2-category $\mathfrak{N}$ such that the pair $(\mathfrak{D},\mathfrak{N})$ is strict cubical,
    \item An equivalence of monoidal 2-categories $F:\mathfrak{C}\rightarrow \mathfrak{D}$,
    \item An equivalence of left $\mathfrak{C}$-module 2-categories $H:\mathfrak{M}\rightarrow \mathfrak{N}$ (the left $\mathfrak{C}$-module structure on $\mathfrak{N}$ is that supplied by lemma \ref{lem:functorsgivemodules}).
\end{enumerate}
\end{Proposition}
\begin{proof}
We can view the pair $(\mathfrak{C},\mathfrak{M})$ as a 3-category $\mathbb{A}$ with two objects $A$, and $B$ by making the following assignments

$$\begin{tabular}{rclrcl}
$Hom_{\mathbb{A}}(A,A)$ &$=$& $\mathfrak{C},$& $Hom_{\mathbb{A}}(B,A)$ &$=$ &$\mathfrak{M},$\\
$Hom_{\mathbb{A}}(B,B)$ &$=$& $*,$&$Hom_{\mathbb{A}}(A,B)$ &$=$&$\emptyset.$
\end{tabular}$$

Now, thanks to the coherence theorem for 3-categories (see \cite{Gur}), there exists an equivalence of 3-categories $\mathbb{A}\rightarrow \mathbb{B},$ where $\mathbb{B}$ is a Gray-category. Unpacking what this means, we find exactly the statement of the proposition.
\end{proof}

\begin{Remark}\label{rem:coherence}
In fact, the proof of proposition \ref{prop:coherence} can be generalized to apply simultaneously to more than one left $\mathfrak{C}$-module 2-category. Namely, if $\mathfrak{M}$ and $\mathfrak{N}$ are two left $\mathfrak{C}$-module 2-categories, we can assemble them into a 3-category with three object, and use the coherence theorem for 3-categories.
\end{Remark}

As an application of proposition \ref{prop:coherence}, we prove the following generalization of lemma 1.2.8 of \cite{D2}, which will be used later on.

\begin{Lemma}\label{lem:adjunctionduals}
Let $\mathfrak{C}$ be a monoidal 2-category, $\mathfrak{M}$ a left $\mathfrak{C}$-module 2-category, $C$ an object of $\mathfrak{C}$, and $M$, $N$ objects of $\mathfrak{M}$. If $C$ has a left dual, then there is an adjoint 2-natural equivalence

$$b_{C,M,N}:Hom_{\mathfrak{M}}(^{\sharp}C\Diamond M, N)\simeq Hom_{\mathfrak{M}}(M,C\Diamond N).$$

\end{Lemma}
\begin{proof}
By proposition \ref{prop:coherence}, we may assume that the pair $(\mathfrak{C},\mathfrak{M})$ is strict cubical. Let $(^{\sharp}C,C,i_C,e_C,C_C,D_C)$ be a coherent dual pair (see \cite{Pstr} or \cite{D2}). Using definition \ref{def:module2catunpacked}, we can write down linear functors

\begin{center}
\begin{tabular}{ r c l }
$Hom_{\mathfrak{M}}(^{\sharp}C\Diamond M, N)$ & $\rightleftarrows$ & $Hom_{\mathfrak{M}}(M, C\Diamond N)$ \\ 
$f$ & $\mapsto$ & $(C\Diamond f)\circ(i_C\Diamond M)$ \\  
$(e_C\Diamond N)\circ (^{\sharp}C\Diamond g)$ & $\mapsfrom$ & $g$   
\end{tabular}
\end{center}

that form adjoint equivalences. Namely, the unit $\eta^b_{C,M,N}$ and the counit $\epsilon^b_{C,M,N}$ of these adjunctions are given by

\settoheight{\module}{\includegraphics[width=22.5mm]{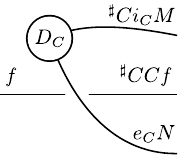}}

\begin{center}
\begin{tabular}{@{} c c c c c @{}}
\raisebox{0.45\module}{$\eta^b_{C,M,N}:=$}&\includegraphics[width=22.5mm]{Module2CategoriesPictures/prebalanced/unitb.pdf},&\phantom{AA}&\raisebox{0.45\module}{$\epsilon^b_{C,M,N}:=$}&\includegraphics[width=22.5mm]{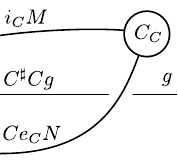}.
\end{tabular}
\end{center}

The triangle identities follow from the swallowtail equations satisfied by $C_C$ and $D_C$. Finally, these adjoint equivalences are manifestly 2-natural in $M$ and $N$, which proves the result.
\end{proof}

Finally, for application to our main theorem, we prove the following technical lemma, which is a generalization of the categorical Schur lemma of \cite{DR}.

\begin{Lemma}\label{lem:Schurmodule}
Let $\mathfrak{C}$ be a multifusion 2-category, and $\mathfrak{M}$ a finite semisimple left $\mathfrak{C}$-module 2-category. Let $A,B$ be objects of $\mathfrak{C}$ such that $B$ is simple, and $M,N$ be objects of $\mathfrak{M}$. Then, for any non-zero 1-morphism $f:B\Diamond M\rightarrow N$ in $\mathfrak{M}$, and non-zero 1-morphism $g:A\rightarrow B$ in $\mathfrak{C}$, the composite $f\circ (g\Diamond M)$ is non-zero.
\end{Lemma}
\begin{proof}
Let us assume that the pair $(\mathfrak{C},\mathfrak{M})$ is strict cubical. As $\mathfrak{C}$ is a finite semisimple 2-category, $g$ has a right adjoint, which we denote by $g^*$, and we write $\epsilon^g$ for the counit of this adjunction. Now, as $g$ is non-zero, $\epsilon^g:g\circ g^*\Rightarrow Id_B$ is non-zero. Further, as $B$ is simple, $Id_B$ is simple, so $\epsilon^g$ is a projection onto a summand. If the composite $f\circ (g\Diamond M)$ were zero, so would the composite $f\circ (g\Diamond M)\circ (g^*\Diamond M)$, and consequently the 2-morphism $f\circ (\epsilon^g\Diamond M)$. As $\epsilon^g$ has a section, we would get that $f$ is zero, which contradicts our hypotheses.
\end{proof}

\section{Algebras and Modules in Monoidal 2-Categories}\label{sec:algebramodule}

\subsection{Algebras}

Throughout, we work in a fixed monoidal 2-category $\mathfrak{C}$.

\begin{Definition}\label{def:algebra}
An algebra in $\mathfrak{C}$ consists of:
\begin{enumerate}
    \item An object $A$ of $\mathfrak{C}$;
    \item Two 1-morphisms $m:A\Box A\rightarrow A$ and $i:I\rightarrow A$;
    \item Three 2-isomorphisms
\end{enumerate}
\begin{center}
$\begin{tikzcd}[sep=tiny]
                           & A(AA) \arrow[ld, "1m"'] &                                 & (AA)A \arrow[ll, "\alpha"'] \arrow[rd, "m1"] &                           \\
AA \arrow[rrdd, "m"']  & {}\arrow[rr, "\mu^{A}", Rightarrow, shorten >=5ex, shorten <=5ex ]                                                      &  &  {}                                      & AA, \arrow[lldd, " m"] \\
                           &                                                        &                               &                                       &                           \\
                           &                                                        & A                              &                                       &                          
\end{tikzcd}$

\begin{tabular}{@{}c c@{}}
$$\begin{tikzcd}[sep=small]
AI  \arrow[rr, "1i"] &{}\arrow[dd, Rightarrow,"\rho^{A}", near start, shorten > = 2ex, shorten < = 1ex]  & AA \arrow[dd, "m"]\\
&  & \\
A \arrow[uu, "r"] & {} & A, \arrow[equal, ll]
\end{tikzcd}$$

&

$$\begin{tikzcd}[sep=tiny]
A \arrow[ddd, equal]\arrow[Rightarrow, rrrddd, "\lambda^{A}", shorten > = 1.6ex, shorten < = 1.6ex]  &                                        &    & IA \arrow[ddd, "i1"] \arrow[lll, "l"'] \\
 &  &    & \\
  &  &  &  \\
A                                            &                                        &    &  AA; \arrow[lll, "m"]
\end{tikzcd}$$
\end{tabular}
\end{center}

satisfying:

\begin{enumerate}
\item [a.] We have:
\end{enumerate}

\newlength{\algebra}
\settoheight{\algebra}{\includegraphics[width=52.5mm]{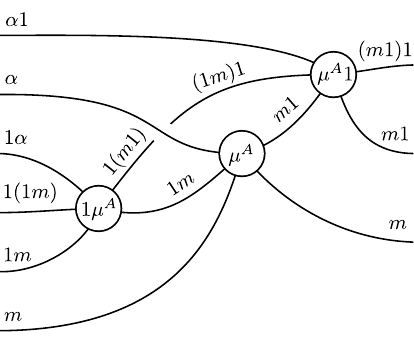}}

\begin{center}
\begin{tabular}{@{}ccc@{}}

\includegraphics[width=52.5mm]{Module2CategoriesPictures/algebra/associativity1.pdf} & \raisebox{0.45\algebra}{$=$} &
\includegraphics[width=56.25mm]{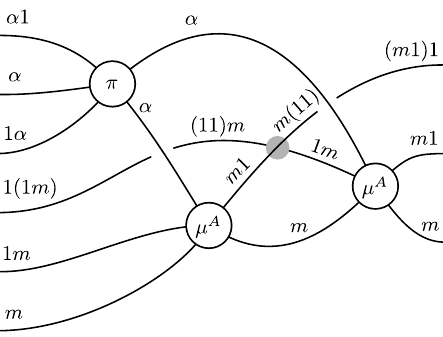};

\end{tabular}
\end{center}

\begin{enumerate}
\item [b.] We have:
\end{enumerate}

\settoheight{\algebra}{\includegraphics[width=45mm]{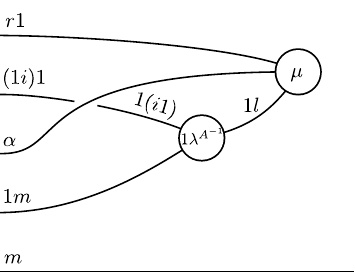}}

\begin{center}
\begin{tabular}{@{}ccc@{}}

\includegraphics[width=45mm]{Module2CategoriesPictures/algebra/unitality1.pdf} & \raisebox{0.45\algebra}{$=$} &

\includegraphics[width=37.5mm]{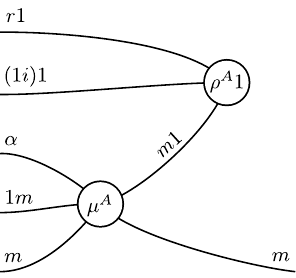}.

\end{tabular}
\end{center}
\end{Definition}

\begin{Remark}
The objects we have called algebras are precisely what \cite{DS} refers to as pseudo-monoids.
\end{Remark}

\begin{Example}\label{ex:algebra2vect}
It is well-known (see \cite{DS}) that algebras in the 2-category of (small) categories with monoidal structure given by the Cartesian product correspond precisely to (small) monoidal categories.
\end{Example}

\begin{Example}\label{ex:algebra2vectG}
Algebras in $\mathbf{2Vect}$ correspond precisely to finite semisimple monoidal categories. (This is essentially lemma 3.2 of \cite{BDSPV}.) Slightly more generally, given a finite group $G$, algebras in $\mathbf{2Vect}_G$ correspond to $G$-graded finite semisimple monoidal categories.
\end{Example}

\begin{Example}
Algebras in the Drinfel'd center $\mathscr{Z}(\mathbf{2Vect}_G)$ (see \cite{KTZ} for a construction) are given exactly finite semisimple $G$-crossed monoidal categories (originally introduced in section 2.1 of \cite{T}, see also definition 5.1 of \cite{Gal}). Furthermore, braided algebras in $\mathscr{Z}(\mathbf{2Vect}_G)$ are precisely finite semisimple $G$-crossed braided monoidal categories. We plan on investigating the properties of these algebras further in future work.
\end{Example}

\begin{Example}\label{ex:algebrasbraidedfusion}
Let $\mathcal{B}$ be a braided fusion category.By proposition 3.2 of \cite{BJS}, algebras in $\mathbf{Mod}(\mathcal{B})$ correspond to finite semisimple monoidal categories $\mathcal{C}$ equipped with a central monoidal functor $\mathcal{B}\rightarrow \mathcal{C}$.
\end{Example}

\begin{Definition}\label{def:1morphismalgebras}
Let $A$ and $B$ be two algebras in $\mathfrak{C}$. A 1-morphism of algebras in $\mathfrak{C}$ from $A$ to $B$ consists of a 1-morphism $f:A\rightarrow B$ in $\mathfrak{C}$ together with two invertible 2-morphisms
 
\begin{center}
\begin{tabular}{@{}c c@{}}
$$\begin{tikzcd}[sep=tiny]
AA \arrow[rrr, "f1"] \arrow[dddd, "m^A"'] & & & BA \arrow[rrr, "1f"] &                                     &    & BB \arrow[dddd, "m^B"] \\
                                                  &&&    &  &      &                             \\
                                                   &       {} \arrow[rrrr, "\kappa^f", Rightarrow, shorten <=4ex, shorten >=4ex]         &&&&                   {}    &                              \\
                                                   &               &&&                     &  &                              \\
A \arrow[rrrrrr, "f"']                                &         &&  &                          &    & B,                         
\end{tikzcd}$$

&

$$\begin{tikzcd}[sep=tiny]
I \arrow[rrr, equal] \arrow[ddd, "i^A"'] &                                     &    & I \arrow[ddd, "i_B"] \\
                                                     &  &  {}   &                              \\
                                                   &  {}                            &  &                              \\
A \arrow[rrr, "f"'] \arrow[uuurrr, "\iota^f", Rightarrow, shorten <=4ex, shorten >=4ex]                               &                                     &    & B;                         
\end{tikzcd}$$
\end{tabular}
\end{center}

satisfying:

\begin{enumerate}
\item [a.] We have:
\end{enumerate}

\settoheight{\algebra}{\includegraphics[width=45mm]{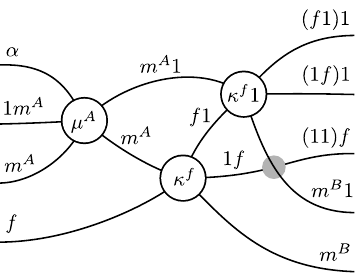}}

\begin{center}
\begin{tabular}{@{}ccc@{}}

\includegraphics[width=45mm]{Module2CategoriesPictures/algebra/morphismassociativity1.pdf} & \raisebox{0.45\algebra}{$=$} &

\includegraphics[width=45mm]{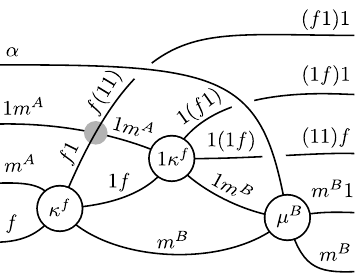};

\end{tabular}
\end{center}

\begin{enumerate}
\item [b.] We have:
\end{enumerate}

\settoheight{\algebra}{\includegraphics[width=45mm]{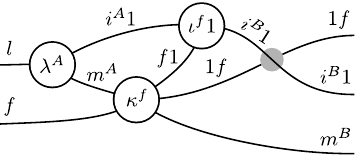}}

\begin{center}
\begin{tabular}{@{}ccc@{}}

\includegraphics[width=45mm]{Module2CategoriesPictures/algebra/morphismunitality1.pdf} & \raisebox{0.45\algebra}{$=$} &

\includegraphics[width=30mm]{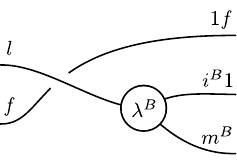};

\end{tabular}
\end{center}

\begin{enumerate}
\item [c.] We have:
\end{enumerate}

\settoheight{\algebra}{\includegraphics[width=45mm]{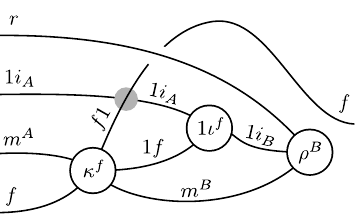}}

\begin{center}
\begin{tabular}{@{}ccc@{}}

\includegraphics[width=45mm]{Module2CategoriesPictures/algebra/morphismunitality3.pdf} & \raisebox{0.45\algebra}{$=$} &

\includegraphics[width=30mm]{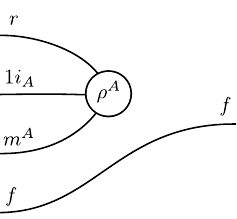}.

\end{tabular}
\end{center}
\end{Definition}

\begin{Example}
Expanding on example \ref{ex:algebrasbraidedfusion} slightly, one can check that 1-morphisms of algebras in $\mathbf{Mod}(\mathcal{B})$ correspond to monoidal functors compatible with the central structure.
\end{Example}

\begin{Definition}\label{def:2morphismalgebras}
Let $f,g:A\rightarrow B$ be two 1-morphisms of algebras in $\mathfrak{C}$. A 2-morphism of algebras in $\mathfrak{C}$ from $f$ to $g$ is a 2-morphism $\delta:f\Rightarrow g$ in $\mathfrak{C}$ such that:

\begin{enumerate}
\item [a.] We have:
\end{enumerate}

\settoheight{\algebra}{\includegraphics[width=30mm]{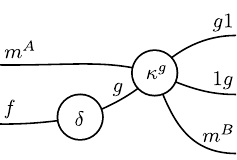}}

\begin{center}
\begin{tabular}{@{}ccc@{}}

\includegraphics[width=30mm]{Module2CategoriesPictures/algebra/algebra2morphism1.pdf} & \raisebox{0.45\algebra}{$=$} &

\includegraphics[width=30mm]{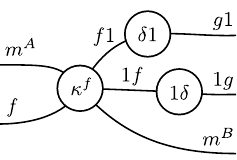},

\end{tabular}
\end{center}

\begin{enumerate}
\item [b.] We have:
\end{enumerate}

\settoheight{\algebra}{\includegraphics[width=30mm]{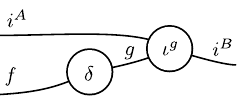}}

\begin{center}
\begin{tabular}{@{}ccc@{}}

\includegraphics[width=30mm]{Module2CategoriesPictures/algebra/algebra2morphism3.pdf} & \raisebox{0.45\algebra}{$=$} &

\includegraphics[width=22.5mm]{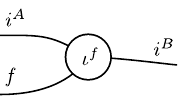}.

\end{tabular}
\end{center}
\end{Definition}

\subsection{Modules}

Let us fix $A$ an algebra in $\mathfrak{C}$.

\begin{Definition}\label{def:module}
A right $A$-module consists of:
\begin{enumerate}
    \item An object $M$ of $\mathfrak{C}$;
    \item A 1-morphism $n:M\Box A\rightarrow M$;
    \item Two 2-isomorphisms
\end{enumerate}

\begin{center}
\begin{tabular}{@{}c c@{}}
$\begin{tikzcd}[sep=tiny]
                           & M(AA) \arrow[ld, "1m"'] &                                 & (MA)A \arrow[ll, "\alpha"'] \arrow[rd, "n1"] &                           \\
MA \arrow[rrdd, "n"']  & {}\arrow[rr, "\nu^{M}", Rightarrow, shorten >=5ex, shorten <=5ex ]                                                      &  &  {}                                      & MA, \arrow[lldd, " n"] \\
                           &                                                        &                               &                                       &                           \\
                           &                                                        & M                              &                                       &                          
\end{tikzcd}$

&

$$\begin{tikzcd}[sep=small]
MI  \arrow[rr, "1i"] &{}\arrow[dd, Rightarrow,"\rho^{M}", near start, shorten > = 2ex, shorten < = 1ex]  & MA \arrow[dd, "n"]\\
&  & \\
M \arrow[uu, "r"] & {} & M; \arrow[equal, ll]
\end{tikzcd}$$
\end{tabular}
\end{center}
    
Such that:
\begin{enumerate}
\item [a.] We have:
\end{enumerate}

\settoheight{\algebra}{\includegraphics[width=52.5mm]{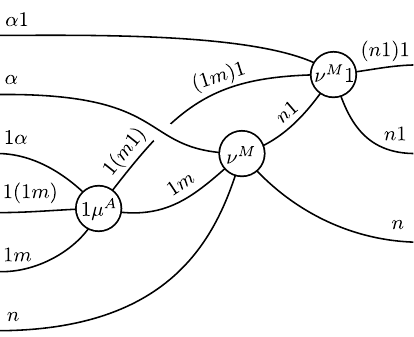}}

\begin{center}
\begin{tabular}{@{}ccc@{}}

\includegraphics[width=52.5mm]{Module2CategoriesPictures/modules/associativity1.pdf} & \raisebox{0.45\algebra}{$=$} &

\includegraphics[width=56.25mm]{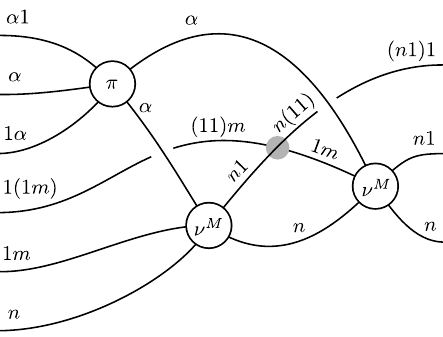},

\end{tabular}
\end{center}

\begin{enumerate}
\item [b.] We have:
\end{enumerate}

\settoheight{\algebra}{\includegraphics[width=45mm]{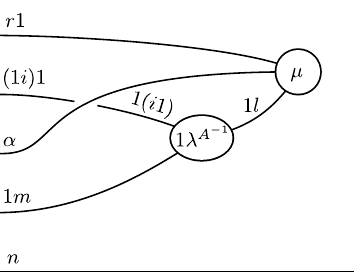}}

\begin{center}
\begin{tabular}{@{}ccc@{}}

\includegraphics[width=45mm]{Module2CategoriesPictures/modules/unitality1.pdf} & \raisebox{0.45\algebra}{$=$} &

\includegraphics[width=37.5mm]{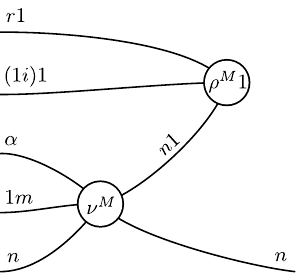}.

\end{tabular}
\end{center}
\end{Definition}

\begin{Example}
The algebra $A$ can canonically be viewed as a right $A$-module using the 2-isomorphisms $\mu^A$ and $\rho^A$.
\end{Example}

\begin{Example}
Given an algebra $A$ in $\mathbf{2Vect}$, which corresponds to a finite semisimple monoidal category $\mathcal{C}$ via example \ref{ex:algebra2vect}, right $A$-modules in $\mathbf{2Vect}$ correspond exactly to finite semisimple right $\mathcal{C}$-module categories.
\end{Example}

\begin{Example}
Slightly more generally, using example \ref{ex:algebra2vectG}, we have that giving an algebra $A$ in $\mathbf{2Vect}_G$ is the same thing as giving a $G$-graded finite semisimple monoidal category $\mathcal{C}$. Right $A$-modules in $\mathbf{2Vect}_G$ are precisely those $G$-graded finite semisimple right $\mathcal{C}$-module categories $\mathcal{M}$ for which the action $\mathcal{M}\times\mathcal{C}\rightarrow\mathcal{M}$ is compatible with the $G$-gradings.
\end{Example}

\begin{Example}
Let $A$ be an algebra in $\mathbf{Mod}(\mathcal{B})$ for some braided fusion category $\mathcal{B}$. By example \ref{ex:algebrasbraidedfusion}, this corresponds to a fusion category $\mathcal{C}$ equipped with a central monoidal functor $\mathcal{B}\rightarrow \mathcal{C}$. One can check that right $A$-modules in $\mathbf{Mod}(\mathcal{B})$ are exactly given by finite semisimple right $\mathcal{C}$-module categories.
\end{Example}

\begin{Definition}\label{def:modulemap}
Let $M$ and $N$ be two $A$-modules. A right $A$-module 1-morphism consists of a 1-morphism $f:M\rightarrow N$ in $\mathfrak{C}$ together with an invertible 2-morphism

$$\begin{tikzcd}[sep=tiny]
MA \arrow[rrr, "n^M"] \arrow[ddd, "f1"'] &                                     &    & M \arrow[ddd, "f"] \\
                                                     &  &  {}    &                              \\
                                                   &       {} \arrow[ur, "\psi^f", Rightarrow]                              &  &                              \\
NA \arrow[rrr, "n^N"']                                &                                     &    & N,                         
\end{tikzcd}$$
subject to the coherence relations:

\begin{enumerate}
\item [a.] We have:
\end{enumerate}

\settoheight{\algebra}{\includegraphics[width=52.5mm]{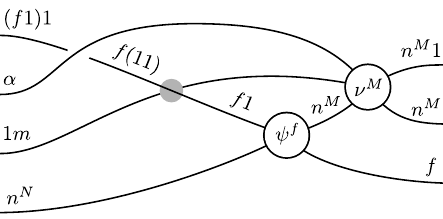}}

\begin{center}
\begin{tabular}{@{}ccc@{}}

\includegraphics[width=52.5mm]{Module2CategoriesPictures/modules/map1.pdf} & \raisebox{0.45\algebra}{$=$} &

\includegraphics[width=56.25mm]{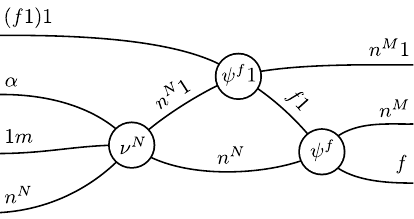};

\end{tabular}
\end{center}

\begin{enumerate}
\item [b.] We have:
\end{enumerate}

\settoheight{\algebra}{\includegraphics[width=30mm]{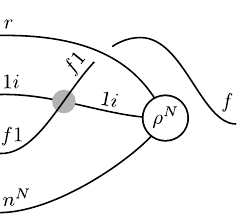}}

\begin{center}
\begin{tabular}{@{}ccc@{}}

\includegraphics[width=30mm]{Module2CategoriesPictures/modules/map3.pdf} & \raisebox{0.45\algebra}{$=$} &

\includegraphics[width=30mm]{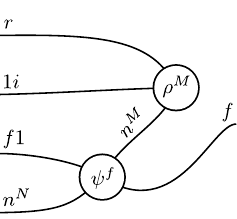}.

\end{tabular}
\end{center}
\end{Definition}

\begin{Definition}\label{def:moduleintertwiner}
Let $M$ and $N$ be two right $A$-modules, and $f,g:M\rightarrow M$ two right $A$-module 1-morphisms. A right $A$-module 2-morphism $f\Rightarrow g$ is a 2-morphism $\gamma:f\Rightarrow g$ in $\mathfrak{C}$ that satisfies the equality

\settoheight{\algebra}{\includegraphics[width=30mm]{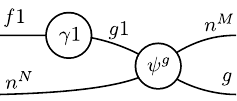}}

\begin{center}
\begin{tabular}{@{}ccc@{}}

\includegraphics[width=30mm]{Module2CategoriesPictures/modules/2morphism1.pdf} & \raisebox{0.45\algebra}{$=$} &

\includegraphics[width=30mm]{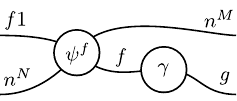}.

\end{tabular}
\end{center}
\end{Definition}

\begin{Remark}
One can similarly define left $A$-modules, left $A$-module 1-morphisms, and left $A$-module 2-morphisms. Further, given two algebras $A$ and $B$ in $\mathfrak{C}$, it is possible to define $(A,B)$-bimodules, etc. We leave the details to the reader, but present an alternative perspective below.
\end{Remark}

\begin{Remark}\label{rem:pseudomonads}
Observe that any algebra $A$ in $\mathfrak{C}$ yields a pseudo-monad $T:=(-)\Box A:\mathfrak{C}\rightarrow \mathfrak{C}$ (see \cite{CMV} for the definition). In particular, it is not hard to see that our definition of a right $A$-module in $\mathfrak{C}$ corresponds exactly to that of a pseudo-$T$-algebra. Further, right $A$-module 1-morphisms are precisely 1-morphism of $T$-algebras, and right $A$-module 2-morphisms are precisely 2-morphisms of $T$-algebras. Through this perspective the following lemmas are standard.
\end{Remark}

\begin{Lemma}\label{lem:mod2cat}
Right $A$-modules in $\mathfrak{C}$, right $A$-module 1-morphisms, and right $A$-module 2-morphisms form a 2-category, which we denote by $\mathbf{Mod}_{\mathfrak{C}}(A)$.
\end{Lemma}
\begin{proof}
This can be shown directly, or one can use remark \ref{rem:pseudomonads}, and appeal to the well-known fact that there is a 2-category of pseudo-algebras associated to any pseudo-monad. For latter use, let us mention that the composite of two right $A$-module maps $f:M\rightarrow N$ and $g:N\rightarrow P$ is given by $g\circ f:M\rightarrow P$ equipped with the 2-isomorphism

\settoheight{\algebra}{\includegraphics[width=30mm]{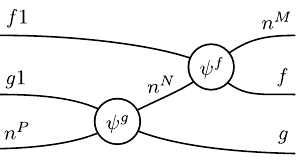}}

\begin{center}
\begin{tabular}{@{}cc@{}}

\raisebox{0.45\algebra}{$\psi^{g\circ f}=$} &

\includegraphics[width=30mm]{Module2CategoriesPictures/modules/compositionmap.pdf}.

\end{tabular}
\end{center}
\end{proof}

\begin{Remark}
If we assume that $\mathfrak{C}$ is strict cubical, then $\mathbf{Mod}_{\mathfrak{C}}(A)$ is in fact a strict 2-category.
\end{Remark}

\begin{Notation}
Let $M$, $N$ be two right $A$-modules. We use $Hom_A(M,N)$ to denote the category of right $A$-module 1-morphisms $M\rightarrow N$, and right $A$-module 2-morphisms between them.
\end{Notation}

\begin{Lemma}\label{lem:freeforgetadj}
For any $C$ in $\mathfrak{C}$, and right $A$-module $M$, there is an adjoint 2-natural equivalence $$Hom_A(C\Box A, M)\simeq Hom_{\mathfrak{C}}(C,M).$$
\end{Lemma}
\begin{proof}
Following remark \ref{rem:pseudomonads}, this is an instance of the standard free-forgetful 2-adjunction for the pseudomonad $(-)\Box A$.
\end{proof}

\begin{Remark}
We can generalize the perspective of remark \ref{rem:pseudomonads}. We can consider the pseudo-monad given on $\mathfrak{C}$ by $T':=A\Box (-):\mathfrak{C}\rightarrow \mathfrak{C}$. Pseudo-algebras for $T'$ are precisely left $A$-modules, which shows that there is a 2-category of left $A$-modules in $\mathfrak{C}$. Given another algebra $B$, we can consider the pseudo-monad $T'':=(A\Box (-))\Box B:\mathfrak{C}\rightarrow \mathfrak{C}$. Pseudo-algebras for $T''$ are exactly $(A,B)$-bimodules in $\mathfrak{C}$. Thence, we get a 2-category of $(A,B)$-bimodules in $\mathfrak{C}$, which we denote by $\mathbf{Bimod}_{\mathfrak{C}}(A,B)$.
\end{Remark}

\subsection{Properties}

Expanding on lemma \ref{lem:module2cat} a little bit, we can prove the following result.

\begin{Proposition}\label{prop:catmodmodcat}
The 2-category $\mathbf{Mod}_{\mathfrak{C}}(A)$ is a left $\mathfrak{C}$-module 2-category.
\end{Proposition}
\begin{proof}
Using the coherence theorem for monoidal 2-categories, we may assume that $\mathfrak{C}$ is a strict cubical monoidal 2-category. As pointed out above, this implies that $\mathbf{Mod}_{\mathfrak{C}}(A)$ is a strict 2-category. Let us begin by constructing the 2-functor $\Box:\mathbf{Mod}_{\mathfrak{C}}(A)\times\mathfrak{C}\rightarrow \mathbf{Mod}_{\mathfrak{C}}(A)$ providing us with the desired action.

Firstly, given a right $A$-module $M$ and an object $C$ of $\mathfrak{C}$, the object $C\Box M$ can be endowed with a canonical right $A$-module structure as follows: The action 1-morphism is given by $C\Box n^M:C\Box M\Box A\rightarrow C\Box M$. The structure 2-isomorphisms are given by $\nu^{C\Box M}=C\Box \nu^{M}$ and $\rho^{C\Box M}=C\Box \rho^M$. Given $f:M\rightarrow N$ a right $A$-module 1-morphism and $j:C\rightarrow D$ a 1-morphism in $\mathfrak{C}$, the left $A$-module structure on $j\Box f = (D\Box f)\circ (j\Box M)$ is given by:

\settoheight{\algebra}{\includegraphics[width=37.5mm]{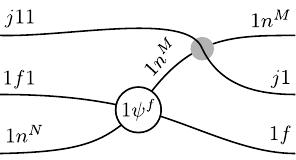}}

\begin{center}
\begin{tabular}{@{}cc@{}}

\raisebox{0.45\algebra}{$\psi^{j\Box f}=$} &

\includegraphics[width=37.5mm]{Module2CategoriesPictures/modules/tensormap.pdf}.

\end{tabular}
\end{center}

Finally, given $\gamma:f\Rightarrow g$ a right $A$-module 2-morphism, and $\alpha:j\Rightarrow k$ a 2-morphism in $\mathfrak{C}$, we see that $\alpha\Box\gamma$ is a right $A$-module 2-morphism.

Thanks to our strictness hypothesis, the only non-trivial structure 2-isomorphisms that are not identity 2-morphisms are the interchangers $\phi$ for the 2-functor $\Box:\mathbf{Mod}_{\mathfrak{C}}(A)\times \mathfrak{C}\rightarrow \mathbf{Mod}_{\mathfrak{C}}(A)$. Given $f:M\rightarrow N$, $g:N\rightarrow P$ two right $A$-module 1-morphisms, and $j:C\rightarrow D$, $k:D\rightarrow E$ two 1-morphisms in $\mathfrak{C}$, we define $\phi_{(k,g),(j,f)}:=\phi^{\Box}_{(k,g),(j,f)}$, the interchanger for $\Box:\mathfrak{C}\times\mathfrak{C}\rightarrow\mathfrak{C}$. It follows easily from the definition that $\phi$ is suitably natural and satisfies the required coherence axioms. It remains to be proven that $\phi_{(k,g),(j,f)}$ is a right $A$-module 2-morphism. This follows from the fact that the following two string diagrams are equal:

\settoheight{\algebra}{\includegraphics[width=60mm]{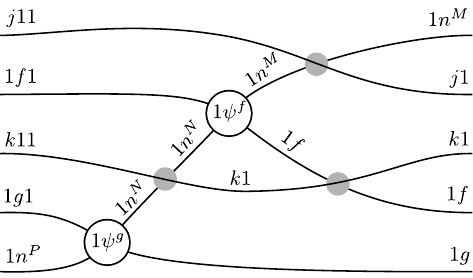}}

\begin{center}
\begin{tabular}{@{}ccc@{}}

\includegraphics[width=60mm]{Module2CategoriesPictures/modules/tensor2morphism1.pdf} & \raisebox{0.45\algebra}{$=$} &

\includegraphics[width=45mm]{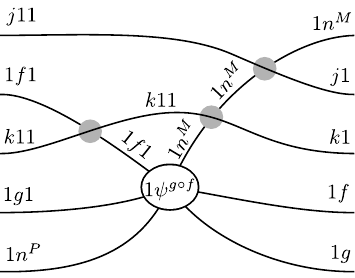}.

\end{tabular}
\end{center}
    
As we have assumed that $\mathfrak{C}$ is strict cubical, we can safely set $\alpha^{\mathbf{Mod}_{\mathfrak{C}}(A)} = Id$, and $l^{\mathbf{Mod}_{\mathfrak{C}}(A)}=Id$ as adjoint 2-natural equivalences, and $\pi^{\mathbf{Mod}_{\mathfrak{C}}(A)}=Id$, $\lambda^{\mathbf{Mod}_{\mathfrak{C}}(A)}=Id$, and $\mu^{\mathbf{Mod}_{\mathfrak{C}}(A)}=Id$ as modifications. Thence, it is not hard to show that the remaining coherence axioms hold.
\end{proof}

\begin{Remark}
Let $R$ be a ring, and assume that $\mathfrak{C}$ is an $R$-linear monoidal 2-category. Given algebras $A$, $B$ algebras in $\mathfrak{C}$, we see that $\mathbf{Mod}_{\mathfrak{C}}(A)$ and $\mathbf{Bimod}_{\mathfrak{C}}(A,B)$ are $R$-linear 2-categories. Further, as the left $\mathfrak{C}$-module structure on $\mathbf{Mod}_{\mathfrak{C}}(A)$ comes from the monoidal structure of $\mathfrak{C}$, we see readily that the 2-functor $\mathfrak{C}\times \mathbf{Mod}_{\mathfrak{C}}(A)\rightarrow \mathbf{Mod}_{\mathfrak{C}}(A)$ is $R$-bilinear, whence $\mathbf{Mod}_{\mathfrak{C}}(A)$ satisfies the conditions of convention \ref{con:linearity}.
\end{Remark}

Now, we explain how 1-morphisms of algebras yield 2-functors between 2-categories of modules. 

\begin{Lemma}\label{lem:1morphismalgebras2functor}
Let $f:A\rightarrow B$ be a 1-morphism of algebras in $\mathfrak{C}$. Then, there is a 2-functor $f^{\bigstar}:\mathbf{Mod}_{\mathfrak{C}}(B)\rightarrow \mathbf{Mod}_{\mathfrak{C}}(A)$.
\end{Lemma}
\begin{proof}
Without loss of generality, we may assume that the $\mathfrak{C}$ is strict cubical. Let $M$ be a right module over $B$. We set $f^{\bigstar}M$ be the right $A$-module given by the object $M$, the 1-morphisms $n^{f^{\bigstar}M}:= n^M\circ (M\Box f)$, and the 2-isomorphisms

\settoheight{\algebra}{\includegraphics[width=45mm]{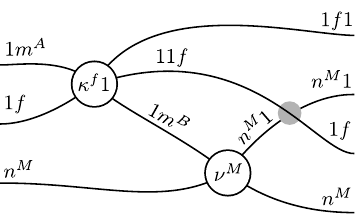}}

\begin{center}
\begin{tabular}{@{}cc@{}}

\raisebox{0.45\algebra}{$\nu^{f^{\bigstar}M}=$} &

\includegraphics[width=45mm]{Module2CategoriesPictures/moduleslemmas/pullbackmodulenu.pdf},

\end{tabular}

\settoheight{\algebra}{\includegraphics[width=30mm]{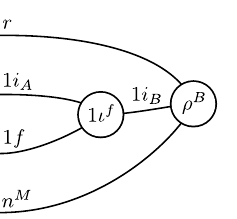}}

\begin{tabular}{@{}cc@{}}

\raisebox{0.45\algebra}{$\rho^{f^{\bigstar}M}=$} &

\includegraphics[width=30mm]{Module2CategoriesPictures/moduleslemmas/pullbackmodulelambda.pdf}.

\end{tabular}
\end{center}

Given a right $B$-module 1-morphism $g:M\rightarrow N$, we let $f^{\bigstar}g$ be the right $A$-module 1-morphism given by the 1-morphism $g$, and the 2-isomorphism

\settoheight{\algebra}{\includegraphics[width=30mm]{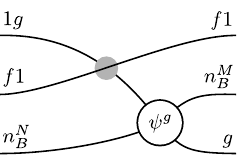}}

\begin{center}
\begin{tabular}{@{}cc@{}}

\raisebox{0.45\algebra}{$\psi^{f^{\bigstar}g}=$} &

\includegraphics[width=30mm]{Module2CategoriesPictures/moduleslemmas/pullback1morphism.pdf}.

\end{tabular}
\end{center}

Given a right $B$-module 2-morphism $\alpha:g\Rightarrow h$, it is not hard to check that $\alpha$ defines a right $A$-module 2-morphism $g_A\Rightarrow h_A$, so we set $f^{\bigstar}\alpha := \alpha$.

Finally, using the definitions above, it is easy to check that $f^{\bigstar}$ is a strict 2-functor.
\end{proof}

\begin{Remark}
The 2-functor constructed in lemmas \ref{lem:1morphismalgebras2functor} can be upgraded to a right $\mathfrak{C}$-module 2-functor. As we shall not need this fact, we leave the details to the interested reader. Further, we note that, as expected, a 2-morphisms of algebras yield a 2-natural transformation between the corresponding 2-functors.
\end{Remark}

Now, fix an algebraically closed field $\mathds{k}$ of characteristic zero. We derive some properties of the 2-category of modules over an algebra in a multifusion 2-category.

\begin{Proposition}\label{prop:categoryofmodulesproperties}
Let $\mathfrak{C}$ be a (multi)fusion 2-category over $\mathds{k}$ and $A$ an algebra in $\mathfrak{C}$, then the $\mathds{k}$-linear 2-category $\mathbf{Mod}_{\mathfrak{C}}(A)$ has the following properties:
\begin{enumerate}
    \item Its $Hom$-categories are Cauchy complete.
    \item Give two parallel right $A$-module 1-morphisms $f$ and $g$, $Hom_A(f,g)$ is a finite dimensional $\mathds{k}$-vector space.
    \item The 2-category $\mathbf{Mod}_{\mathfrak{C}}(A)$ is Cauchy complete.
\end{enumerate}
\end{Proposition}

\begin{proof}
The first point follows from the fact that direct sums and splittings of idempotents are preserved by all linear functors, and that the $Hom$-categories of $\mathfrak{C}$ are Cauchy complete by definition. We leave the details to the reader.

We now prove 2. Let $M,N$ be two right $A$-modules, and $f,g:M\rightarrow N$ be two right $A$-module 1-morphisms. We have already seen that $Hom_A(f,g)$ is a $\mathds{k}$-vector space. Further, by hypothesis, we know that $Hom_{\mathfrak{C}}(f,g)$ is finite dimensional vector space. But $Hom_A(f,g)\subseteq Hom_{\mathfrak{C}}(f,g)$, thus $Hom_{A}(f,g)$ is a finite dimensional vector space.

The existence of direct sums in $\mathbf{Mod}_{\mathfrak{C}}(A)$ follows immediately from their existence in $\mathfrak{C}$. So, to prove 3, it only remains to show that every 2-condensation monad in $\mathbf{Mod}_{\mathfrak{C}}(A)$ splits (see \cite{D1} for a precise definition). Without loss of generality, we assume that $\mathfrak{C}$ is strict cubical. Let $(M,e,\xi,\delta)$ be an arbitrary 2-condensation in $\mathbf{Mod}_{\mathfrak{C}}(A)$. Forgetting the right $A$-module structure, we get a 2-condensation in $\mathfrak{C}$, which we also denote by $(M,e,\xi,\delta)$. By hypothesis, $\mathfrak{C}$ is Cauchy complete, so there exists a 2-condensation $(M,N,f,g,\phi,\gamma)$, and a 2-isomorphism $\theta:g\circ f\cong e$ splitting $(M,e,\xi,\delta)$, i.e. satisfying $$\xi = \theta\cdot (g\circ \phi\circ f)\cdot(\theta^{-1} \circ \theta^{-1}),$$
$$\delta = (\theta \circ \theta)\cdot (g\circ \gamma\circ f)\cdot\theta^{-1}.$$ We claim that $(M,N,f,g,\phi,\gamma)$ and $\theta$ can be upgraded to a splitting of the 2-condensation monad $(M,e,\xi,\delta)$ in $\mathbf{Mod}_{\mathfrak{C}}(A)$. Firstly, we endow $N$ with the structure of a right $A$-module by setting $n^N:= f\circ n^M\circ (g\Box A)$, and the 2-isomorphisms

\settoheight{\algebra}{\includegraphics[width=75mm]{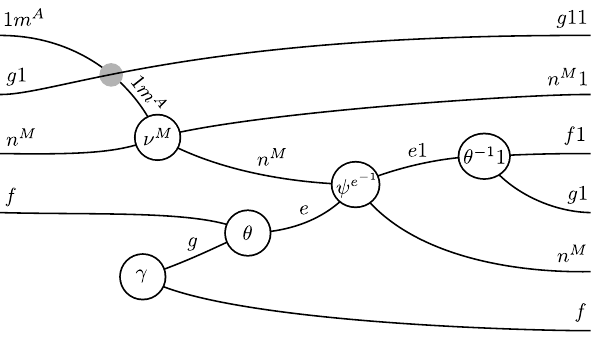}}

\begin{center}
\begin{tabular}{@{}cc@{}}

\raisebox{0.45\algebra}{$\nu^{N}=$} &

\includegraphics[width=75mm]{Module2CategoriesPictures/condensation/nuN.pdf},

\end{tabular}

\settoheight{\algebra}{\includegraphics[width=30mm]{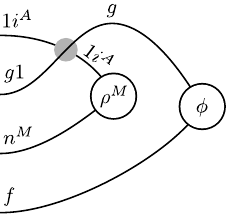}}

\begin{tabular}{@{}cc@{}}

\raisebox{0.45\algebra}{$\rho^{N}=$} &

\includegraphics[width=30mm]{Module2CategoriesPictures/condensation/rhoN.pdf}.

\end{tabular}
\end{center}

Now, we need to show how to upgrade $f$ and $g$ to right $A$-module 1-morphisms. This is achieved by making the following assignments:

\settoheight{\algebra}{\includegraphics[width=45mm]{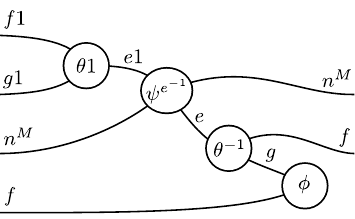}}

\begin{center}
\begin{tabular}{@{}cc@{}}

\raisebox{0.45\algebra}{$\psi^f=$} &

\includegraphics[width=45mm]{Module2CategoriesPictures/condensation/psif.pdf},

\end{tabular}

\settoheight{\algebra}{\includegraphics[width=45mm]{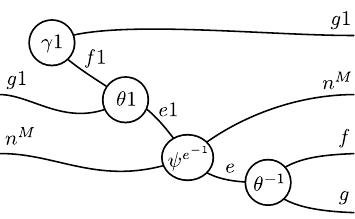}}

\begin{tabular}{@{}cc@{}}

\raisebox{0.45\algebra}{$\psi^g=$} &

\includegraphics[width=45mm]{Module2CategoriesPictures/condensation/psig.pdf}.

\end{tabular}
\end{center}

One can check that the pair consisting of $f$ and $\psi^f$ , and that consisting of $g$ and $\psi^g$ are indeed a right $A$-module 1-morphism. Finally, using the fact that $e$ is a right $A$-module 1-morphism together with the two equations satisfied by $\theta$, it is not difficult to show that the 2-morphisms $\phi$, $\gamma$, and $\theta$ are right $A$-module 2-morphisms. This upgrades $(M,N,f,g,\phi,\gamma)$ to a 2-condensation in $\mathbf{Mod}_{\mathfrak{C}}(A)$, which splits $(M,e,\xi,\delta)$ via $\theta$.
\end{proof}

\begin{Remark}
It would be interesting to characterise those 2-categories that can be obtained as the 2-category of modules over an algebra in $\mathbf{2Vect}$. More generally, it would be interesting to give a definition of ``finite'' 2-category. One would expect that the 2-category of exact module categories over a finite tensor category (see \cite{EGNO} for the definition) to be an example of such a ``finite'' 2-category.
\end{Remark}

\begin{Remark}
Proposition \ref{prop:categoryofmodulesproperties} also holds in the more general context of compact semisimple tensor 2-categories introduced in \cite{D5}. Namely, if $\mathds{k}$ is any field, $\mathfrak{C}$ is a compact semisimple tensor 2-category, and $A$ is an algebra in $\mathfrak{C}$, then one finds that $\mathbf{Mod}_{\mathfrak{C}}(A)$ satisfies all three properties given in the statement of proposition \ref{prop:categoryofmodulesproperties} by following the proof supplied above essentially verbatim.
\end{Remark}

\section{Enrichment of Finite Semisimple Module 2-Categories}\label{sec:enrichment}

\subsection{Definition}

Throughout this section, unless otherwise specified, we work over an algebraically closed field $\mathds{k}$ of characteristic zero, $\mathfrak{C}$ is a multifusion 2-category over $\mathds{k}$, and $\mathfrak{M}$ is a finite semisimple left $\mathfrak{C}$-module 2-category. Further, we write $\mathfrak{C}^{1op}$, and $\mathfrak{M}^{1op}$ for the 2-category obtained from $\mathfrak{C}$, and $\mathfrak{M}$ be reversing the direction of the 1-morphisms.

\begin{Proposition}\label{prop:adjointmodule}
There is a linear 2-functor $\underline{Hom}_{\mathfrak{M}}(-,-):\mathfrak{M}^{1op}\times \mathfrak{M}\rightarrow \mathfrak{C}$, such that there is an adjoint 2-natural equivalence $$t:Hom_{\mathfrak{M}}(C\Diamond M,N)\simeq Hom_{\mathfrak{C}}(C,\underline{Hom}_{\mathfrak{M}}(M,N))$$ for every $M,N$ in $\mathfrak{M}$ and $C$ in $\mathfrak{C}$.
\end{Proposition}
\begin{proof}
Let us consider the linear 2-functor $$\begin{tabular}{cccc}$F:$ &$\mathfrak{M}^{op}\times \mathfrak{M}$&$\rightarrow$ &$\mathbf{Fun}(\mathfrak{C}^{1op},\mathbf{2Vect}).$\\  &$(M,N)$&$\mapsto$ &$Hom_{\mathfrak{M}}((-)\Diamond M,N)$\end{tabular}$$
By proposition 1.4.1 of \cite{DR}, the assignment $$\begin{tabular}{cccc}$G:$&$\mathfrak{C}$&$\simeq$&$ \mathbf{Fun}(\mathfrak{C}^{1op},\mathbf{2Vect})$\\ &$C$&$\mapsto$&$ Hom_{\mathfrak{C}}(-,C)$\end{tabular}$$ is an equivalence of 2-categories, and thus $G$ admits a pseudo-inverse, which we denote by $G^*$. We write $$\underline{Hom}_{\mathfrak{M}}(-,-):=G^*\circ F:\mathfrak{M}^{op}\times \mathfrak{M}\rightarrow \mathfrak{C}$$ for the composite 2-functor. Now, using the Yoneda lemma for linear 2-categories, we obtain a 2-natural equivalence of categories
$$\begin{tabular}{rcl} $Hom_{\mathfrak{C}}(C, \underline{Hom}_{\mathfrak{M}}(M,N))$ &$\simeq$& $Hom_{\mathbf{Fun}(\mathfrak{C}^{1op},\mathbf{2Vect})}(G(C), F(M,N))$\\ &$\simeq$& $Hom_{\mathfrak{M}}(C\Diamond M, N),$\end{tabular}$$ for every $C$ in $\mathfrak{C}$, and $M,N$ in $\mathfrak{M}$. As every 2-natural equivalence can be upgraded to an adjoint 2-natural equivalence, the proof of the result is completed by performing said upgrade.
\end{proof}

\begin{Example}
Every finite semisimple 2-category $\mathfrak{N}$ is a left $\mathbf{2Vect}$-module category. In this case, we have $\underline{Hom}_{\mathfrak{N}}(M,N)= Hom_{\mathfrak{N}}(M,N)$ for every $M,N$ in $\mathfrak{N}$.
\end{Example}

\begin{Example}
The 2-category $\mathfrak{C}$ is canonically a left $\mathfrak{C}$-module 2-category. Using lemma 1.2.8 of \cite{D2}, we find that for every $A,B$ in $\mathfrak{C}$, we have: $$\underline{Hom}_{\mathfrak{C}}(A,B)\simeq B\Box (^{\sharp}A).$$
\end{Example}

\begin{Notation}
If there is no risk of ambiguity, we will write $\underline{Hom}(M,N)=\underline{Hom}_{\mathfrak{M}}(M,N)$ for every $M,N$ in $\mathfrak{M}$. In string diagrams, we will abbreviate the notation further by using $\underline{(M,N)}$ to denote $\underline{Hom}(M,N)$. In addition, when considering a finite semisimple 2-category $\mathfrak{N}$ as a finite semisimple left $\mathbf{2Vect}$-module 2-category, we write $Hom(P,Q)=\underline{Hom}_{\mathfrak{N}}(P,Q)$ for every $P,Q$ in $\mathfrak{N}$. In string diagrams, we will write $(P,Q)$ to denote $Hom(P,Q)$.
\end{Notation}

\begin{Remark}\label{rem:biadjunction}
Let us fix an object $M$ of $\mathfrak{M}$. Proposition \ref{prop:adjointmodule} implies that $(-)\Diamond M$ is a left 2-adjoint to $\underline{Hom}(M,-)$, meaning that there exists a unit 2-natural transformations $u_M$, given on  $C$ in $\mathfrak{C}$ by $$u_{M,C}:C\rightarrow \underline{Hom}(M,C\Diamond M),$$ a counit 2-natural transformation $c_M$, given on $N$ in $\mathfrak{M}$ by $$c_{M,N}:\underline{Hom}(M,N)\Diamond M\rightarrow N,$$ and invertible modifications $\Phi_M$ and $\Psi_M$, given on $C$ in $\mathfrak{M}$ and $N$ in $\mathfrak{M}$ by $$\Phi_{M,C}:(c_{M,C\Diamond M})\circ (u_{M,C}\Diamond Id_M)\cong Id_{C\Diamond M},$$ $$\Psi_{M,N}:\underline{Hom}(M,c_{M,N})\circ(u_{M,\underline{Hom}(M,N)})\cong Id_{\underline{Hom}(M,N)}.$$ The equivalence between these two definitions follows in the usual fashion from the Yoneda lemma for linear 2-categories. In particular, the value of the 2-natural equivalence $t$ of proposition \ref{prop:adjointmodule} on $f:C\Diamond M\rightarrow N$ is recovered by $$t(f)= \underline{Hom}(M,f)\circ u_{M,C}.$$ Analogously, the value of its pseudo-inverse $t^*$ on $g:C\rightarrow \underline{Hom}(M,N)$ is recovered by $$t^*(g)=c_{M,N}\circ (g\Diamond M).$$

For our purposes, it is important that this 2-adjunction be coherent, i.e. that $\Phi_M$ and $\Psi_M$ satisfy the swallowtail equations depicted in section \ref{sec:prelim}. Thanks to the multi-object version of the coherence of duals in monoidal 2-categories derived in \cite{Pstr}, this is always possible. For any $N$ in $\mathfrak{M}$, the first swallowtail equation corresponds to the following equality in $Hom_{\mathfrak{M}}(\underline{Hom}(M,N)\Diamond M,N)$:

\newlength{\adj}
\settoheight{\adj}{\includegraphics[width=37.5mm]{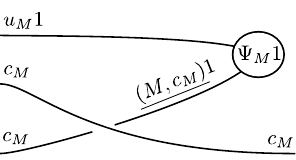}}

\begin{equation}\label{eqn:swallowtail1}
\begin{tabular}{@{}ccc@{}}
\includegraphics[width=37.5mm]{Module2CategoriesPictures/biadjunction/coherence1.pdf} & \raisebox{0.45\adj}{$=$} &

\includegraphics[width=30mm]{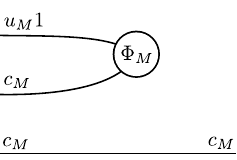}.
\end{tabular}
\end{equation}

As for the second swallowtail equation, for any $C$ in $\mathfrak{C}$, it corresponds to the equality in $Hom_{\mathfrak{C}}(C,\underline{Hom}(M,C\Diamond M))$ depicted below:

\begin{equation}\label{eqn:swallowtail2}
\begin{tabular}{@{}ccc@{}}
\includegraphics[width=45mm]{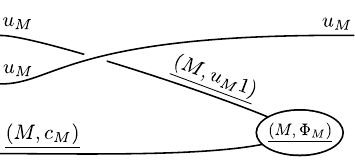} & \raisebox{0.45\adj}{$=$} &

\includegraphics[width=30mm]{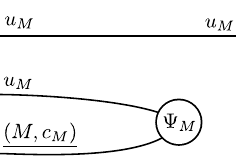}.
\end{tabular}
\end{equation}
\end{Remark}

As an application of the above remark, let us prove the following lemma.

\begin{Lemma}\label{lem:enrichmentidentification}
For every $M,N$ in $\mathfrak{M}$ and $C$ in $\mathfrak{C}$, the 1-morphism \begin{align*}
\chi^M_{C,N}: &\  C\Box \underline{Hom}(M,N)\xrightarrow{u_M} \underline{Hom}(M,(C\Box \underline{Hom}(M,N))\Diamond M)\xrightarrow{\underline{Hom}(M,\alpha^{\mathfrak{M}})}\\ 
&\  \underline{Hom}(M,C\Diamond (\underline{Hom}(M,N)\Diamond M))\xrightarrow{\underline{Hom}(M,1\Diamond c_M)}\underline{Hom}(M,C\Diamond N)
\end{align*} is an equivalence.
\end{Lemma}
\begin{proof}
Without loss of generality, we may assume that the pair $(\mathfrak{C},\mathfrak{M})$ is strict cubical. Recall the adjoint 2-natural equivalence $b$ constructed in lemma \ref{lem:adjunctionduals}. We claim that for every $A$ in $\mathfrak{C}$, there is a natural isomorphism $\theta^M_{A,C,N}$ witnessing the commutativity of the following diagram:
$$\adjustbox{scale=0.88}{\begin{tikzcd}[sep=tiny]
{Hom_{\mathfrak{C}}(A,C\Box \underline{Hom}(M,N))} \arrow[rrrr, "{Hom_{\mathfrak{C}}(A,\chi^M_{C,N})}"] \arrow[dd, "b^*"'] &  & {} \arrow[dd, Rightarrow, "\theta^M_{A,C,N}"] &  & {Hom_{\mathfrak{C}}(A,\underline{Hom}(M,C\Diamond N))} \\
 &  &  &  &   \\
{Hom_{\mathfrak{C}}(^{\sharp}C\Box A, \underline{Hom}(M,N))} \arrow[rr, "t^*"']  &  & {Hom_{\mathfrak{C}}((^{\sharp}C\Box A)\Diamond M, N)} \arrow[rr, "b"'] &  & {Hom_{\mathfrak{C}}(A\Diamond M, C\Diamond N)}. \arrow[uu, "t"']
\end{tikzcd}}$$

For any given 1-morphism $f:A\rightarrow C\Box\underline{Hom}(M,N)$ in $\mathfrak{C}$, we let $(\theta^M_{A,C,N})_f$ be the 2-isomorphism in $Hom_{\mathfrak{C}}(A, \underline{Hom}(M,C\Diamond N))$ defined by

\settoheight{\adj}{\includegraphics[width=90mm]{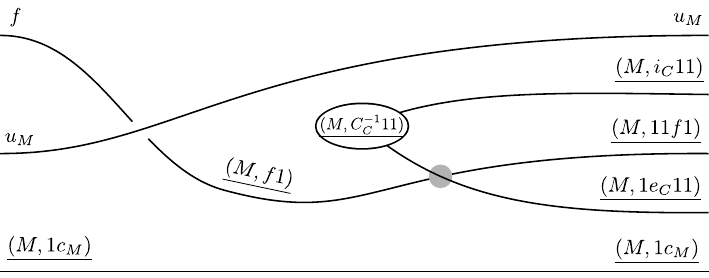}}

\begin{center}
\begin{tabular}{@{}cc@{}}
\raisebox{0.45\adj}{$(\theta^M_{A,C,N})_f:=$} &
\includegraphics[width=90mm]{Module2CategoriesPictures/enrichmentfunctor/chiequivalence.pdf}.
\end{tabular}
\end{center}

It is clear from the definition of $(\theta^M_{A,C,N})_f$ that these 2-isomorphisms assemble to give a natural isomorphism $\theta^M_{A,C,N}$. Further inspection of the construction of $\theta^M_{A,C,N}$ shows that the natural isomorphisms $\theta^M_{A,C,N}$ for varying $A$ assemble into an invertible modification. Consequently, the Yoneda lemma for 2-categories implies that the 1-morphism $\chi^M_{C,N}$ is an equivalence because the bottom composite of the above diagram is an equivalence.
\end{proof}

\begin{Proposition}\label{prop:enrichedhommod}
For any fixed $M$ in $\mathfrak{M}$, the 2-functor $$\underline{Hom}_{\mathfrak{M}}(M,-):\mathfrak{M}\rightarrow \mathfrak{C}$$ is a left $\mathfrak{C}$-module 2-functor.
\end{Proposition}
\begin{proof}
Without loss of generality, we may assume that the pair $(\mathfrak{C},\mathfrak{M})$ is strict cubical. For varying $C$ in $\mathfrak{C}$ and $N$ in $\mathfrak{M}$, the 1-morphisms $\chi^M_{C,N}$ of lemma \ref{lem:enrichmentidentification} assemble to give a 2-natural equivalence. Picking an adjoint, this provides us with the desired adjoint 2-natural equivalence.

We still have to specify two invertible modifications $\omega^M$ and $\gamma^M$. Given $N$ in $\mathfrak{M}$, we set $\gamma^M_N := \Psi_{M,N}$. Given $A,B$ in $\mathfrak{C}$, and $N$ in $\mathfrak{M}$, we let $\omega^M_{A,B,N}$ be the 2-isomorphism in $Hom_{\mathfrak{C}}(A\Box B\Box \underline{Hom}(M,N), \underline{Hom}(M, A\Diamond (B\Diamond M)))$ depicted below:

\settoheight{\adj}{\includegraphics[width=67.5mm]{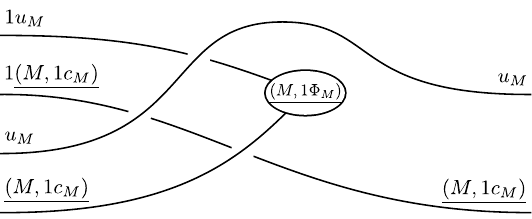}}

\begin{center}
\begin{tabular}{@{}cc@{}}
\raisebox{0.45\adj}{$\omega^M_{A,B,N}:=$} &
\includegraphics[width=67.5mm]{Module2CategoriesPictures/enrichmentfunctor/omega.pdf}.
\end{tabular}
\end{center}

We now have to check the coherence axioms of definition \ref{def:module2fununpacked}. The diagrams we will use to do this are given in appendix \ref{sec:diagcoherenceomega}. We begin by checking axioms b and c. Axiom b is equivalent to the assertion that the string diagram depicted in figure \ref{fig:omegaidentity} is the identity 2-morphism on $\chi^M_{A,N}$. Using naturality, we can pass the coupon labelled $1\Psi_M^{-1}$ underneath the strand labelled $u_M$. Then, using equation (\ref{eqn:swallowtail1}), we get the desired equality. Similarly, one can use naturality and equation (\ref{eqn:swallowtail2}) to prove that the diagram depicted in figure \ref{fig:omegaidentityc} is the identity 2-morphism on $\chi^M_{B,N}$, which proves that axiom c holds.

We go on to show that axiom a is satisfied. Figure \ref{fig:enrichedhomassociativity1} corresponds to the right hand-side of axiom a in definition \ref{def:module2fununpacked}. In order to pass to figure \ref{fig:enrichedhomassociativity2}, we move the strand labelled $u_M$ up using naturality. Using naturality again, we can move the coupon labelled $\underline{(M,1\underline{(M,1\Phi_M)}1)}$ under the one labelled $\underline{(M,1\Phi_M)}$, this gets us to figure \ref{fig:enrichedhomassociativity3}. Now, using naturality to move the coupon labelled $\underline{(M,1\Phi_M)}$ down, we arrive at figure \ref{fig:enrichedhomassociativity4}. Finally, using an isotopy, we end up with figure \ref{fig:enrichedhomassociativity5}, which corresponds to the left hand-side of axiom a in definition \ref{def:module2fununpacked}.
\end{proof}

\begin{Remark}\label{rem:EnrichmentGeneral}
Let us now momentarily assume that $\mathds{k}$ is an arbitrary field, $\mathfrak{C}$ is a compact semisimple tensor 2-category, and $\mathfrak{M}$ is a compact semisimple left $\mathfrak{C}$-module 2-category. Proposition \ref{prop:adjointmodule} does not hold at this level of generality, as can be seen from remark 3.2.2 of \cite{D5}. Nonetheless, if we assume that $\mathds{k}$ is perfect, and that $\mathfrak{C}$ is locally separable, then proposition \ref{prop:adjointmodule} does hold. Namely, with these additional hypotheses, we may appeal to theorem 3.2.1 of \cite{D5} in lieu of proposition 1.4.1 of \cite{DR} in the proof of proposition \ref{prop:adjointmodule}. Further, this is the only change required to make the proof hold. The other results in this subsection follow using the proofs we have given above up to the obvious minor modifications.
\end{Remark}

\subsection{Proof}

We begin by recalling definition 3.1 of \cite{GS}.

\begin{Definition} \label{def:enriched2cat}
A 2-category enriched over $\mathfrak{C}$, $\underline{\mathfrak{B}}$, consists of the following data:
\begin{enumerate}
\item A set of objects $Ob(\underline{\mathfrak{B}})$;
\item For every objects $A,B$ in $\underline{\mathfrak{B}}$, a $Hom$-object $Hom_{\underline{\mathfrak{B}}}(A,B)$ in $\mathfrak{C}$;
\item For every object $A$ in $\underline{\mathfrak{B}}$, a 1-morphism $j_A\rightarrow Hom_{\underline{\mathfrak{B}}}(A,A)$ in $\mathfrak{C}$;
\item For every objects $A,B,C$ in $\underline{\mathfrak{B}}$, a 1-morphism in $\mathfrak{C}$ $$m_{A,B,C}:Hom_{\underline{\mathfrak{B}}}(B,C)\Box Hom_{\underline{\mathfrak{B}}}(A,B)\rightarrow Hom_{\underline{\mathfrak{B}}}(A,C);$$
\item For every $A,B$ in $\underline{\mathfrak{B}}$, 2-isomorphisms $\sigma^{\underline{\mathfrak{B}}}_{A,B}$ and $\tau^{\underline{\mathfrak{B}}}_{A,B}$ in $\mathfrak{C}$
\end{enumerate}

$$\begin{tikzcd}[sep=tiny]
{Hom_{\underline{\mathfrak{B}}}(B,B)Hom_{\underline{\mathfrak{B}}}(A,B)}  \arrow[rrr, "m"]\arrow[rrrddd,Rightarrow, "\sigma^{\underline{\mathfrak{B}}}", shorten <= 4ex, shorten >= 4ex] &  & & {Hom_{\underline{\mathfrak{B}}}(A,B)}\\
&  &  &\\
& & &\\
{I\, Hom_{\underline{\mathfrak{B}}}(A,B)} \arrow[uuu, "j 1"]\arrow[rrr,equal] & & & {I\, Hom_{\underline{\mathfrak{B}}}(A,B)}\arrow[uuu, "l"'],
\end{tikzcd}$$

$$\begin{tikzcd}[sep=tiny]
{Hom_{\underline{\mathfrak{B}}}(A,B)Hom_{\underline{\mathfrak{B}}}(A,A)}  \arrow[rrr, "m"] &  & {}\arrow[ddd, Rightarrow,"\tau^{\underline{\mathfrak{B}}}\ "', shorten > = 2ex, shorten < = 1ex] & {Hom_{\underline{\mathfrak{B}}}(A,B)}\arrow[ddd, "r"]\\
&  &  &\\
& & &\\
{Hom_{\underline{\mathfrak{B}}}(A,B)\, I} \arrow[uuu, "1j"]\arrow[rrr,equal] & & {} & {Hom_{\underline{\mathfrak{B}}}(A,B)\, I};
\end{tikzcd}$$

\begin{enumerate}
    \item[6.] For every $A,B,C, D$ in $\underline{\mathfrak{B}}$, a 2-isomorphism $\pi^{\underline{\mathfrak{B}}}_{A,B,C,D}$ in $\mathfrak{C}$
\end{enumerate}

$$\begin{tikzcd}[column sep=tiny]
{{(Hom_{\underline{\mathfrak{B}}}(C,D)Hom_{\underline{\mathfrak{B}}}(B,C))Hom_{\underline{\mathfrak{B}}}(A,B)}} \arrow[d, "\alpha"'] \arrow[rrr, "m1"] &  &  & {{Hom_{\underline{\mathfrak{B}}}(C,D)Hom_{\underline{\mathfrak{B}}}(A,C)}} \arrow[dd, "m"] \\
{Hom_{\underline{\mathfrak{B}}}(C,D)(Hom_{\underline{\mathfrak{B}}}(B,C)Hom_{\underline{\mathfrak{B}}}(A,B))} \arrow[d, "m"'] \arrow[rrr, Rightarrow, "\pi^{\underline{\mathfrak{B}}}", shorten <= 6ex, shorten >= 6ex]              &  &  & {}                                                                                         \\
{Hom_{\underline{\mathfrak{B}}}(B,D) Hom_{\underline{\mathfrak{B}}}(A,B)} \arrow[rrr, "m"']                                                            &  &  & {Hom_{\underline{\mathfrak{B}}}(A,D);}                                                    
\end{tikzcd}$$

Subject to the following relations:
\begin{enumerate}
    \item [a.] For every $A,B,C,D,E$ in $\underline{\mathfrak{B}}$, we have
\end{enumerate}

\newlength{\enrichment}
\settoheight{\enrichment}{\includegraphics[width=52.5mm]{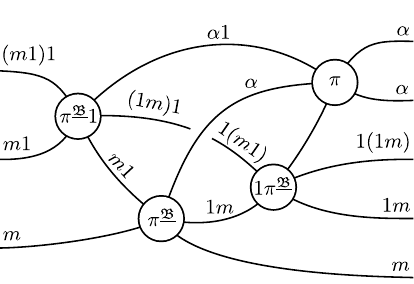}}

\begin{center}
\begin{tabular}{@{}ccc@{}}

\includegraphics[width=52.5mm]{Module2CategoriesPictures/enrichmentcat/enrichedassociator1.pdf} & \raisebox{0.45\enrichment}{$=$} &

\includegraphics[width=46.5mm]{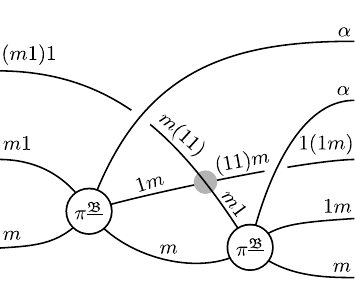}

\end{tabular}
\end{center}

\begin{enumerate}
    \item[] in $Hom_{\mathfrak{C}}(((\underline{\mathfrak{B}}(D,E)\Box\underline{\mathfrak{B}}(C,D))\Box\underline{\mathfrak{B}}(B,C))\Box\underline{\mathfrak{B}}(A,B),\underline{\mathfrak{B}}(A,E))$,

    \item [b.] For every every $A,B,C$ in $\underline{\mathfrak{B}}$, we have
    
\settoheight{\enrichment}{\includegraphics[width=45mm]{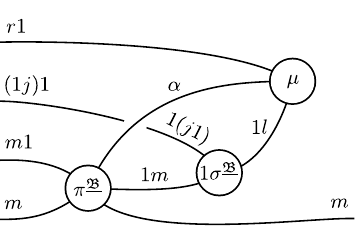}}

\begin{center}
\begin{tabular}{@{}ccc@{}}

\includegraphics[width=45mm]{Module2CategoriesPictures/enrichmentcat/enrichedcatunit1.pdf} & \raisebox{0.45\enrichment}{$=$} &

\includegraphics[width=30mm]{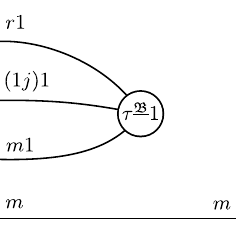}

\end{tabular}
\end{center}

in $Hom_{\mathfrak{C}}(\underline{\mathfrak{B}}(B,C)\Box\underline{\mathfrak{B}}(A,B), \underline{\mathfrak{B}}(A,C))$.
\end{enumerate}
\end{Definition}

\begin{Theorem}\label{thm:enrichmentmodule}
The 2-category $\mathfrak{M}$ is canonically enriched over $\mathfrak{C}$.
\end{Theorem}
\begin{proof}
Without loss of generality, we may assume that the pair $(\mathfrak{C},\mathfrak{M})$ is strict cubical. In order to avoid confusion, we will denote the $\mathfrak{C}$-enriched 2-category we will construct by $\underline{\mathfrak{M}}$. The object of $\underline{\mathfrak{M}}$ are taken to be the objects of $\mathfrak{M}$. Given objects $M,N$ of $\mathfrak{M}$, we define $$Hom_{\underline{\mathfrak{M}}}(M,N):=\underline{Hom}(M,N).$$ For every $M$ in $\mathfrak{M}$, we set $j_M:=u_{M,I}$ in $\mathfrak{C}$. For every $M,N,P$ in $\mathfrak{M}$, we let $m_{M,N,P}$ be given by the following composite 1.morphism in $\mathfrak{C}$: $$\adjustbox{scale=0.93}{$m_{M,N,P}:= \underline{Hom}(M,c_{P,N})\circ \underline{Hom}(M,\underline{Hom}(N,P)\Box c_{M,N})\circ u_{M,\underline{Hom}(N,P)\Box \underline{Hom}(M,N)}$}.$$

Using the invertible modifications $\Phi$ and $\Psi$ of remark \ref{rem:biadjunction}, we can define the two invertible modifications $\sigma$ and $\tau$. More precisely, for every $M,N$ in $\mathfrak{M}$, we let $\sigma^{\underline{\mathfrak{M}}}_{M,N}$ and $\tau^{\underline{\mathfrak{M}}}_{M,N}$ be the 2-isomorphisms in $Hom_{\mathfrak{C}}(\underline{Hom}(M,N), \underline{Hom}(M,N))$ given by:

\settoheight{\enrichment}{\includegraphics[width=52.5mm]{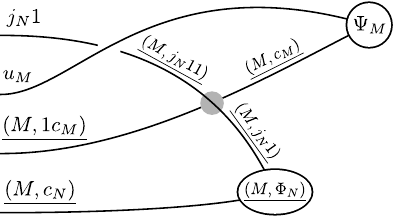}}

\begin{center}
\begin{tabular}{@{}cc@{}}

\raisebox{0.4\enrichment}{$\sigma^{\underline{\mathfrak{M}}}_{M,N}:=$} &

\includegraphics[width=46.5mm]{Module2CategoriesPictures/enrichmentcat/unitor2.pdf},

\end{tabular}
\end{center}

\begin{center}
\begin{tabular}{@{}cc@{}}

\raisebox{0.45\enrichment}{$\tau^{\underline{\mathfrak{M}}}_{M,N}:=$} &

\includegraphics[width=46.5mm]{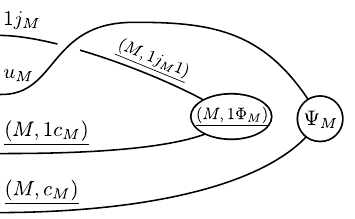}.

\end{tabular}
\end{center}

For every $M,N,P,Q$ in $\mathfrak{M}$, the invertible modification $\pi^{\underline{\mathfrak{M}}}_{M,N,P,Q}$ is the 2-isomorphism in $Hom_{\mathfrak{C}}(\underline{Hom}(P,Q)\Box \underline{Hom}(N,P)\Box \underline{Hom}(M,N),\underline{Hom}(M,Q))$ given by the diagram depicted in figure \ref{fig:enrichedpentagonator}.

Finally, we have to check that the two axioms of definition \ref{def:enriched2cat} are satisfied. In order to do this, we use the string diagrams depicted in appendix \ref{sec:diagcoherenceenriched}. We begin by checking axiom b. Given $M,N,P$ in $\mathfrak{M}$, figure \ref{fig:enrichedcatunitality1} depicts the left hand-side of this equation between 2-isomorphisms in $$Hom_{\mathfrak{C}}(\underline{Hom}(N,P)\Box \underline{Hom}(M,N), \underline{Hom}(M,P)).$$

\begin{landscape}
    \vspace*{\fill}
    \begin{figure}[!hbt]
    \centering
    \includegraphics[width=150mm]{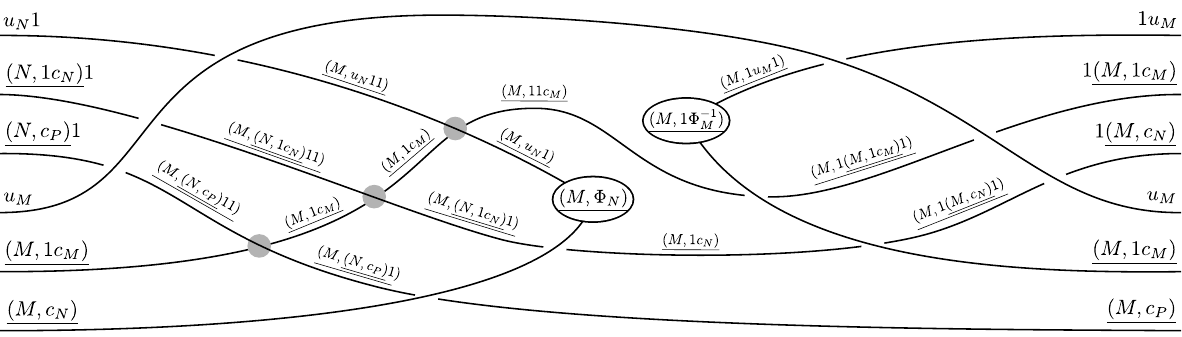}
    \caption{The enriched pentagonator}
    \label{fig:enrichedpentagonator}
    \end{figure}
    \vfill
\end{landscape}

\noindent As a first step, we use naturality to move the string labelled $u_M$ above all the coupons, which gets us to figure \ref{fig:enrichedcatunitality2}. Moving the coupon labelled $\underline{(M,1\underline{(M,\Phi_N)}1)}$ to the left, we obtain figure \ref{fig:enrichedcatunitality3}. Figure \ref{fig:enrichedcatunitality4} is obtained by moving the coupon labelled $\underline{(M,1\Phi_N)}$ to the left, and using equation (\ref{eqn:swallowtail1}) combined with lemma \ref{lem:biadj2functor} on the green coupons. Applying equation (\ref{eqn:swallowtail1}) with lemma \ref{lem:biadj2functor} to the blue coupon, we get figure \ref{fig:enrichedcatunitality5}. Finally, we arrive at figure \ref{fig:enrichedcatunitality6} by moving the strings labelled $u_M$ and $\underline{(M,1c_M)}$ down. As this last figure depicts the right hand-side of the axiom b, we find that axiom b holds.

Let us now check that axiom a of definition \ref{def:enriched2cat} holds. Given objects $M,N,P,Q,R$ in $\mathfrak{M}$, figure \ref{fig:enrichedcatassociativity1} depicts the left hand-side of this equation between 2-isomorphisms in $$Hom_{\mathfrak{C}}(\underline{Hom}(Q,R)\Box\underline{Hom}(P,Q)\Box\underline{Hom}(N,P)\Box \underline{Hom}(M,N), \underline{Hom}(M,R)).$$ Using naturality, we can move the string labelled $u_M$ on top of all the coupons, which gets us to figure \ref{fig:enrichedcatassociativity2}. To arrive at figure \ref{fig:enrichedcatassociativity3}, we move the coupons labelled $\underline{(M,\underline{(M,\Phi_N)}1)}$ and $\underline{(M,\underline{(M,1\Phi_M^{-1})}1)}$, and the string inbetween upwards. By naturality, we are allowed to bring the coupon labelled $\underline{(M,\underline{(N,1\Phi_N^{-1})11})}$ down, getting us to figure \ref{fig:enrichedcatassociativity4}. Cancelling the two blue coupons, we find ourselves contemplating figure \ref{fig:enrichedcatassociativity5}. Figure \ref{fig:enrichedcatassociativity6} is derived by moving the coupon labelled $\underline{(M,\underline{(N,\Phi_P)11})}$ to the right, and the one labelled $\underline{(M,1\underline{(M,1\Phi_M^{-1})11})}$ to the left. We can then use naturality to bring the string labelled $u_P11$ up, and the coupon labelled $\underline{(M,1\Phi_M^{-1})}$ down, arriving at figure \ref{fig:enrichedcatassociativity7}. Then, we proceed to bring the coupon labelled $\underline{(M,\Phi_N)}$ together with the string immediately underneath to the bottom left of the diagram, which gets us to figure \ref{fig:enrichedcatassociativity8}. Figure \ref{fig:enrichedcatassociativity9} is a mere isotopy of the previous one. Finally, figure \ref{fig:enrichedcatassociativity10} is obtained by first bringing the coupon labelled $\underline{(M,\Phi_M^{-1})}$ left and down, and then moving the string labelled $u_M$ down. This last figure is exactly the right hand-side of the equation we wished to check, so we are done.
\end{proof}

\begin{Remark}
The linear monoidal 2-functor $\mathbf{2Vect}\rightarrow \mathfrak{C}$ induces a 3-functor from the 3-category of 2-categories enriched over $\mathfrak{C}$ to the 3-category of 2-categories enriched over $\mathbf{2Vect}$. Applying this 3-functor to $\underline{\mathfrak{M}}$ gives back $\mathfrak{M}$ up to equivalence.
\end{Remark}

\begin{Corollary}\label{cor:endalgebra}
Let $M$ be any object of $\mathfrak{M}$. The object $\underline{End}(M)$ of $\mathfrak{C}$ has a canonical algebra structure.
\end{Corollary}
\begin{proof}
For simplicity, we assume that the pair $(\mathfrak{C},\mathfrak{M})$ is strict cubical. Using the notations of the proof of theorem \ref{thm:enrichmentmodule}, we take $$i:=j_M:I\rightarrow \underline{End}(M),\ \ m:= m_{M,M,M}:\underline{End}(M)\Box\underline{End}(M)\rightarrow \underline{End}(M),$$ $\mu:=\alpha^{-1}_{M,M,M,M}$, $\lambda:=\sigma^{-1}_{M,M}$, and $\rho:=\tau_{M,M}$. The coherence axioms of definition \ref{def:algebra} follow readily from the coherence axioms of an enriched 2-category.
\end{proof}

The following lemma will play a key role in the proof of our main theorem below.

\begin{Lemma}\label{lem:pre2ostrik}
Let $M$ be an object of $\mathfrak{M}$. Then, $\underline{Hom}(M,-)$ defines a linear 2-functor $$\mathfrak{M}\rightarrow \mathbf{Mod}_{\mathfrak{C}}(\underline{End}(M)).$$ Further, this 2-functor is a left $\mathfrak{C}$-module 2-functor.
\end{Lemma}
\begin{proof}
The first part is a consequence of theorem \ref{thm:enrichmentmodule}. Now, recall from proposition \ref{prop:catmodmodcat} that $\mathbf{Mod}_{\mathfrak{C}}(\underline{End}(M))$ has a canonical $\mathfrak{C}$-module structure. The second part of the statement essentially follows from proposition \ref{prop:enrichedhommod}. To be precise, we have to upgrade the 1-morphism $\chi^M_{C,N}$ constructed in the proof of lemma \ref{lem:enrichmentidentification} to a right $\underline{End}(M)$-module 1-morphism, and show that the 2-isomorphisms $\omega^M$ and $\gamma^M$ constructed in the proof of proposition \ref{prop:enrichedhommod} are $\underline{End}(M)$-module 2-morphisms. For simplicity, let us assume that the pair $(\mathfrak{C},\mathfrak{M})$ is strict cubical. The 2-isomorphism $\psi^{\chi^M_{C,N}}$ providing $\chi^M_{C,N}$ with the desired module structure is given by

\newlength{\mainthm}

\settoheight{\mainthm}{\includegraphics[width=90mm]{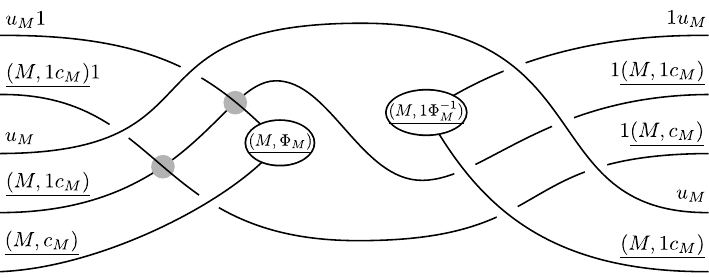}}

\begin{center}
\begin{tabular}{@{}cc@{}}

\raisebox{0.45\mainthm}{$\psi^{\chi^M_{C,N}}:=$} &
\includegraphics[width=90mm]{Module2CategoriesPictures/enrichmentexpanded/psichi.pdf}.
\end{tabular}
\end{center}

Now, the string diagram defining $\psi^{\chi^M_{C,N}}$, and that used to define $\pi^{\underline{\mathfrak{M}}}$ in figure \ref{fig:enrichedpentagonator} are extremely close to one another. It should therefore come as not surprise that the coherence axioms that have to be satisfied by $\psi^{\chi^M_{C,N}}$ (see definition \ref{def:modulemap}) follow using essentially the same arguments as those used in the proof of theorem \ref{thm:enrichmentmodule} (and depicted in appendix \ref{sec:diagcoherenceenriched}). Finally, proving that $\omega^M$ and $\gamma^M$ are compatible with this right module structure is easy and left to the reader.
\end{proof}

\begin{Remark}\label{rem:EnrichmentModuleGeneral}
The assumptions that $\mathfrak{C}$ be a multifusion 2-category and $\mathfrak{M}$ be finite semisimple 2-category are never explicitly used in the proof of theorem \ref{thm:enrichmentmodule}. Rather, they manifest themselves in the form of proposition \ref{prop:adjointmodule}. Therefore, for a completely general monoidal 2-category $\mathfrak{C}$ and module 2-category $\mathfrak{M}$ over $\mathfrak{C}$, the statement of theorem \ref{thm:enrichmentmodule} holds provided that proposition \ref{prop:adjointmodule} can be established for $\mathfrak{M}$. In particular, if $\mathds{k}$ is an arbitrary perfect field, $\mathfrak{C}$ is a locally separable compact semisimple tensor 2-category, and $\mathfrak{M}$ is a compact semisimple left $\mathfrak{C}$-module 2-category, then theorem \ref{thm:enrichmentmodule} holds thanks to remark \ref{rem:EnrichmentGeneral}.
\end{Remark}

\section{Comparing Finite Semisimple Module 2-Cate\-gories \& 2-Categories of Modules}\label{sec:2Ostrik}

\subsection{Bimodules associated to Module 2-Functors}\label{sec:2Ostrikfunctor}

Unless explicitly stated otherwise, we will work with $\mathfrak{C}$ a multifusion 2-category over an algebraically closed field $\mathds{k}$ of characteristic zero. Let us also fix $\mathfrak{M}$ and $\mathfrak{N}$ two finite semisimple left $\mathfrak{C}$-module 2-categories over $\mathds{k}$. Without loss of generality, we may simultaneously assume that the pairs $(\mathfrak{C},\mathfrak{M})$ and $(\mathfrak{C},\mathfrak{N})$ are strict cubical, by appealing to remark \ref{rem:coherence}. We have seen in corollary \ref{cor:endalgebra} that, given any object $M$ of $\mathfrak{M}$, and $N$ of $\mathfrak{N}$, then $\underline{End}_{\mathfrak{M}}(M)$ and $\underline{End}_{\mathfrak{N}}(N)$ are algebras in $\mathfrak{C}$. Given $F:\mathfrak{M}\rightarrow\mathfrak{N}$ a left $\mathfrak{C}$-module 2-functor, we can consider the right $\underline{End}_{\mathfrak{N}}(N)$-module $\underline{Hom}_{\mathfrak{N}}(N,F(M))$. We show that this object carries a compatible left $\underline{End}_{\mathfrak{M}}(M)$-module, so that it is in fact a bimodule. In order to achieve this, we use the notion of a $\mathfrak{C}$-enriched 2-functor (definition 3.5 of \cite{GS}), and the following proposition.

\begin{Proposition}\label{prop:enrichedmodule2functor}
The left $\mathfrak{C}$-module 2-functor $F$ can be canonically upgraded to a $\mathfrak{C}$-enriched 2-functor $\underline{F}:\underline{\mathfrak{M}}\rightarrow \underline{\mathfrak{N}}$.
\end{Proposition}
\begin{proof}
For any two objects $P, Q$ of $\mathfrak{M}$, we let $$\underline{F}_{PQ}:\underline{Hom}_{\mathfrak{M}}(P,Q)\rightarrow \underline{Hom}_{\mathfrak{N}}(F(P),F(Q))$$ be the 1-morphism in $\mathfrak{C}$ given by the composite \begin{align*}\underline{Hom}_{\mathfrak{M}}(P,Q)&\xrightarrow{u_{F(P)}}\underline{Hom}_{\mathfrak{N}}(F(P),\underline{Hom}_{\mathfrak{M}}(P,Q)\Diamond^{\mathfrak{N}}F(P))\\ &\xrightarrow{\underline{Hom}_{\mathfrak{N}}(F(P),\chi^F)}\underline{Hom}_{\mathfrak{N}}(F(P),F(\underline{Hom}_{\mathfrak{M}}(P,Q)\Diamond^{\mathfrak{M}}P))\\ &\xrightarrow{\underline{Hom}_{\mathfrak{N}}(F(P),F(c_P))}\underline{Hom}_{\mathfrak{N}}(F(P),F(Q)).\end{align*}

Given objects $P,Q,R$ of $\mathfrak{M}$, the 2-isomorphism $\mu_{PQR}$ in $\mathfrak{C}$ is depicted in figure \ref{fig:muPQR}, and given $P$ in $\mathfrak{M}$, the 2-isomorphism $\iota_P$ in $\mathfrak{C}$ is given by

\newlength{\enrichexpanded}
\settoheight{\enrichexpanded}{\includegraphics[width=60mm]{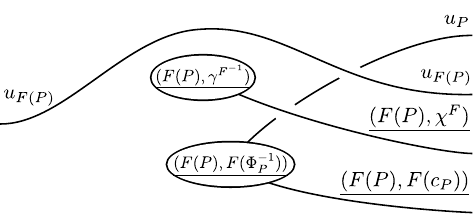}}

\begin{center}
\begin{tabular}{@{}cc@{}}

\raisebox{0.45\enrichexpanded}{$\iota_P =$} &

\includegraphics[width=60mm]{Module2CategoriesPictures/enrichmentexpanded/iotaP.pdf}.

\end{tabular}
\end{center}

Checking the coherence axioms for $\mu_{P,Q,R}$, and $\iota_P$ can be done using arguments similar to the ones we have already given (together with the axioms of definition \ref{def:module2nat}). We leave the details to the reader.
\end{proof}

\begin{landscape}
    \vspace*{\fill}
    \begin{figure}[!hbt]
    \centering
    \includegraphics[width=180mm]{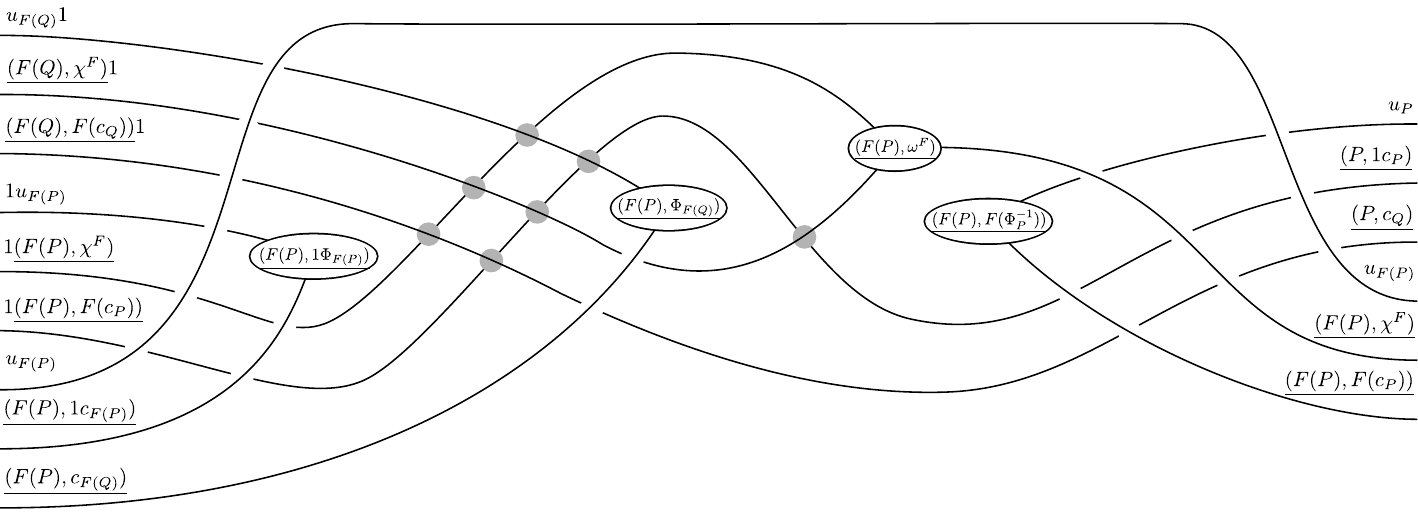}
    \caption{The coherence 2-isomorphism $\mu_{P,Q,R}$}
    \label{fig:muPQR}
    \end{figure}
    \vfill
\end{landscape}

\begin{Corollary}\label{cor:bimodule}
The object $\underline{Hom}_{\mathfrak{N}}(N,F(M))$ is an $(\underline{End}_{\mathfrak{M}}(M), \underline{End}_{\mathfrak{N}}(N))$-bimodule in $\mathfrak{C}$.
\end{Corollary}
\begin{proof}
By theorem \ref{thm:enrichmentmodule}, it follows immediately that $\underline{Hom}_{\mathfrak{N}}(N,F(M))$ has an $(\underline{End}_{\mathfrak{N}}(F(M)), \underline{End}_{\mathfrak{N}}(N))$-bimodule structure given by composition. Now, proposition \ref{prop:enrichedmodule2functor} implies that there is a canonical 1-morphism $$\underline{End}_{\mathfrak{M}}(M)\rightarrow\underline{End}_{\mathfrak{N}}(F(M))$$ of algebras in $\mathfrak{C}$. The bimodule version of lemma \ref{lem:1morphismalgebras2functor} proves the claim.
\end{proof}

In order to prove our next theorem, we will need a stronger result.

\begin{Proposition}\label{prop:enrichedbimodules}
The assignment $$\begin{tabular}{r c c c}$\mathbf{B}:$& $Fun_{\mathfrak{C}}(\mathfrak{M},\mathfrak{N})$ & $\rightarrow$ & $\mathbf{Bimod}_{\mathfrak{C}}(\underline{End}_{\mathfrak{M}}(M), \underline{End}_{\mathfrak{N}}(N)),$\\[0.2cm] & $F$ & $\mapsto$ & $\underline{Hom}_{\mathfrak{N}}(N,F(M))$\end{tabular}$$ defines a linear 2-functor from the 2-category of left $\mathfrak{C}$-module 2-functors from $\mathfrak{M}$ to $\mathfrak{N}$ to the 2-category of $(\underline{End}_{\mathfrak{M}}(M), \underline{End}_{\mathfrak{N}}(N))$-bimodules in $\mathfrak{C}$.
\end{Proposition}
\begin{proof}
Given a left $\mathfrak{C}$-module 2-functor $F:\mathfrak{M}\rightarrow \mathfrak{N}$, following corollary \ref{cor:bimodule}, we set $$\mathbf{B}(F):=\underline{Hom}_{\mathfrak{N}}(N,F(M))$$ as an $(\underline{End}_{\mathfrak{M}}(M), \underline{End}_{\mathfrak{N}}(N))$-bimodule.

Now, let $\theta:F\Rightarrow G$ be a left $\mathfrak{C}$-module 2-natural transformations between two left $\mathfrak{C}$-module 2-functors $F,G:\mathfrak{M}\rightarrow \mathfrak{N}$. We claim that the 1-morphism $\underline{Hom}_{\mathfrak{N}}(N,\theta_M)$ in $\mathfrak{C}$ can be upgraded to an $(\underline{End}_{\mathfrak{M}}(M), \underline{End}_{\mathfrak{N}}(N))$-bimodule 1-morphism. The right $\underline{End}_{\mathfrak{N}}(N)$-module structure is simply given by naturality, whereas the left $\underline{End}_{\mathfrak{M}}(M)$ is given by the 2-isomorphism $\omega^{\underline{(N,\theta)}}$ depicted in figure \ref{fig:omega(N,theta)}.

\begin{figure}[!hbt]
    \centering
    \includegraphics[width=120mm]{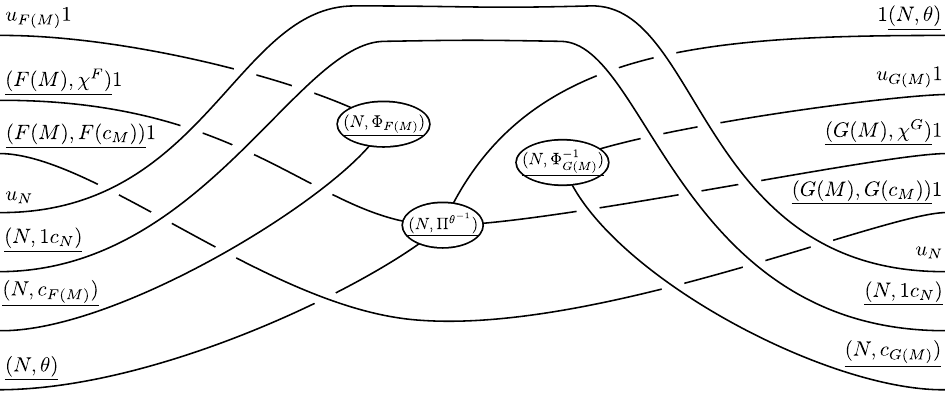}
    \caption{The coherence 2-isomorphism $\omega^{\underline{(N,\theta)}}$}
    \label{fig:omega(N,theta)}
\end{figure}

The required compatibility conditions follow readily from the axioms of definition \ref{def:module2nat}. This allows us to define $$\mathbf{B}(\theta):= \underline{Hom}_{\mathfrak{N}}(N,\theta_M)$$ as a $(\underline{End}_{\mathfrak{M}}(M), \underline{End}_{\mathfrak{N}}(N))$-bimodule 1-morphism.

Given a left $\mathfrak{C}$-module modification $\Xi:\sigma\Rrightarrow\theta$ between two left left $\mathfrak{C}$-module 2-natural modifications $\theta,\sigma:F\Rightarrow G$, it is easy to check that $\underline{Hom}_{\mathfrak{N}}(N,\Xi_M)$ defines a $(\underline{End}_{\mathfrak{M}}(M), \underline{End}_{\mathfrak{N}}(N))$-bimodule 2-morphism, which allows us to set $$\mathbf{B}(\Xi):= \underline{Hom}_{\mathfrak{N}}(N,\Xi_M).$$ Observe that this assignment is  functorial.

It remains to check that $\mathbf{B}$ is a 2-functor. Let $Id_F^{\mathfrak{C}}$ be the identity left $\mathfrak{C}$-module 2-natural transformation constructed in proposition \ref{prop:hom2cat}. Note that thanks to our strictness hypothesis, $\Pi^{Id_F^{\mathfrak{C}}} = Id_{\chi^F}$. Thus, by inspecting figure \ref{fig:omega(N,theta)}, we find that we can set $\phi^{\mathbf{B}}_F:=\phi^{\underline{Hom}(N,-)}_{F(M)}$, as the right hand-side is an invertible $(\underline{End}_{\mathfrak{M}}(M), \underline{End}_{\mathfrak{N}}(N))$-bimodule 2-morphism. Furthermore, given $\theta:F\Rightarrow G$, and $\xi:G\Rightarrow H$ two left $\mathfrak{C}$-module 2-natural transformations, we define $\phi^{\mathbf{B}}_{\xi,\theta}=\phi^{\underline{Hom}(N,-)}_{\xi_M,\theta_M}$. Examination shows that the right hand-side is an invertible $(\underline{End}_{\mathfrak{M}}(M), \underline{End}_{\mathfrak{N}}(N))$-bimodule 2-morphism, justifying the definition.

Finally, the coherence axioms one has to check follow immediately from the coherence axioms for $\underline{Hom}(N,-)$.
\end{proof}

\begin{Remark}\label{rem:Morita3Cat}
Left $\mathfrak{C}$-module 2-functors can be composed in an obvious way. We expect that there is corresponding ``tensor product'' operation on the bimodule side. In particular, if $\mathfrak{M}=\mathfrak{N}$, and $M=N$, then the equivalence of proposition \ref{prop:enrichedbimodules} should send composites of $\mathfrak{C}$-module 2-functors to ``tensor product'' of bimodules.
\end{Remark}

\begin{Remark}\label{rem:BimoduleGeneral}
Thanks remark \ref{rem:2OstrikGeneral}, all the results of this section and their proofs remain valid under the more general assumptions that $\mathds{k}$ is a perfect field, $\mathfrak{C}$ is a locally separable compact semisimple tensor 2-category, and $\mathfrak{M}$, $\mathfrak{N}$ are compact semisimple left $\mathfrak{C}$-module 2-categories.
\end{Remark}

\subsection{Rigid Algebras}\label{sec:rigidseparable}

We begin by recalling two definitions from \cite{JFR}.

\begin{Definition}\label{def:rigidseparable}
An algebra $A$ in $\mathfrak{C}$ is called rigid provided that the multiplication map $m:A\Box A\rightarrow A$ has a right adjoint $m^*$ as a 1-morphism of $(A,A)$-bimodules. 
\end{Definition}

\begin{Definition}
An algebra $A$ is called separable if it is rigid, and the counit witnessing the adjunction of $(A,A)$-bimodule 1-morphisms between $m$ and $m^*$ admits a section as an $(A,A)$-bimodule 2-morphism.
\end{Definition}

\begin{Example}\label{ex:rigidalgebra2Vect}
Expanding on example \ref{ex:algebra2vect}, proposition 1.3 of \cite{BJS} shows that rigid algebras in $\mathbf{2Vect}$ are precisely multifusion categories. Further, as it was proven in \cite{DSPS14} that every multifusion category is separable, separable algebras in $\mathbf{2Vect}$ also correspond exactly to multifusion categories.
\end{Example}

\begin{Example}\label{ex:rigidalgebra2VectG}
By example \ref{ex:algebra2vectG}, we know that algebra in $\mathbf{2Vect}_G$ correspond exactly to $G$-graded monoidal finite semisimple categories. Now, observe that the right adjoint to a $G$-graded functor between $G$-graded finite semisimple categories has to be $G$-graded. Using this observation together with the above example, we find that rigid algebras in $\mathbf{2Vect}_G$ are exactly given by $G$-graded multifusion categories. A similar observation holds for the section of a $G$-graded natural transformation, which implies that separable algebras in $\mathbf{2Vect}_G$ are precisely $G$-graded multifusion categories.
\end{Example}

\begin{Example}\label{ex:rigidalgebrabraidedfusion}
We have recalled in example \ref{ex:algebrasbraidedfusion} that algebras in $\mathbf{Mod}(\mathcal{B})$ are precisely monoidal finite semisimple categories $\mathcal{C}$ equipped with a central monoidal functor $\mathcal{B}\rightarrow \mathcal{C}$. A slight adaptation of lemma 2.21 of \cite{JFR} proves that the algebra associated to $\mathcal{C}$ is rigid if and only if $\mathcal{C}$ is multifusion. Further, one can check that every rigid algebra in $\mathbf{Mod}(\mathcal{B})$ is separable, whence, separable algebras in $\mathbf{Mod}(\mathcal{B})$ are also given by multifusion categories equipped with a central monoidal functor from $\mathcal{B}$.
\end{Example}

\begin{Remark}\label{rem:rigidi=separable}
The above examples provide motivation for the conjecture made in \cite{JFR} that every rigid algebra in a multifusion 2-category is separable. We wish to remark that this conjecture is stronger than the assertion that every multifusion 2-category is separable, in the sense that their Drinfeld centers are finite semisimple 2-categories. Namely, combining this conjecture of \cite{JFR} with theorem \ref{thm:rigid}, one can prove that every multifusion 2-category is separable (see \cite{D6}).
\end{Remark}

We now prove that the algebras produced by corollary \ref{cor:endalgebra} are always rigid.

\begin{Theorem}\label{thm:rigid}
Given $\mathfrak{M}$ a finite semisimple left $\mathfrak{C}$-module 2-category, and $M$ an object of $\mathfrak{M}$, the algebra $\underline{End}(M)$ is rigid.
\end{Theorem}

Before giving the proof, we establish a technical lemma.

\begin{Lemma}\label{lem:adjointmodule}
Let $\mathfrak{M}$, and $\mathfrak{N}$ be two finite semisimple left $\mathfrak{C}$-module 2-categories, $F,G:\mathfrak{M}\rightarrow \mathfrak{N}$ two left $\mathfrak{C}$-module 2-functors, and $\theta:F\Rightarrow G$ a left $\mathfrak{C}$-module 2-natural transformation. Then, $\theta$ has a right adjoint as a left $\mathfrak{C}$-module 2-natural transformation.
\end{Lemma}
\begin{proof}
Firstly, observe that $\theta$ has a right adjoint $\theta^*$ as a 2-natural transformation. This follows from the fact that $Hom(\mathfrak{M},\mathfrak{N})$ is a finite semisimple 2-category (see \cite{D1}). Concretely, on the object $M$ of $\mathfrak{M}$, we have $\theta^*_M := (\theta_M)^*$, and on the 1-morphism $f:M\rightarrow N$ of $\mathfrak{M}$, we have

\settoheight{\mainthm}{\includegraphics[width=30mm]{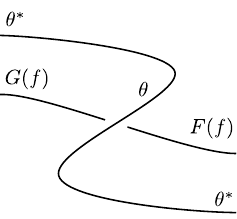}}

\begin{center}
\begin{tabular}{@{}cc@{}}

\raisebox{0.45\mainthm}{$\theta^*_f=$} &
\includegraphics[width=30mm]{Module2CategoriesPictures/rigidseparable/thetastarf.pdf}
\end{tabular}
\end{center}

\noindent in $Hom_{\mathfrak{N}}(\theta^*_N\circ G(f), F(f)\circ \theta^*_M)$. Secondly, given $C$ in $\mathfrak{C}$, and $M$ in $\mathfrak{M}$, we set

\settoheight{\mainthm}{\includegraphics[width=37.5mm]{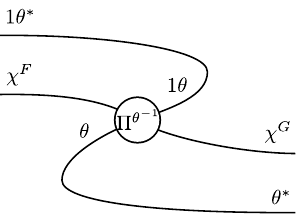}}

\begin{center}
\begin{tabular}{@{}cc@{}}

\raisebox{0.45\mainthm}{$\Pi^{\theta^*}_{C,M}:=$} &
\includegraphics[width=37.5mm]{Module2CategoriesPictures/rigidseparable/Pithetastar.pdf}
\end{tabular}
\end{center}

\noindent in $Hom_{\mathfrak{N}}(\theta^*_{C\Diamond^{\mathfrak{M}} M}\circ \chi^F_{C,M}, \chi^G_{C,M}\circ (C\Diamond^{\mathfrak{N}}\theta^*_M))$.

It is easy to see that $\Pi^{\theta^*}$ is a modification. It remains to argue that $\Pi^{\theta^*}$ is invertible. To this end, observe that $(\Pi^{\theta^{-1}})^*$ is invertible because $\Pi^{\theta}$ is. Further, as $\chi^F$ and $\chi^G$ are adjoint 2-natural equivalences, the units and counits of the corresponding adjunctions are 2-isomorphisms. Now, using the snake equations, we see that $\Pi^{\theta^*}$ is also represented by the following string diagram:

\settoheight{\mainthm}{\includegraphics[width=67.5mm]{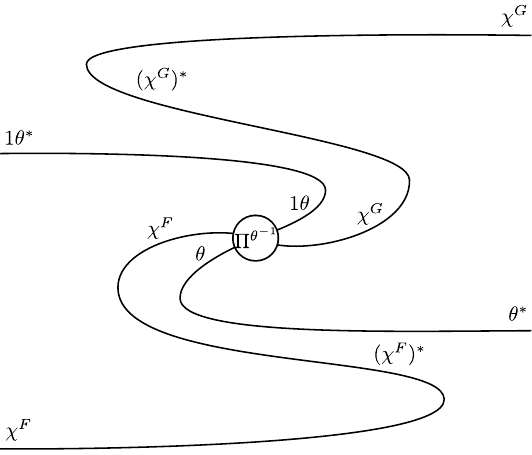}}

\begin{center}
\begin{tabular}{@{}cc@{}}

\raisebox{0.45\mainthm}{$\Pi^{\theta^*}_{C,M}=$} &
\includegraphics[width=67.5mm]{Module2CategoriesPictures/rigidseparable/Pithetastarexpanded.pdf}.
\end{tabular}
\end{center}

This alternative description makes it evident that $\Pi^{\theta^*}$ is a composite of $(\Pi^{\theta^{-1}})^*$, $\epsilon_{\chi^F}$, and $\eta_{\chi^G}$, and so is indeed invertible. Finally, it is not hard to check that $\Pi^{\theta^*}$ satisfy the axioms of definition \ref{def:module2nat}, which finishes the proof of the lemma.
\end{proof}

\begin{proof}[Proof of thm.\ \ref{thm:rigid}.]
Without loss of generality, we may assume that the pair $(\mathfrak{C},\mathfrak{M})$ is strict cubical. We note that the 2-functor $F:\mathfrak{M}\rightarrow \mathfrak{M}$ given by $N\mapsto \underline{Hom}(M,N)\Diamond M$ has a canonical left $\mathfrak{C}$-module structure provided by lemma \ref{lem:pre2ostrik}. Moreover, the 2-natural transformation $\mu:F\Rightarrow Id$ given on the object $N$ of $\mathfrak{M}$ by $\mu_N=c_{M,N}:\underline{Hom}(M,N)\Diamond M\rightarrow N$ also has a canonical left $\mathfrak{C}$-module structure. So, we can consider their images under the 2-functor $\mathbf{B}:Fun_{\mathfrak{C}}(\mathfrak{M},\mathfrak{M})\rightarrow \mathbf{Bimod}_{\mathfrak{C}}(\underline{End}(M),\underline{End}(M))$ of proposition \ref{prop:enrichedbimodules}. We find that $\mathbf{B}(Id)\simeq \underline{End}(M)$ as an $(\underline{End}(M),\underline{End}(M))$-bimodule. Further, we have that $\mathbf{B}(F)=\underline{Hom}(M,\underline{End}(M)\Diamond M)$. But, there is a canonical equivalence $$\chi^M_{\underline{End}(M),M}:\underline{End}(M)\Box \underline{End}(M)\simeq\underline{Hom}(M,\underline{End}(M)\Diamond M) ,$$ of right $\underline{End}(M)$-modules. Inspecting the construction given in corollary \ref{cor:bimodule}, we find that this equivalence is in fact compatible with the $(\underline{End}(M),\underline{End}(M))$-bimodule structures. Moreover, we have that $\mathbf{B}(\mu)$ corresponds under the above equivalence to the $(\underline{End}(M),\underline{End}(M))$-bimodule 1-morphism given by $m:\underline{End}(M)\Box \underline{End}(M)\rightarrow \underline{End}(M)$ . Finally, we can appeal to lemma \ref{lem:adjointmodule} to obtain a right adjoint $\mu^*:Id\Rightarrow F$ to $\mu$ as a left $\mathfrak{C}$-module 2-natural transformation. As right adjoints are preserved by 2-functors, we find that $\mathbf{B}(\mu^*)$ produces the desired right adjoint to $m\simeq \mathbf{B}(\mu)$ as an $(\underline{End}(M),\underline{End}(M))$-bimodule 1-morphism.
\end{proof}

Using the theorem above, we obtain a characterization of rigid algebras in multifusion 2-categories.

\begin{Theorem}\label{thm:characterizationrigid}
An algebra $A$ in a multifusion 2-category $\mathfrak{C}$ is rigid if and only if $\mathbf{Mod}_{\mathfrak{C}}(A)$ is a finite semisimple 2-category.
\end{Theorem}
\begin{proof}
Let us assume that $A$ is rigid. As the underlying 2-category of $\mathfrak{C}$ is finite semisimple, it follows from theorem 1.4.9 of \cite{DR} that there exists a multifusion category $\mathcal{C}$ and equivalence $\mathfrak{C}\simeq \mathbf{Mod}(\mathcal{C})$ of finite semisimple 2-categories. Using $\mathbf{Bimod}(\mathcal{C})$ to denote the monoidal 2-category of finite semisimple $\mathcal{C}$-$\mathcal{C}$-bimodule categories, it then follows from the main theorem of \cite{D1} that $End(\mathfrak{C})\simeq \mathbf{Bimod}(\mathcal{C})^{\boxtimes_{\mathcal{C}}op}$ as monoidal 2-categories.

Now, let us write $F:\mathfrak{C}^{\Box op}\rightarrow End(\mathfrak{C})$ for the canonical monoidal 2-functor $C\mapsto \{D\mapsto D\Box C\}$. As rigid algebras are preserved by all monoidal 2-functors, we find that $F(A)$ is a rigid algebra in $\mathbf{Bimod}(\mathcal{C})$. But, $\mathbf{Mod}(\mathcal{C})$ is canonically a right $\mathbf{Bimod}(\mathcal{C})$-module 2-category, so that there is an equivalence $$\mathbf{Mod}_{\mathfrak{C}}(A)\simeq \mathbf{Mod}_{\mathbf{Mod}(\mathcal{C})}(F(A))$$ between the 2-category of right $A$-modules in $\mathfrak{C}$, and the 2-category of right $F(A)$-modules in $\mathbf{Mod}(\mathcal{C})$.

Finally, it follows from lemma 2.23 of \cite{JFR} that the data of an algebra $B$ in $\mathbf{Bimod}(\mathcal{C})$ corresponds precisely to the data of a finite semisimple monoidal 1-category $\mathcal{D}$ together with a monoidal functor $\mathcal{C}\rightarrow \mathcal{D}$, and that $B$ is rigid if and only if $\mathcal{D}$ is multifusion. In particular, $F(A)$ corresponds to a mutlifusion category $\mathcal{A}$. Combining this fact with the observation that there is an equivalence $$\mathbf{Mod}_{\mathbf{Mod}(\mathcal{C})}(F(A))\simeq \mathbf{Mod}(\mathcal{A})$$ of 2-categories concludes the proof of this part of the statement thanks to theorem 1.4.8 of \cite{DR}.

Conversely, if $\mathbf{Mod}_{\mathfrak{C}}(A)$ is a finite semisimple 2-category, then $A$ is canonical a right $A$-module. Moreover, it follows from lemma \ref{lem:freeforgetadj} above that there is an equivalence $\underline{End}(A)\simeq A$ of algebras. But, $\underline{End}(A)$ is rigid thanks to theorem \ref{thm:rigid}, so that $A$ is rigid.
\end{proof}

\begin{Remark}\label{rem:rigidgeneral}
Let $\mathds{k}$ be a perfect field, $\mathfrak{C}$ a locally separable compact semisimple tensor 2-category over $\mathds{k}$, and $\mathfrak{M}$, $\mathfrak{N}$ be compact semisimple left $\mathfrak{C}$-module 2-categories. The first results of this subsection hold at this level of generality. More precisely, the proof of lemma \ref{lem:adjointmodule} needs to be modified as follows. By theorem 3.1.4 of \cite{D5}, in order to show that the 2-natural transformation $\theta$ admits a right adjoint, it is enough to prove the following statement: For any two finite semisimple tensor categories $\mathcal{C}$ and $\mathcal{D}$, every $\mathcal{C}$-$\mathcal{D}$-bimodule functor between two finite semisimple $\mathcal{C}$-$\mathcal{D}$-bimodule categories admits a right adjoint as a $\mathcal{C}$-$\mathcal{D}$-bimodule functor. This last claim follows from corollary 2.13 of \cite{DSPS14}. The remaining part of the proof of this lemma applies without change. The proofs of lemma \ref{lem:pre2ostrik} and theorem \ref{thm:rigid} remain valid thanks to remarks \ref{rem:EnrichmentModuleGeneral}, and \ref{rem:BimoduleGeneral}.

Thanks to remark 2.6.10 of \cite{DSPS13}, the proof of theorem \ref{thm:characterizationrigid} holds over any field of characteristic zero. On the other hand, theorem \ref{thm:characterizationrigid} is no longer valid in positive characteristic. For instance, assume that $\mathds{k}$ is an algebraically closed field of characteristic $p$, $\mathfrak{C}=\mathbf{2Vect}$, and take $A = \mathbf{Vect}_{\mathbb{Z}/p}$, a rigid algebra. Then, $\mathbf{Mod}_{\mathfrak{C}}(A)$ is the 2-category of all finite semisimple right $\mathbf{Vect}_{\mathbb{Z}/p}$-module 1-categories, which is not a semisimple 2-category. Nevertheless, let us record that the argument used in the proof of theorem \ref{thm:characterizationrigid} only breaks at the last step. More precisely, let us fix a rigid algebra $A$ in the locally separable compact semisimple tensor 2-category $\mathfrak{C}$. As in the above proof, we find that there exists a monoidal functor of finite semisimple tensor 1-categories $\mathcal{C}\rightarrow \mathcal{A}$. Further, the 2-category $\mathbf{Mod}_{\mathfrak{C}}(A)$ is equivalent to the 2-category of finite semisimple right $\mathcal{A}$-module 1-categories that are separable as right $\mathcal{C}$-module 1-categories in the sense of \cite{DSPS13}. But, the finite semisimple tensor 1-category $\mathcal{C}$ is separable given that $\mathfrak{C}$ is locally separable, so that every finite semisimple right $\mathcal{A}$-module 1-categories is separable as a $\mathcal{C}$-module 1-category thanks to proposition 2.5.10 of \cite{DSPS13}. Thence, we find that $\mathbf{Mod}_{\mathfrak{C}}(A)$ is equivalent to $\mathbf{Mod}(\mathcal{A})$, the 2-category of finite semisimple right $\mathcal{A}$-module 1-categories. We note that this is not a compact semisimple 2-category in general.
\end{Remark}

\subsection{Main Theorem}\label{sec:main}

We want to understand under what condition on the object $M$ the left $\mathfrak{C}$-module 2-functor $\underline{Hom}(M,-)$ considered in lemma \ref{lem:pre2ostrik} is an equivalence. The next definition provides a complete answer.

\begin{Definition}
An object $M$ of $\mathfrak{M}$ is called a $\mathfrak{C}$-generator if for every object $N$ of $\mathfrak{M}$, there exists $C$ in $\mathfrak{C}$ such that there is a non-zero 1-morphism $C\Diamond M \rightarrow N$.
\end{Definition}

\begin{Example}
Let $\mathfrak{N}$ be a finite semisimple 2-category. An object $N$ of $\mathfrak{N}$ is a $\mathbf{2Vect}$-generator if and only if $N$ has a non-zero summand in every connected component of $\mathfrak{N}$.
\end{Example}

\begin{Remark}
Let $\mathfrak{N}$ be a finite semisimple 2-category. Using the categorical Schur lemma, i.e.\ proposition 1.2.19 of \cite{DR}, one can show that an object $N$ of $\mathfrak{N}$ is $\mathbf{2Vect}$-generator if and only if $Hom_{\mathfrak{N}}(N,-):\mathfrak{N}\rightarrow \mathbf{2Vect}$ is faithful on 2-morphisms. More generally, we have that an object $M$ of $\mathfrak{M}$ is a $\mathfrak{C}$-generator if and only if the linear 2-functor $\underline{Hom}(M,-):\mathfrak{M}\rightarrow \mathfrak{C}$ is faithful on 2-morphisms.
\end{Remark}

\begin{Theorem}\label{thm:2Ostrik}
Let $M$ be a $\mathfrak{C}$-generator of $\mathfrak{M}$. Then, $$\underline{Hom}(M,-):\mathfrak{M}\rightarrow \mathbf{Mod}_{\mathfrak{C}}(\underline{End}(M))$$ is an equivalence of left $\mathfrak{C}$-module 2-categories.
\end{Theorem}
\begin{proof}
By lemma \ref{lem:pre2ostrik}, it is enough to show that the underlying 2-functor $\underline{Hom}(M,-)$ induces an equivalence of linear 2-categories. To this end, let us fix an object $N$ of $\mathfrak{M}$, and consider the diagram below:

\begin{equation}\label{eqn:2Ostrik}\adjustbox{scale=0.93}{\begin{tikzcd}[sep=4.1ex]
{Hom_{\mathfrak{M}}(C\Diamond M, N)} \arrow[dd, "t"'] \arrow[rr, "{\underline{Hom}(M,-)}"] &  & {Hom_{\underline{End}(M)}(\underline{Hom}(M,C\Diamond M), \underline{Hom}(M,N))} \arrow[dd, "{(\chi^M_{C,M})^*}"]\\
  &  &  \\
{Hom_{\mathfrak{C}}(C,\underline{Hom}(M,N))} \arrow[rr] &  & {Hom_{\underline{End}(M)}(C\Box\underline{End}(M), \underline{Hom}(M,N)).} \end{tikzcd}}\end{equation} The left vertical map is given by proposition \ref{prop:adjointmodule}, and the bottom map is supplied by lemma \ref{lem:freeforgetadj}. We note that these two maps are equivalences. Further, as the 1-morphism $\chi^M_{C,M}$ is an equivalence of right $\underline{End}(M)$-module by lemmas \ref{lem:enrichmentidentification} and \ref{lem:pre2ostrik}, the right vertical map is also an equivalence. Therefore, in order to show that the top horizontal map is an equivalence, it is enough to show that the diagram depicted above in (\ref{eqn:2Ostrik}) commutes up to natural isomorphism.

Let $f:C\Diamond M\rightarrow N$ be a 1-morphism in $\mathfrak{M}$. Its image under the top right composite in (\ref{eqn:2Ostrik}) is given by the right $\underline{End}(M)$-module 1-morphism $$C\underline{(M,M)}\xrightarrow{u_M}\underline{(M,C\underline{(M,M)}M)}\xrightarrow{\ \underline{(M,1c_M)}\ }\underline{(M,CM)}\xrightarrow{\ \underline{(M,f)}\ }\underline{(M,N)}.$$ On the other hand, the image of $f$ under the bottom left composite is the right $\underline{End}(M)$-module 1-morphism $$C\underline{(M,M)}\xrightarrow{u_M1}\underline{(M,CM)}\ \underline{(M,M)}\xrightarrow{\ \underline{(M,f)}1}\underline{(M,N)}\ \underline{(M,M)}$$ $$\xrightarrow{u_M}\underline{(M,\underline{(M,N)}\ \underline{(M,M)}M)}\xrightarrow{\ \underline{(M,1c_M)}\ }\underline{(M,\underline{(M,N)}M)}\xrightarrow{\ \underline{(M,c_M)}\ }\underline{(M,N)}.$$ The two right $\underline{End}(M)$-module 1-morphisms above are isomorphic via

\settoheight{\mainthm}{\includegraphics[width=90mm]{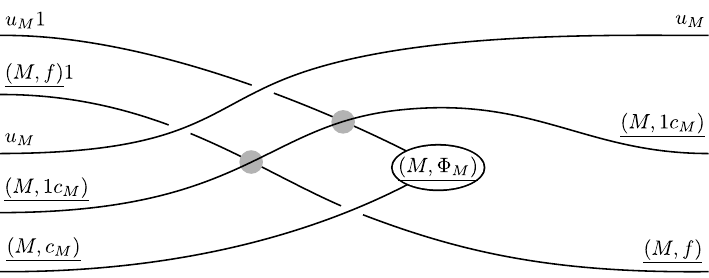}}

\begin{center}
\begin{tabular}{@{}cc@{}}
\raisebox{0.45\mainthm}{$\xi^f:=$} &
\includegraphics[width=90mm]{Module2CategoriesPictures/maintheorem/2Ostrikcommute.pdf}.
\end{tabular}
\end{center}

\noindent In addition, $\xi^f$ is compatible with the right $\underline{End}(M)$-module structures. This follows from a variant of the computation given in the proof of theorem \ref{thm:enrichmentmodule}. Furthermore, it follows from the definitions that $\xi^f$ is natural in $f$. This proves that (\ref{eqn:2Ostrik}) commutes up to natural isomorphism.

Let us write $\widetilde{\mathfrak{M}}$ for the sub-2-category of $\mathfrak{M}$ whose objects are $C\Diamond M$ for some $C$ in $\mathfrak{C}$ and that is full on 1-morphisms and 2-morphisms. As we have assumed that $M$ is a $\mathfrak{C}$-generator of $\mathfrak{M}$, we find that the Cauchy completion of the 2-category $\widetilde{\mathfrak{M}}$ is $\mathfrak{M}$. Namely, thanks to lemma A.2.5 of \cite{D1}, there is a linear 2-functor $Cau(\widetilde{\mathfrak{M}})\rightarrow \mathfrak{M}$ that is essentially surjective on 1-morphisms and fully faithful on 2-morphisms. Further, it follows from the first part of the proof of theorem 1.4.9 of \cite{DR} that this 2-functor is essentially surjective on objects.

We use $\widetilde{\mathbf{Mod}}_{\mathfrak{C}}(\underline{End}(M))$ to denote the sub-2-category of $\mathbf{Mod}_{\mathfrak{C}}(\underline{End}(M))$ whose objects are equivalent to $C\Diamond \underline{End}(M)$ for some $C$ in $\mathfrak{C}$ and that is full on 1-morphisms and 2-morphisms. But, $\underline{End}(M)$ is a $\mathfrak{C}$-generator of $\mathbf{Mod}_{\mathfrak{C}}(\underline{End}(M))$. Namely, for any right $\underline{End}(M)$-module $P$, the 1-morphism $n^P:P\Box \underline{End}(M)\rightarrow P$ has canonical $\underline{End}(M)$-module structure, and is non-zero if $P$ is non-zero. As above, this implies that $\mathbf{Mod}_{\mathfrak{C}}(\underline{End}(M))$ is the Cauchy completion of the 2-category $\widetilde{\mathbf{Mod}}_{\mathfrak{C}}(\underline{End}(M))$.

Finally, it follows from the first part of the proof above that the 2-functor $\underline{Hom}(M,-):\widetilde{\mathfrak{M}}\rightarrow \widetilde{\mathbf{Mod}}_{\mathfrak{C}}(\underline{End}(M))$ is an equivalence of 2-categories. Namely, this 2-functor is manifestly essentially surjective on objects, and we have shown that it is essentially surjective on 1-morphisms as well as fully faithful on 2-morphisms. Thanks to lemma A.2.5 of \cite{D1}, we therefore find that $\underline{Hom}(M,-):\mathfrak{M}\rightarrow \mathbf{Mod}_{\mathfrak{C}}(\underline{End}(M))$ is an equivalence of 2-categories as claimed.
\end{proof}

Let us now examine some examples.

\begin{Example}
In the case $\mathfrak{C}=\mathbf{2Vect}$, theorem \ref{thm:2Ostrik} becomes the statement that every finite semisimple 2-category is the 2-category of finite semisimple module categories over a finite semisimple monoidal category, which is in fact rigid by theorem \ref{thm:characterizationrigid} and example \ref{ex:rigidalgebra2Vect}. This is exactly the content of theorem 1.4.9 of \cite{DR}.
\end{Example}

\begin{Example}
Let us fix a finite group $G$, and consider the finite semisimple left $\mathbf{2Vect}_G$-module 2-category $\mathbf{2Vect}$ of example \ref{ex:2VectGmodule}. A $\mathbf{2Vect}_G$-generator is given $\mathbf{Vect}$, and it is easy to compute that $$\underline{End}(\mathbf{Vect}) \simeq \boxplus_{g\in G}\mathbf{Vect}_g.$$ Further, the composition map is given by the multiplication in $G$, and it is easy to see directly that $\mathbf{Mod}_{\mathbf{2Vect}_G}(\underline{End}(\mathbf{Vect}))\simeq \mathbf{2Vect}$ with unique equivalence class of simple object given by $\underline{End}(\mathbf{Vect})$. Alternatively, theorem \ref{thm:2Ostrik} asserts that $$\underline{Hom}(\mathbf{Vect},-):\mathbf{2Vect}\rightarrow \mathbf{Mod}_{\mathbf{2Vect}_G}(\underline{End}(\mathbf{Vect}))$$ is an equivalence of $\mathbf{2Vect}_G$-module 2-categories. Finally, let us observe that the 1-morphism $$\begin{tabular}{ccc}$\underline{End}(\mathbf{Vect})$&$\rightarrow $&$\underline{End}(\mathbf{Vect})\boxtimes\underline{End}(\mathbf{Vect})$\\$\mathbf{Vect}_f$&$\mapsto$&$\boxplus_{g\in G}\mathbf{Vect}_{g}\boxtimes \mathbf{Vect}_{g^{-1}f}$\end{tabular}$$ splits the multiplication map of $\underline{End}(\mathbf{Vect})$ as a bimodule 1-morphism. This means that $\underline{End}(\mathbf{Vect})$ is rigid. In fact, it is possible to show that this algebra is separable.
\end{Example}

\begin{Example}
Let $\mathcal{B}$ be a braided fusion category. Thanks to lemma \ref{lem:Mod(B)modules}, finite semisimple left $\mathbf{Mod}(\mathcal{B})$-module 2-categories correspond exactly to multifusion categories $\mathcal{C}$ equipped with a central monoidal functor $\mathcal{B}\rightarrow \mathcal{C}$. If we pick $\mathcal{C}$ as our $\mathbf{Mod}(\mathcal{B})$-generator for $\mathbf{Mod}(\mathcal{C})$, we find that $$\underline{End}(\mathcal{C}) \simeq \mathcal{C},$$ where the right hand-side is viewed as a right $\mathcal{B}$-module category through the monoidal functor $\mathcal{B}\rightarrow \mathcal{C}$. Theorem \ref{thm:2Ostrik} thus recovers the equivalence $$\mathbf{Mod}(\mathcal{C})\simeq \mathbf{Mod}_{\mathbf{Mod}(\mathcal{B})}(\mathcal{C}),$$ i.e.\ finite semisimple right $\mathcal{C}$-module categories correspond precisely to finite semisimple right $\mathcal{B}$-module categories with a compatible right action by $\mathcal{C}$.
\end{Example}

\begin{Remark}\label{rem:separability}
We have seen in theorem \ref{thm:rigid} that the algebra $\underline{End}(M)$ is rigid. It is expected that these algebras are separable. Proving this implies that every multifusion 2-category is separable, as we will show in \cite{D6}.
\end{Remark}

\begin{Remark}\label{rem:final}
Motivated by remark \ref{rem:rigidi=separable}, let us assume that every rigid algebra in a multifusion 2-category is separable. Under this hypothesis, we can push the reasoning of remark \ref{rem:Morita3Cat} further. Namely, bimodules over separable algebras admit a canonical tensor product operation (see \cite{GJF}), which ought to correspond to the composition of left $\mathfrak{C}$ module 2-functors. In fact, we expect that there is an equivalence of 3-categories $$\adjustbox{scale=0.8}{$\begin{Bmatrix}
Separable\ Algebras\ in\ \mathfrak{C}\\
Bimodules\ in\ \mathfrak{C}\\
Bimodule\ 1\text{-}morphisms\ in\ \mathfrak{C}\\
Bimodule\ 2\text{-}morphisms\ in\ \mathfrak{C}
\end{Bmatrix}\rightarrow
\begin{Bmatrix}
Finite\ Semisimple\ left\ \mathfrak{C}\text{-}module\ 2\text{-}Categories\\
Left\ \mathfrak{C}\text{-}module\ 2\text{-}Functors\\
Left\ \mathfrak{C}\text{-}module\ 2\text{-}Natural\ Transformations\\
Left\ \mathfrak{C}\text{-}module\ Modifications
\end{Bmatrix}$},
$$
given by $\mathbf{Mod}_{\mathfrak{C}}(-)$. In the special case $\mathfrak{C}=\mathbf{2Vect}$, we have argued in example \ref{ex:rigidalgebra2Vect} that separable algebras are precisely multifusion categories, so that this equivalence of 3-categories is exactly the main result of \cite{D1}.
\end{Remark}

\begin{Remark}\label{rem:2OstrikGeneral}
All the results of this section hold more generally for any perfect field $\mathds{k}$, locally separable compact semisimple tensor 2-category $\mathfrak{C}$ over $\mathds{k}$, and locally separable compact semisimple left $\mathfrak{C}$-module 2-category $\mathfrak{M}$. In more detail, we have already explained in remarks \ref{rem:EnrichmentGeneral}, \ref{rem:EnrichmentModuleGeneral}, and \ref{rem:rigidgeneral} that the results of the previous sections except theorem \ref{thm:rigid} hold at this level of generality. Instead of appealing to theorem \ref{thm:rigid}, we can use lemma \ref{lem:2Ostriktechnical} below in the proof of theorem \ref{thm:2Ostrik}. In addition, we also need to make use of theorem 1.3.6 of \cite{D5} instead of theorem 1.4.9 of \cite{DR}.
\end{Remark}

\begin{Lemma}\label{lem:2Ostriktechnical}
Let $\mathfrak{C}$ be a locally separable compact semisimple tensor 2-category over a perfect field, and $\mathfrak{M}$ a locally separable left $\mathfrak{C}$-module 2-category. Then, $\mathbf{Mod}_{\mathfrak{C}}(\underline{End}(M))$ is a compact semisimple 2-category for every object $M$ in $\mathfrak{M}$.
\end{Lemma}
\begin{proof}
We have seen in remark \ref{rem:rigidgeneral} above that there exists a monoidal functor $\mathcal{C}\rightarrow \mathcal{A}$ between finite semisimple tensor 1-categories such that $$\mathbf{Mod}_{\mathfrak{C}}(\underline{End}(M))\simeq\mathbf{Mod}(\mathcal{A}),$$ that is we can identify $\mathbf{Mod}_{\mathfrak{C}}(\underline{End}(M))$ with the 2-category of finite semisimple right $\mathcal{A}$-module 1-categories. In particular, it follows from proposition 2.5.10 of \cite{DSPS13} that, in order to check that $\mathbf{Mod}_{\mathfrak{C}}(\underline{End}(M))$ is compact semisimple, it is enough to check that for every indecomposable finite semisimple $\mathcal{A}$-module 1-category $\mathcal{N}$, there exists a separable finite semisimple $\mathcal{A}$-module 1-category $\mathcal{M}$ and a non-zero right $\mathcal{A}$-module functor $\mathcal{M}\rightarrow\mathcal{N}$ such that $End_{\mathcal{A}}(\mathcal{M})$ is a separable finite semisimple tensor 1-category.

Now, under the above equivalence, we can view the $\mathcal{A}$-module 1-category $\mathcal{N}$ as a right $\underline{End}(M)$-module $N$ in $\mathfrak{C}$. In particular, there is a non-zero right $\underline{End}(M)$-module 1-morphism $N\Box \underline{End}(M)\rightarrow N$. But, we have equivalences of finite semisimple tensor 1-categories $$End_{\underline{End}(M)}(N\Box \underline{End}(M))\simeq Hom_{\mathfrak{C}}(N, N\Box \underline{End}(M))\simeq End_{\mathfrak{M}}(N\Box M).$$ Finally, the right hand-side is a separable tensor 1-category because we have assumed that $\mathfrak{M}$ is locally separable. This finishes the proof of the lemma.
\end{proof}

\begin{Remark}
We emphasize that the hypothesis that $\mathfrak{M}$ be locally separable cannot be removed in the generalized version of theorem \ref{thm:2Ostrik} given in remark \ref{rem:2OstrikGeneral}. Namely, with $\mathds{k}$ is an algebraically closed field of positive characteristic $p$, $\mathfrak{C}=\mathbf{2Vect}$, $\mathfrak{M}$ the finite semisimple 2-category of separable right $\mathbf{Vect}_{\mathbb{Z}/p}$-module 1-categories, and $M = \mathbf{Vect}_{\mathbb{Z}/p}$, we find that $\underline{End}(M) \simeq \mathbf{Vect}_{\mathbb{Z}/p}$. But, $\mathbf{Mod}_{\mathbf{2Vect}}(\mathbf{Vect}_{\mathbb{Z}/p})$ is the 2-category of all finite semisimple right $\mathbf{Vect}_{\mathbb{Z}/p}$-module 1-categories. In particular, it contains $\mathbf{Vect}$, which is not a separable right $\mathbf{Vect}_{\mathbb{Z}/p}$-module 1-category.
\end{Remark}

\appendix

\section{Appendix}\label{sec:appendix}

\subsection{Diagrams for the Proof of Proposition \ref{prop:enrichedhommod}}\label{sec:diagcoherenceomega}

\begin{figure}[!htpb]
\centering
    \includegraphics[width=67.5mm]{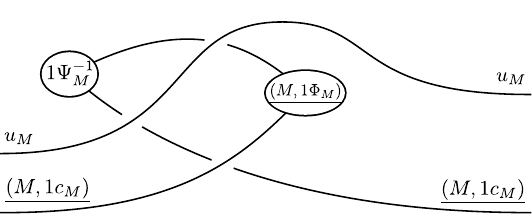}
\caption{Axiom b}
\label{fig:omegaidentity}
\end{figure}

\begin{figure}[!htpb]
\centering
    \includegraphics[width=67.5mm]{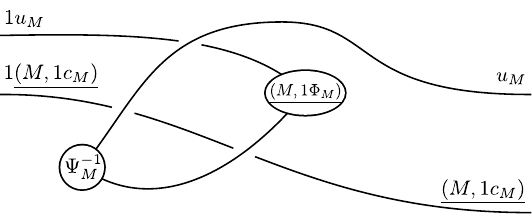}
\caption{Axiom c}
\label{fig:omegaidentityc}
\end{figure}

\begin{figure}[!htpb]
\centering
    \includegraphics[width=105mm]{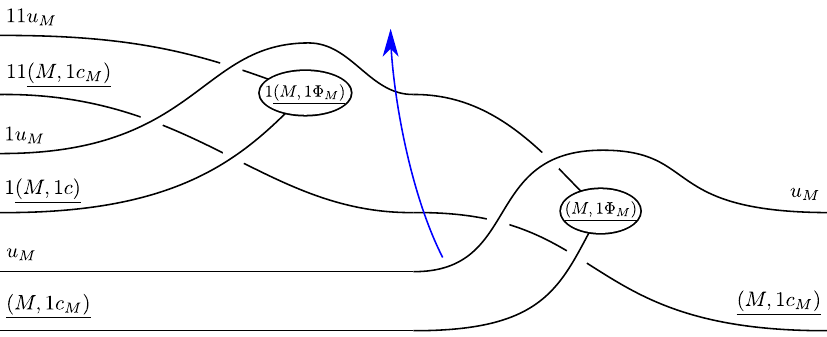}
\caption{Axiom a (Part 1)}
\label{fig:enrichedhomassociativity1}
\end{figure}

\begin{figure}[!htpb]
\centering
    \includegraphics[width=105mm]{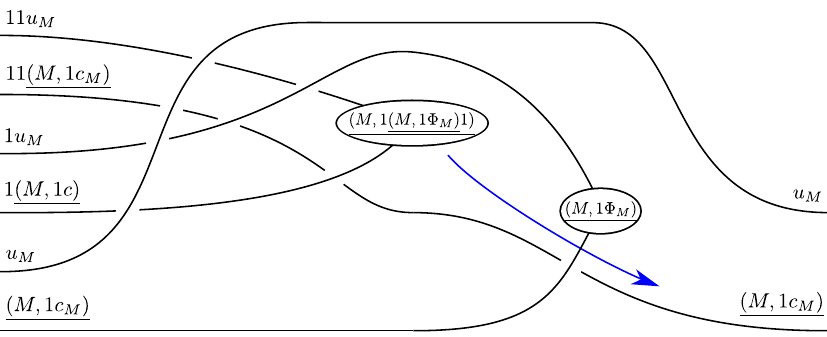}
\caption{Axiom a (Part 2)}
\label{fig:enrichedhomassociativity2}
\end{figure}

\begin{figure}[!htpb]
\centering
    \includegraphics[width=105mm]{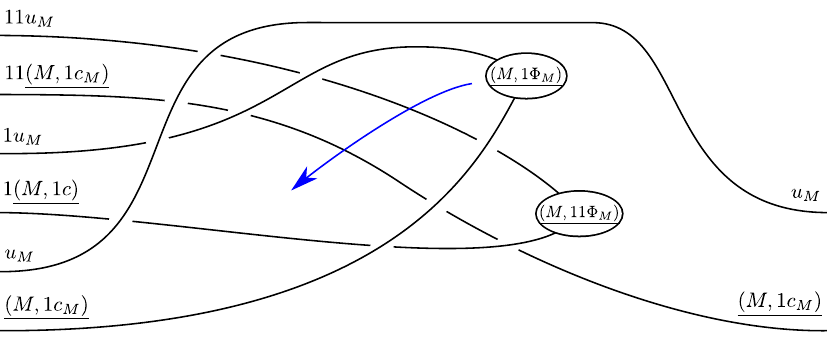}
\caption{Axiom a (Part 3)}
\label{fig:enrichedhomassociativity3}
\end{figure}

\begin{figure}[!htpb]
\centering
    \includegraphics[width=105mm]{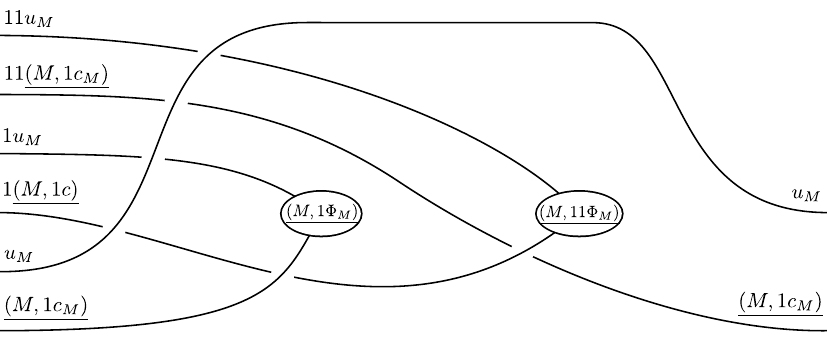}
\caption{Axiom a (Part 4)}
\label{fig:enrichedhomassociativity4}
\end{figure}

\begin{figure}[!htpb]
\centering
    \includegraphics[width=105mm]{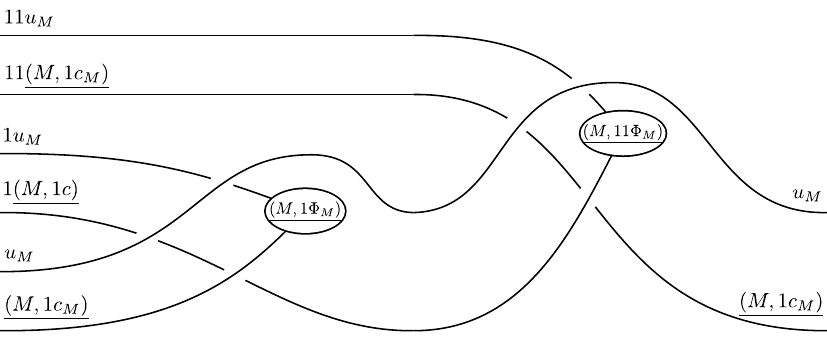}
\caption{Axiom a (Part 5)}
\label{fig:enrichedhomassociativity5}
\end{figure}

\FloatBarrier

\subsection{Diagrams for the Proof of Theorem \ref{thm:enrichmentmodule}}\label{sec:diagcoherenceenriched}

\begin{figure}[!htpb]
\centering
\includegraphics[width=110mm]{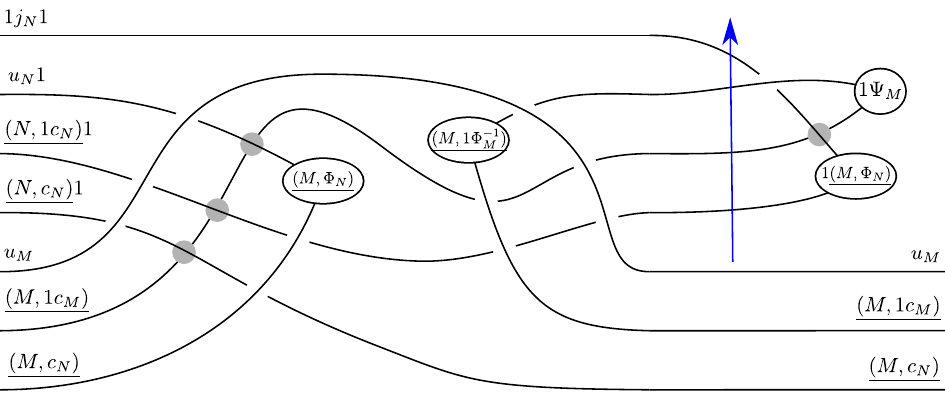}
\caption{Axiom b (Part 1)}
\label{fig:enrichedcatunitality1}
\end{figure}

\begin{figure}[!htpb]
\centering
\includegraphics[width=110mm]{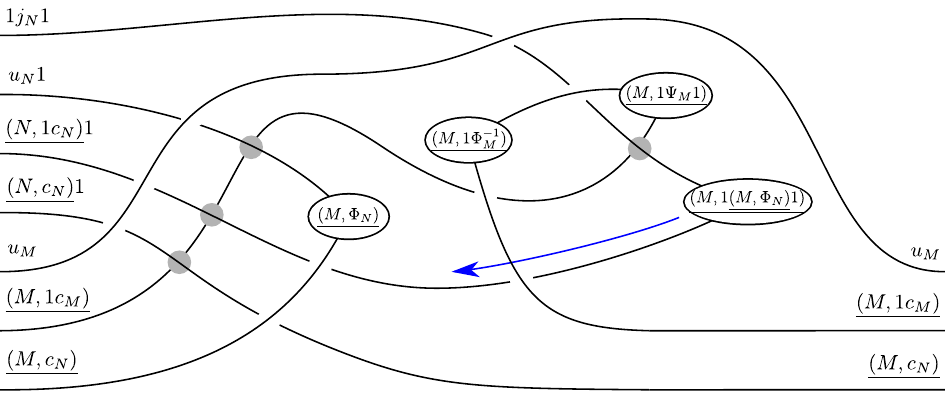}
\caption{Axiom b (Part 2)}
\label{fig:enrichedcatunitality2}
\end{figure}

\begin{figure}[!htpb]
\centering
\includegraphics[width=110mm]{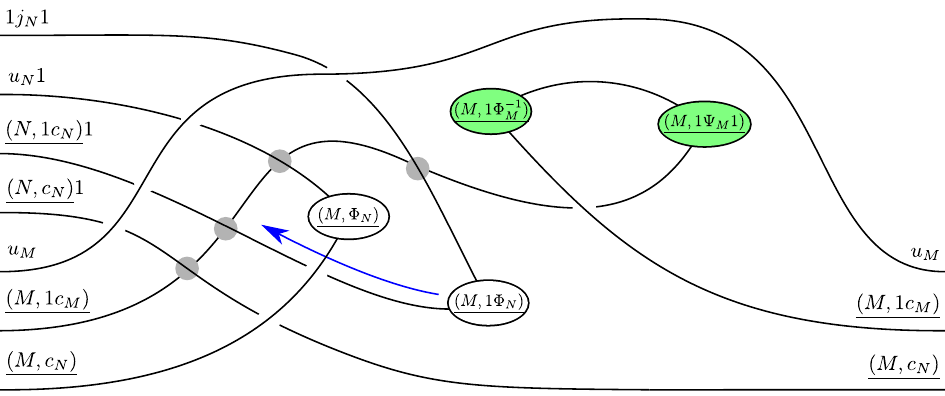}
\caption{Axiom b (Part 3)}
\label{fig:enrichedcatunitality3}
\end{figure}

\begin{figure}[!htpb]
\centering
\includegraphics[width=120mm]{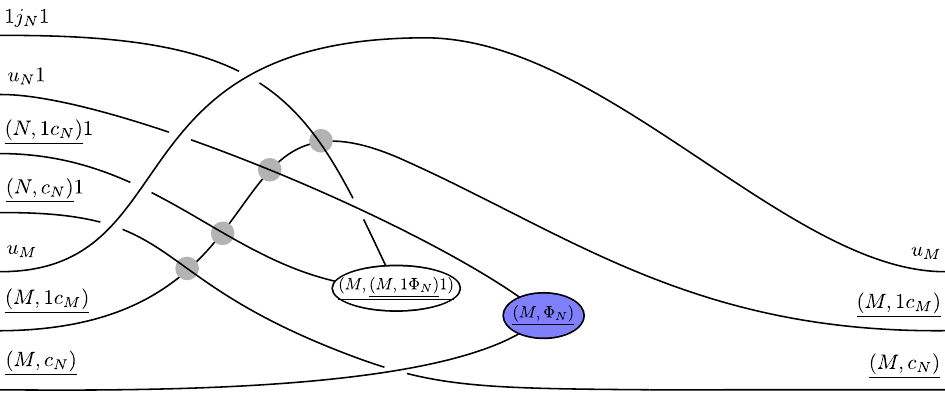}
\caption{Axiom b (Part 4)}
\label{fig:enrichedcatunitality4}
\end{figure}

\begin{figure}[!htpb]
\centering
\includegraphics[width=120mm]{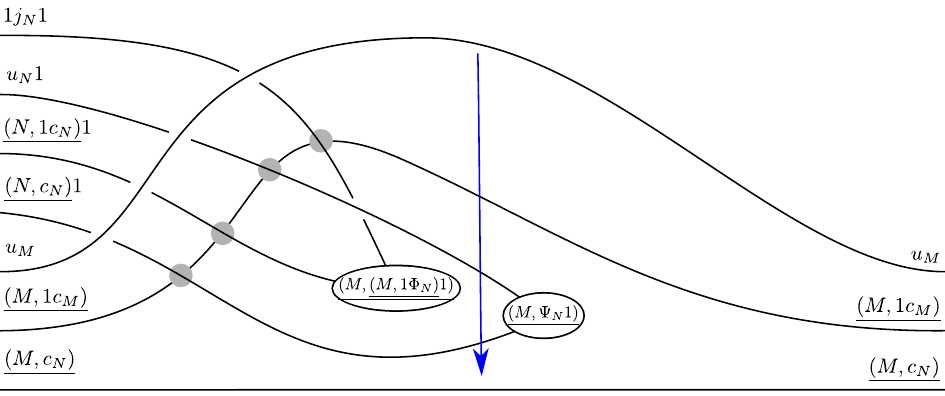}
\caption{Axiom b (Part 5)}
\label{fig:enrichedcatunitality5}
\end{figure}

\begin{figure}[!htpb]
\centering
\includegraphics[width=80mm]{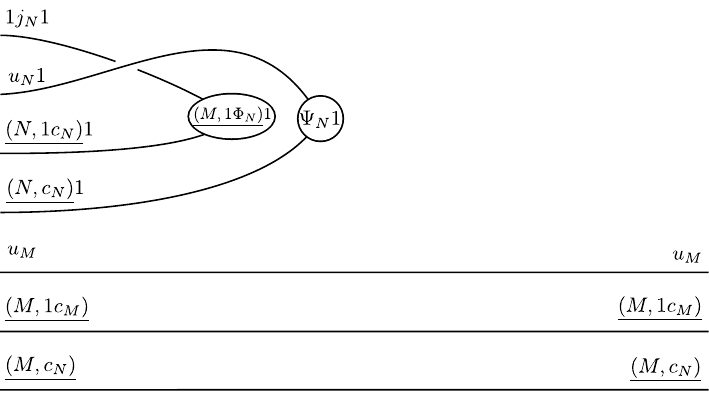}
\caption{Axiom b (Part 6)}
\label{fig:enrichedcatunitality6}
\end{figure}

\FloatBarrier

\begin{landscape}
    \vspace*{\fill}
    \begin{figure}[!htpb]
    \centering
    \includegraphics[width=180mm]{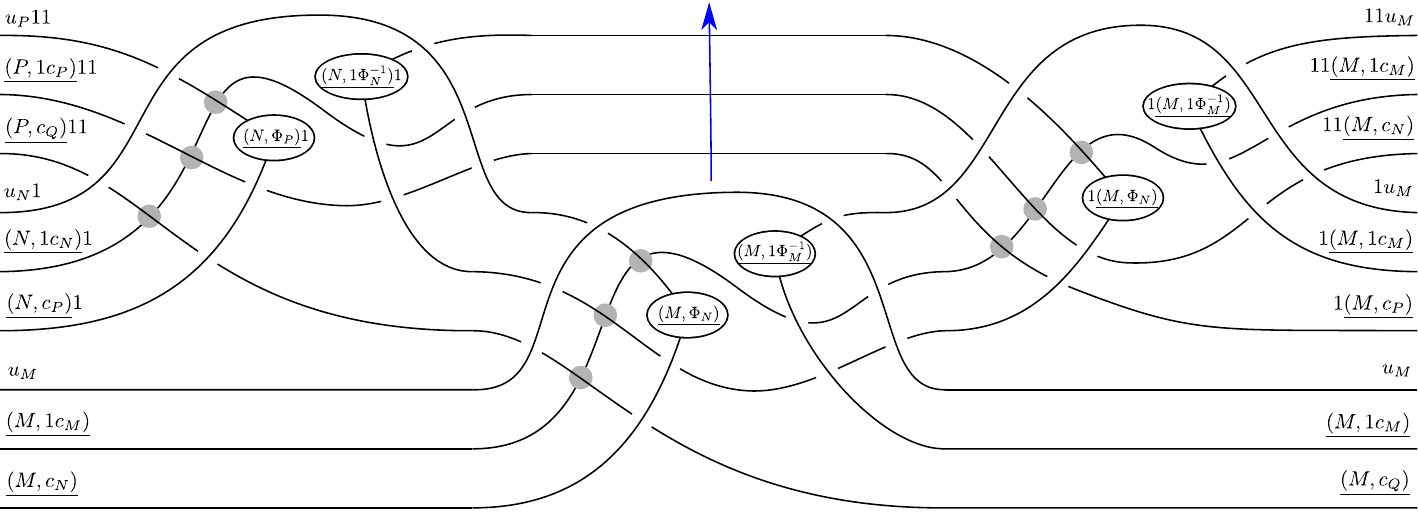}
    \caption{Axiom a (Part 1)}
    \label{fig:enrichedcatassociativity1}
    \end{figure}
    \vfill
\end{landscape}

\begin{landscape}
    \vspace*{\fill}
    \begin{figure}[!htpb]
    \centering
    \includegraphics[width=180mm]{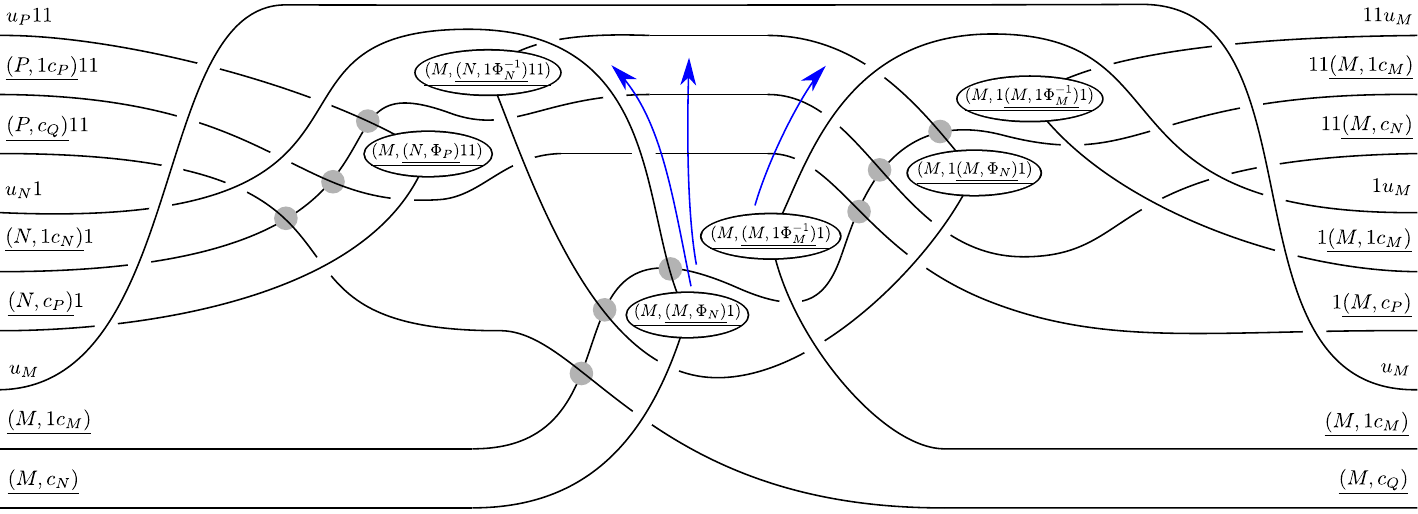}
    \caption{Axiom a (Part 2)}
    \label{fig:enrichedcatassociativity2}
    \end{figure}
    \vfill
\end{landscape}

\begin{landscape}
    \vspace*{\fill}
    \begin{figure}[!htpb]
    \centering
    \includegraphics[width=180mm]{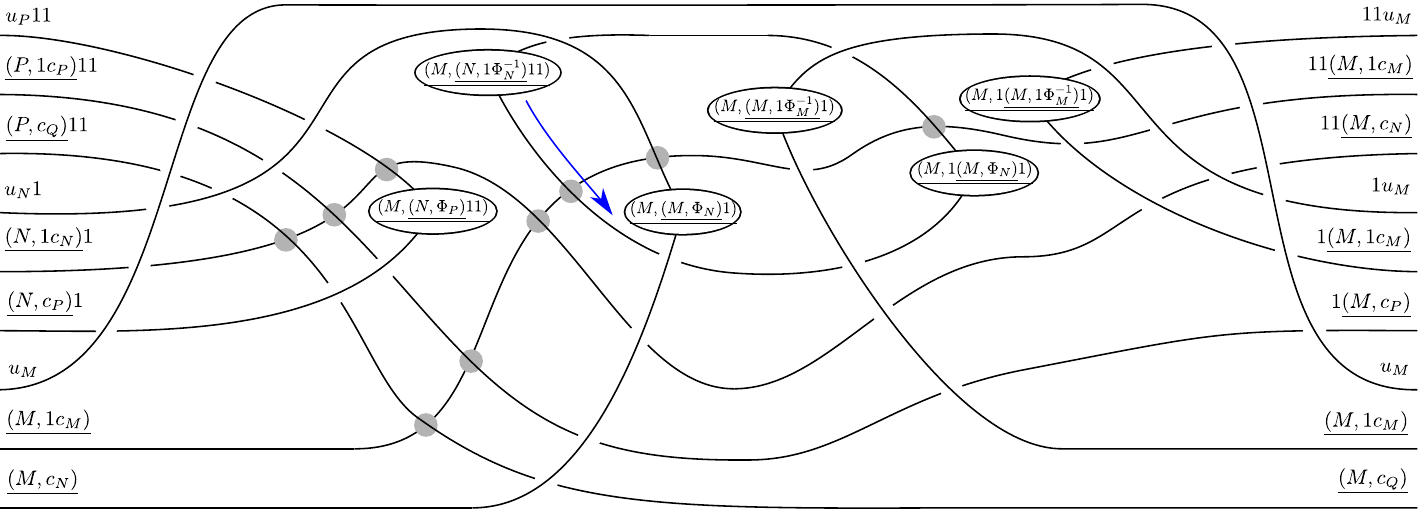}
    \caption{Axiom a (Part 3)}
    \label{fig:enrichedcatassociativity3}
    \end{figure}
    \vfill
\end{landscape}

\begin{landscape}
    \vspace*{\fill}
    \begin{figure}[!htpb]
    \centering
    \includegraphics[width=180mm]{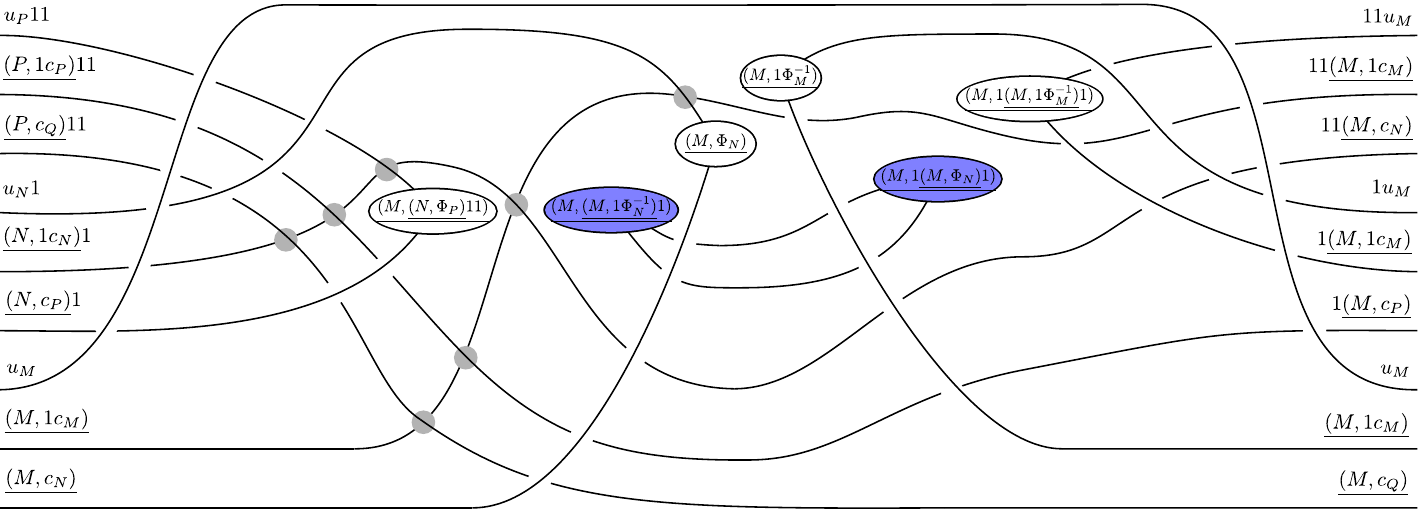}
    \caption{Axiom a (Part 4)}
    \label{fig:enrichedcatassociativity4}
    \end{figure}
    \vfill
\end{landscape}

\begin{landscape}
    \vspace*{\fill}
    \begin{figure}[!htpb]
    \centering
    \includegraphics[width=180mm]{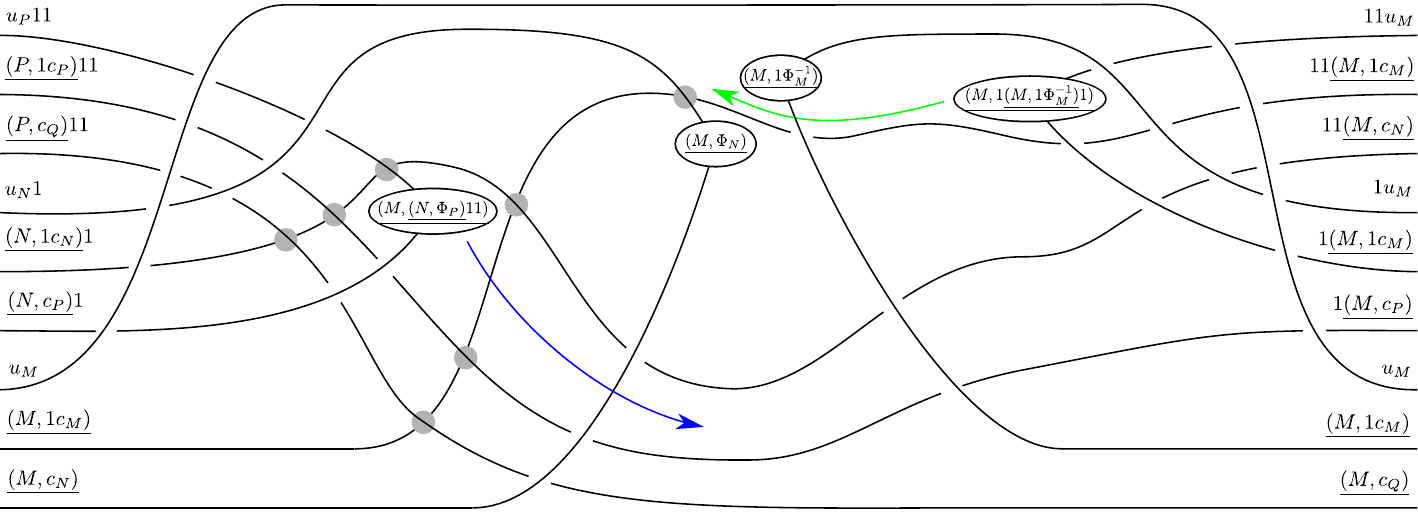}
    \caption{Axiom a (Part 5)}
    \label{fig:enrichedcatassociativity5}
    \end{figure}
    \vfill
\end{landscape}

\begin{landscape}
    \vspace*{\fill}
    \begin{figure}[!htpb]
    \centering
    \includegraphics[width=180mm]{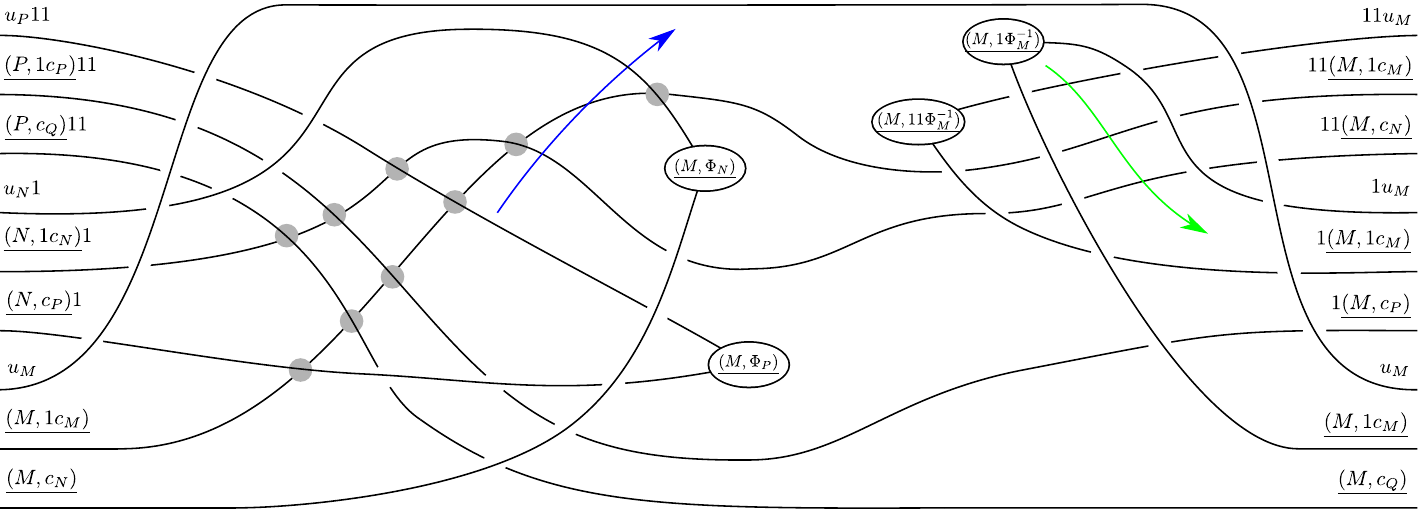}
    \caption{Axiom a (Part 6)}
    \label{fig:enrichedcatassociativity6}
    \end{figure}
    \vfill
\end{landscape}

\begin{landscape}
    \vspace*{\fill}
    \begin{figure}[!htpb]
    \centering
    \includegraphics[width=180mm]{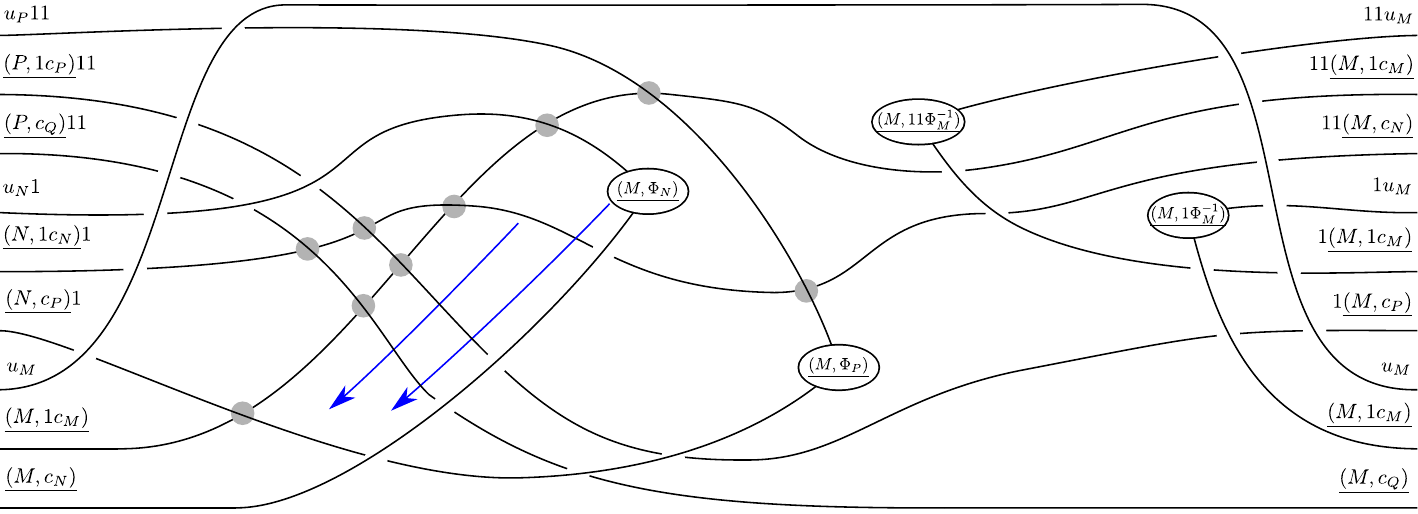}
    \caption{Axiom a (Part 7)}
    \label{fig:enrichedcatassociativity7}
    \end{figure}
    \vfill
\end{landscape}

\begin{landscape}
    \vspace*{\fill}
    \begin{figure}[!htpb]
    \centering
    \includegraphics[width=180mm]{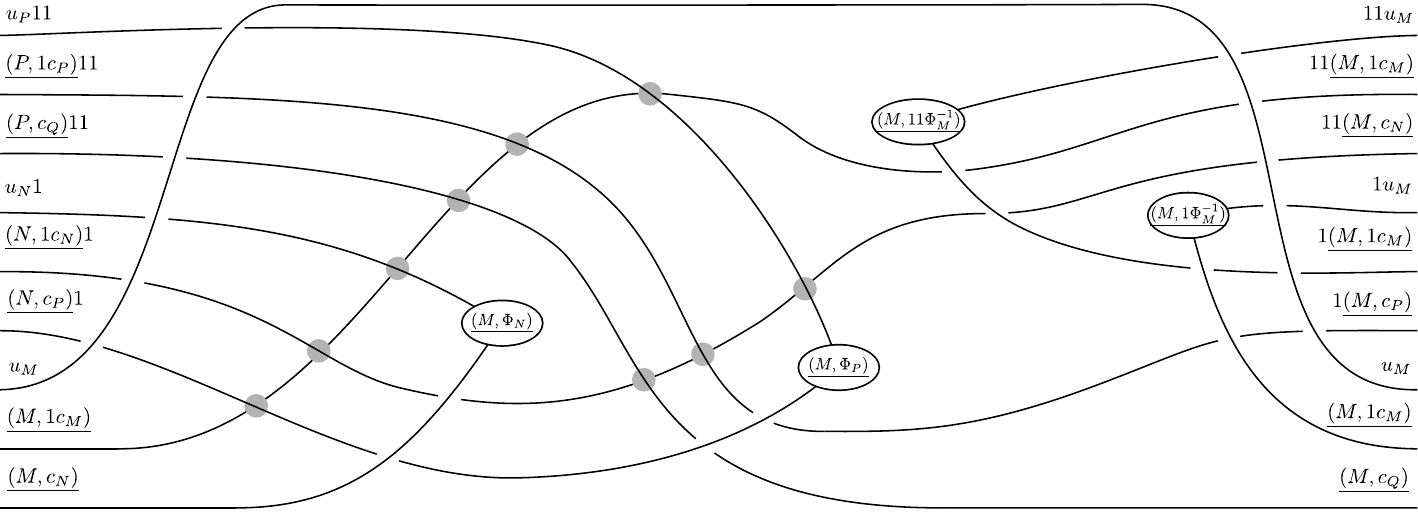}
    \caption{Axiom a (Part 8)}
    \label{fig:enrichedcatassociativity8}
    \end{figure}
    \vfill
\end{landscape}

\begin{landscape}
    \vspace*{\fill}
    \begin{figure}[!htpb]
    \centering
    \includegraphics[width=180mm]{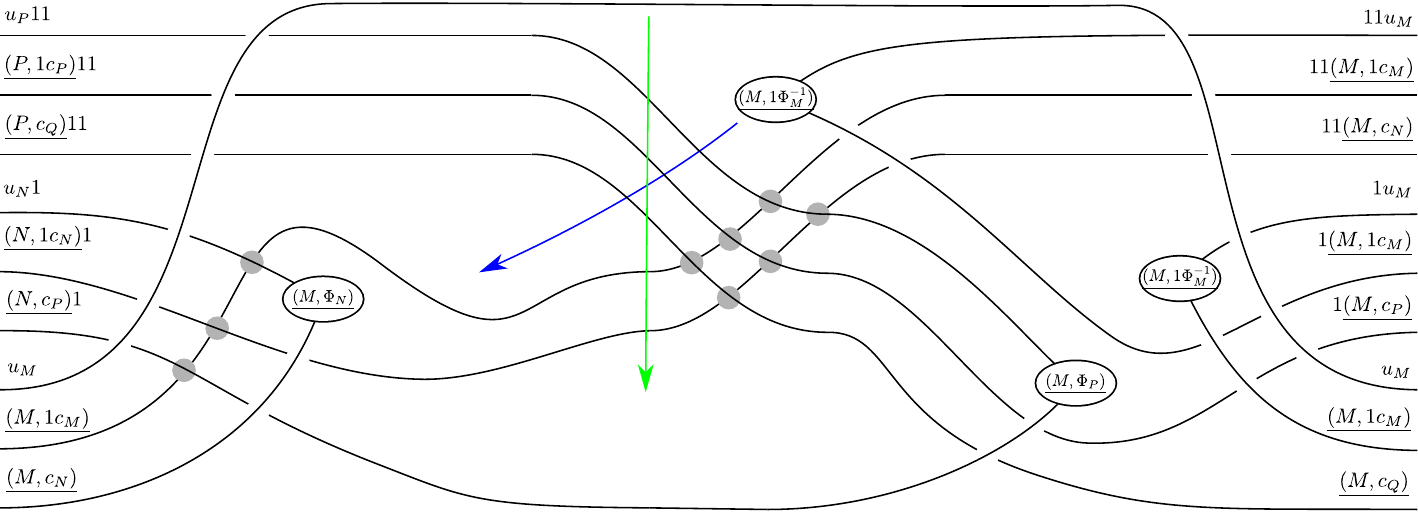}
    \caption{Axiom a (Part 9)}
    \label{fig:enrichedcatassociativity9}
    \end{figure}
    \vfill
\end{landscape}

\begin{landscape}
    \vspace*{\fill}
    \begin{figure}[!htpb]
    \centering
    \includegraphics[width=180mm]{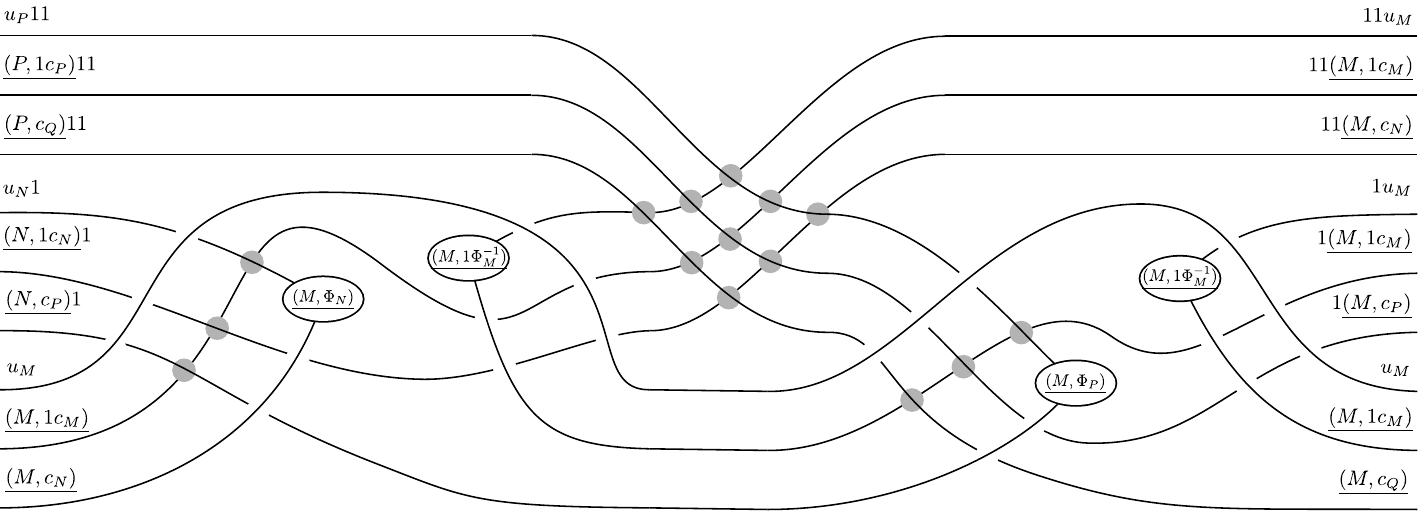}
    \caption{Axiom a (Part 10)}
    \label{fig:enrichedcatassociativity10}
    \end{figure}
    \vfill
\end{landscape}

\FloatBarrier

\bibliography{bibliography.bib}

\begin{thebibliography}{BDSPV14}

\bibitem[BD95]{BD}
John~C. Baez and James Dolan.
\newblock Higher-dimensional algebra and topological quantum field theory.
\newblock {\em J. Math. Phys.}, 36:6073--6105, 1995.

\bibitem[BDSPV14]{BDSPV}
Bruce Bartlett, Christopher~L. Douglas, Christopher~J. Schommer-Pries, and
  Jamie Vicary.
\newblock Extended 3-dimensional bordism as the theory of modular objects,
  2014.
\newblock arXiv: 1411.0945.

\bibitem[BJS21]{BJS}
Adrien Brochier, David Jordan, and Noah Snyder.
\newblock On dualizability of braided tensor categories.
\newblock {\em Compositio Mathematica}, 3:435--483, 2021.
\newblock arXiv:1804.07538.

\bibitem[BJSS21]{BJSS}
Adrien Brochier, David Jordan, Pavel Safronov, and Noah Snyder.
\newblock Invertible braided tensor categories.
\newblock {\em Algebr. Geom. Topol.}, 21(4):2107--2140, 2021.
\newblock arXiv:2003.13812.

\bibitem[CMV02]{CMV}
I.J.~Le Creurer, F.~Marmolejo, and E.M. Vitale.
\newblock Beck's theorem for pseudo-monads.
\newblock {\em Journal of Pure and Applied Algebra}, 173:293--313, 2002.

\bibitem[D{\'e}c]{D6}
Thibault~D. D{\'e}coppet.
\newblock Separable fusion 2-categories.
\newblock In preparation.

\bibitem[D{\'e}c21a]{D3}
Thibault~D. D{\'e}coppet.
\newblock The 2-{D}eligne tensor product, 2021.
\newblock Accepted for publication in the \textit{Kyoto Journal of
  Mathematics}, arXiv:2103.16880.

\bibitem[D{\'e}c21b]{D5}
Thibault~D. D{\'e}coppet.
\newblock Compact semisimple 2-categories, 2021.
\newblock arXiv:2111.09080.

\bibitem[D{\'e}c22a]{D1}
Thibault~D. D{\'e}coppet.
\newblock Multifusion categories and finite semisimple 2-categories.
\newblock {\em Journal of Pure and Applied Algebra}, 226(8), 2022.
\newblock arXiv:2012.15774.

\bibitem[D{\'e}c22b]{D2}
Thibault~D. D{\'e}coppet.
\newblock Weak fusion 2-categories.
\newblock {\em Cahiers de Topologie et G{\'e}om{\'e}trie Diff{\'e}rentielle
  Cat{\'e}goriques}, LXIII(1):3--24, 2022.
\newblock arXiv:2103.15150.

\bibitem[DN22]{DN}
Alexei Davidov and Dimitri Nikshych.
\newblock Braided {P}icard groups and graded extensions of braided tensor
  categories.
\newblock {\em Selecta Mathematica}, 27, 2022.
\newblock arXiv:2006.08022.

\bibitem[DR18]{DR}
Christopher~L. Douglas and David~J. Reutter.
\newblock Fusion 2-categories and a state-sum invariant for 4-manifolds, 2018.
\newblock arXiv: 1812.11933.

\bibitem[DS97]{DS}
Brian Day and Ross Street.
\newblock Monoidal bicategories and {H}opf algebroids.
\newblock {\em Advances in Mathematics}, 129(1):99--157, 1997.

\bibitem[DSPS19]{DSPS14}
Christopher~L. Douglas, Christopher Schommer-Pries, and Noah Snyder.
\newblock The balanced tensor product of module categories.
\newblock {\em Kyoto J. Math.}, 59:167--179, 2019.
\newblock arXiv: 1406.4204.

\bibitem[DSPS21]{DSPS13}
Christopher~L. Douglas, Christopher Schommer-Pries, and Noah Snyder.
\newblock {\em Dualizable tensor categories}.
\newblock Mem. Amer. Math. Soc. AMS, 2021.
\newblock arXiv: 1312.7188.

\bibitem[EGNO15]{EGNO}
Pavel Etingof, Shlomo Gelaki, Dmitri Nikshych, and Victor Ostrik.
\newblock {\em Tensor Categories}.
\newblock Mathematical Surveys and Monographs. AMS, 2015.

\bibitem[ENO05]{ENO}
Pavel Etingof, Dmitri Nikshych, and Viktor Ostrik.
\newblock On fusion categories.
\newblock {\em Ann. Math.}, 162:581--642, 2005.
\newblock arXiv: math/0203060.

\bibitem[Gal17]{Gal}
C{\'e}sar Galindo.
\newblock Coherence for monoidal {G}-categories and braided {G}-crossed
  categories.
\newblock {\em Journal of Algebra}, 487:118--137, 2017.
\newblock arXiv:1604.01679.

\bibitem[GJF19]{GJF}
Davide Gaiotto and Theo Johnson-Freyd.
\newblock Condensations in higher categories, 2019.
\newblock arXiv: 1905.09566v2.

\bibitem[GS16]{GS}
Richard Garner and Michael Schulman.
\newblock Enriched categories as a free cocompletion.
\newblock {\em Advances in Mathematics}, 289:1--94, 2016.
\newblock arXiv:1301.3191v2.

\bibitem[Gur12]{Gur2}
Nick Gurski.
\newblock Biequivalence in tricategories.
\newblock {\em Theory and Applications of Categories}, 26, 2012.

\bibitem[Gur13]{Gur}
Nick Gurski.
\newblock {\em Coherence in Three-Dimensional Category Theory}.
\newblock Cambridge Tracts in Mathematics. Cambridge University Press, 2013.

\bibitem[JF20]{JF}
Theo Johnson-Freyd.
\newblock On the classification of topological orders.
\newblock {\em Communications in Mathematical Physics}, 2020.
\newblock arXiv:2003.06663.

\bibitem[JFR23]{JFR}
Theo Johnson-Freyd and David~J. Reutter.
\newblock Minimal non-degenerate extensions.
\newblock {\em Jour. Amer. Math. Soc.}, 2023.
\newblock arXiv:2105.15167.

\bibitem[JFY21]{JFY}
Theo Johnson-Freyd and Matthew Yu.
\newblock Fusion 2-categories with no line operators are grouplike.
\newblock {\em Bulletin of the Australian Mathematical Society},
  104(3):434–442, 2021.
\newblock arXiv:2010.07950.

\bibitem[JMPP21]{JMPP}
Corey Jones, Scott Morrison, David Penneys, and Julia Plavnik.
\newblock Extension theory for braided-enriched fusion categoriess.
\newblock {\em Int. Math. Res. Not.}, 2021.
\newblock arXiv:1910.03178.

\bibitem[KTZ20]{KTZ}
Liang Kong, Yin Tian, and Shan Zhou.
\newblock The center of monoidal 2-categories in 3+1d {D}ijkgraaf-{W}itten
  theory.
\newblock {\em Advances in Mathematics}, 360(106928), 2020.

\bibitem[KZ21]{KZ2}
Liang Kong and Hao Zheng.
\newblock Categories of quantum liquids {II}, 2021.
\newblock arXiv:2107.03858.

\bibitem[KZ22]{KZ1}
Liang Kong and Hao Zheng.
\newblock Categories of quantum liquids {I}.
\newblock {\em J. High Energ. Phys. 2022}, 70, 2022.
\newblock arXiv:2011.02859.

\bibitem[Lur10]{L}
Jacob Lurie.
\newblock On the classification of topological field theories.
\newblock {\em Curr. Dev. Math. Sci.}, 1:129--280, 2010.

\bibitem[Ost03]{O}
Victor Ostrik.
\newblock Module categories, weak {H}opf algebras and modular invariants.
\newblock {\em Transformation Groups}, 8:177–206, 2003.
\newblock arXiv:math/0111139.

\bibitem[Pst14]{Pstr}
Piotr Pstragowski.
\newblock On dualizable objects in monoidal bicategories, framed surfaces and
  the cobordism hypothesis, 2014.
\newblock arXiv: 1411.6691.

\bibitem[SP11]{SP}
Christopher~J. Schommer-Pries.
\newblock {\em The Classification of Two-Dimensional Extended Topological Field
  Theories}.
\newblock PhD thesis, UC Berkeley, 2011.
\newblock arXiv: 1112.1000.

\bibitem[Tur00]{T}
Vladimir Turaev.
\newblock Homotopy field theory in dimension 3 and crossed group-categories,
  2000.
\newblock arXiv:math/0005291.

\end{thebibliography}

\end{document}